\documentclass{myconm-p-l}

\usepackage{hyperref}
\usepackage[utf8]{inputenc}
\usepackage{amssymb}
\usepackage{enumerate}
\usepackage[normalem]{ulem}

\newcommand{\defstyle}[1]{\textbf{#1}}
\newcommand{\myprob}[1]{\mathbb P \left[ #1 \right]}

\newcommand{\probCond}[2]{\mathbb P \left[ #1 \left| #2 \right. \right]}

\newcommand{\omid}[1]{\mathbb E \left[ #1 \right]}
\newcommand{\omidPalm}[2]{\mathbb E_{#1} \left[ #2 \right]}
\newcommand{\omidCond}[2]{\mathbb E \left[ #1 \left| #2 \right. \right]}
\newcommand{\omidCondL}[2]{\mathbb E \left[\left. #1 \right| #2 \right]}

\newcommand{\norm}[1]{\left| #1 \right|}

\newcommand{\identity}[1]{1_{#1}}
\newcommand{\bs}[1]{\boldsymbol{#1}}
\newcommand{\card}[1]{\# #1}

\renewcommand{\L}[2][]{{\cal L}^{#2}_{#1}}

\newcommand{\transport}[2]{#1^{\rightarrow}#2}

\numberwithin{equation}{section}

\theoremstyle{theorem}
\newtheorem{theorem}{Theorem}
\newtheorem{lemma}{Lemma}
\newtheorem{proposition}{Proposition}
\newtheorem{corollary}{Corollary}

\newtheorem{question}{Question}
\theoremstyle{definition}
\newtheorem{definition}{Definition}

\newtheorem{algorithm}{Algorithm}
\newtheorem{example}{Example}
\theoremstyle{remark}
\newtheorem{remark}{Remark}

\theoremstyle{theorem}
\newtheorem{innercustomthm}{Theorem}
\newenvironment{customthm}[1]
{\renewcommand\theinnercustomthm{#1}\innercustomthm}
{\endinnercustomthm}

\numberwithin{equation}{section}

\makeatletter
\let\orgdescriptionlabel\descriptionlabel
\renewcommand*{\descriptionlabel}[1]{%
	\let\orglabel\label
	\let\label\@gobble
	\phantomsection
	\edef\@currentlabel{#1}%
	\let\label\orglabel
	\orgdescriptionlabel{#1}%
}
\makeatother

\begin{document}

\title{SHIFT-COUPLING OF RANDOM ROOTED GRAPHS AND NETWORKS}


%    author one information
\author{Ali Khezeli}
\address{School of Mathematics\\ Institute for Research in Fundamental Sciences (IPM)\\ P.O. Box 19395-5746, Tehran, Iran}
\email{alikhezeli@ipm.ir}
\thanks{}

%    author two information
%\author{}
%\address{}
%\curraddr{}
%\email{}
%\thanks{}


\date{\today}

\begin{abstract}
	In this paper, we present a result similar to the shift-coupling result of Thorisson (1996) in the context of random graphs and networks. The result is that a given random rooted network can be obtained by {changing the root of} another given one if and only if the distributions of the two agree on the invariant sigma-field. Several applications of the result are presented for the case of unimodular networks. In particular, it is shown that the distribution of a unimodular network is uniquely determined by its restriction to the invariant sigma-filed. Also, the theorem is applied to the existence of an invariant transport kernel that balances between two given (discrete) measures on the vertices. An application is the existence of a so called \textit{extra head scheme} for the Bernoulli process on an infinite unimodular graph. Moreover, a construction is presented for balancing transport kernels that is a generalization of the Gale-Shapley stable matching algorithm in bipartite graphs. Another application is on a general method that covers the situations where some vertices and edges are added to a unimodular network and then, to make it unimodular, the probability measure is biased and then a new root is selected. It is proved that this method provides all possible \textit{unimodularizations} in these situations. Finally, analogous existing results for stationary point processes and unimodular networks are discussed in detail.
\end{abstract}

\maketitle

%\tableofcontents

\section{Introduction}
\label{sec:intro}
%\ali{
%Structure:
%\begin{itemize}
%	\item About unimodular networks.
%	\item An example of making an extension unimodular.
%	\item Intuitions of having the same non-rooted network.
%	\item The main result: Shift-coupling
%	\item The unimodular case
%	\item Extension
%	\item Stable transport
%	\item Sections of the paper
%\end{itemize}
%}

This paper deals with random rooted graphs, which are possibly infinite, but finite-degree connected graphs with a distinguished vertex called the \textit{root}. Roughly speaking, each vertex and edge of a graph can be equipped with marks to form a network. 
Unimodular random rooted networks have been of great interest in the last two decades. 
They %are random rooted networks that 
satisfy a formulation of the heuristic property that all vertices are \textit{equally likely} to be the root, although there may be infinitely many vertices. The formulation, called \textit{the mass transport principle}, will be recalled in Section~\ref{sec:basic}.
%They are random rooted networks that satisfy a property called \textit{the mass transport principle}, which can be interpreted as \textit{statistical homogeneity} in the random network. The definition of unimodularity is postponed to Section~\ref{sec:basic}. 
This concept is introduced in~\cite{BLPS} and developed further in~\cite{processes} to generalize some properties of Cayley graphs, which are highly homogeneous, to more general classes of graphs and random graphs. It also arises in the study of limits of sequences of finite graphs, which is the novel work of~\cite{objective}, and also in stationary point processes. Many concepts and results in stationary point processes have analogues in the context of unimodular networks. This analogy will be addressed many times in this paper.

%We start with an example to introduce the idea of this work. In some examples in the literature, some unimodular graphs are constructed using other ones. 
%For example, let $[\bs G, \bs o]$ be a unimodular planar graph, where $\bs o$ stands for the root of the random rooted network. Consider an i.i.d. sequence of random rooted graphs $(L(v))_{v\in V(\bs G)}$ indexed by the vertices of $\bs G$. Consider the disjoint union of $\bs G$ with all of the $L(v)$'s and identify each vertex $v$ of $\bs G$ with the root of $L(v)$. Let $[\hat{\bs G}, {\bs o}]$ be the resulting random rooted graph. This is not necessarily unimodular in general, but with a special finiteness condition, one can obtain a unimodular graph out of it as follows. Assume the number of vertices $\card{V(L(){\bs o}))}$ of $L({\bs o})$ has finite expectation. If $\mu$ is the distribution of $[\hat{\bs G}, \bs o]$, first choose a random rooted graph with distribution $\mu$ biased by $\card{V(L({\bs o}))}$ and then choose a new root uniformly at random in $L(\bs o)$. The new random rooted graph is unimodular. 

%We start with a simple example to introduce the idea of this work.
To introduce the idea of this work, we get help from the following general construction method.
Let $[\bs G, \bs o]$ be a given (non-unimodular) random rooted network, where $\bs o$ stands for the root. In some examples in the literature, a unimodular network is constructed from $[\bs G, \bs o]$ by the following two steps: \textit{Bias} the probability measure by an appropriate function and then choose a new random root with an appropriate distribution on the vertices of $\bs G$. Denote the resulting random rooted network by $[\bs G', \bs o']$. Explicit examples of such constructions in the literature will be recalled in Section~\ref{sec:extAndTrans}. % this construction will be provided in Section~\ref{sec:extAndTrans}.
%
%Consider a unimodular random (rooted) graph $[\bs G_0, \bs o_0]$, where $\bs o_0$ stands for the root. At each vertex $v$, append a random rooted graph $L(v)$ to the graph, assuming the $L(v)$'s are i.i.d., say to obtain $[\bs G, \bs o]$ (keep the root at the same place). With a special finiteness condition, one can obtain a unimodular graph as follows: Bias the probability measure by the number of vertices of $L({\bs o})$ and then choose a new root uniformly at random in $L(\bs o)$. Call this new random rooted network $[\bs G', \bs o']$. The details are explained in Example~\ref{ex:replacement}.  % (see Example~\ref{ex:replacement} and similar constructions in~\cite{processes}).
%
%Let $[\bs G, \bs o]$ be a unimodular graph, where $\bs o$ stands for the root of the random rooted network. 
%For each vertex $v$, append a new random rooted graph $L(v)$ and identify its root with $v$. Assume the $L(v)$'s are i.i.d. and let $[\hat{\bs G}, {\bs o}]$ be the resulting random rooted graph. This is not necessarily unimodular in general, but with a special finiteness condition, one can obtain a unimodular graph out of it as follows. Assume the number of vertices $\card{V(L({\bs o}))}$ of $L({\bs o})$ has finite expectation. First, bias the probability measure by $\card{V(L({\bs o}))}$ and then choose a new root uniformly at random in $L(\bs o)$. The new random rooted graph, namely $[\widetilde{\bs G}, \widetilde{\bs o}]$, is unimodular (see Example~\ref{ex:replacement} and similar constructions in~\cite{processes}).
%In the above example, 
One may intuitively accept that $[\bs G', \bs o']$ is \textit{equivalent} to $[\bs G, \bs o]$ if we disregard the root (\cite{processes}), or that $[\bs G', \bs o']$ and $[\bs G, \bs o]$ \textit{have the same non-rooted networks}. However, to state this in a mathematically precise way, one should answer the following question.
\begin{question}
	When do two given (not necessarily unimodular) random rooted networks have the same non-rooted networks?
	%Given two (not necessarily unimodular) random rooted networks, \\when do they have the same non-rooted networks?
\end{question}

Note that the question is not limited to the setting of the above example. In general, no special relation is assumed between the two random rooted networks  and they might be given only by two probability distributions.
%It seems 
The answer to this question is not straightforward {since the space of non-rooted networks is non-standard}.
Several definitions {of \textit{unrooted-equivalence}} are provided in Section~\ref{sec:shiftCoupling} as answers to this question, where some of the definitions are shown to be equivalent. It will be shown that in the above example, $[\bs G', \bs o']$ and $[\bs G, \bs o]$ are \textit{weakly} unrooted-equivalent, % have the same non-rooted networks in a \textit{weak sense} 
to be defined later. % (it will be seen that the distribution of the corresponding non-rooted networks, that will be defined precisely, are mutually absolutely continuous and may be different).
%(\ali{delete from here} seen well in Example~\ref{ex:path}). 
%\mar{ater: we can correct the above example to become strong} 
The \textit{strong} sense in our definition is that (the distribution of) the second one can be obtained from the first by a \textit{root-change} (note that there is a biasing in the above definition of $[\bs G', \bs o']$ before changing the root). Another definition is that the two random rooted networks agree on the \textit{invariant sigma-field}; 
i.e. any event that does not depend on the root occurs with equal probabilities.
%The invariant sigma-field is also used to define the distributions of the corresponding non-rooted networks of the two.
%i.e. for any event $A$ that does not depend on the root, the two random rooted networks fall in $A$ with equal probability. 
Some other definitions will also be given (Definition~\ref{def:same}). 

The main theorem (Theorem~\ref{thm:shiftCoupling}) in this work is that the last two definitions mentioned above are equivalent; namely, if two random rooted networks agree on the invariant sigma-field, then they can be obtained from each other by a root-change. This theorem, in its spirit, is similar to a {well known} result by Thorisson~\cite{Th96} that studies \textit{shift-coupling} of random elements in a space equipped with a group action. %\discuss{However, root-change cannot be described by a group action here.} %\discuss{However, there is no group action here to describe root-changes of rooted networks.} %Some applications of this theorem are also considered, mainly in the unimodular case,  as described in the following.

In Section~\ref{sec:unimodularCase}, we discuss applications of the main theorem in the unimodular case.
Theorem~\ref{thm:unimodularUnique} says that the distribution of a unimodular network is uniquely determined by the distribution of its non-rooted network, or equivalently, by its restriction to the invariant sigma-field. Theorem~\ref{thm:balancing} deals with \textit{invariant balancing transport kernels}, which are transport kernels that transport a given measure to another given one.
%Recently, in the context of stationary point processes, the so called \textit{balancing allocation rules} and \textit{balancing transport kernels} have been studied by many authors. \ali{These concepts will be reviewed in Section~\ref{sec:unimodularCase}.} 
In the context of stationary random measures and point processes, this concept has been studied by many authors recently. In this context, {under suitable assumptions,} the existence of a (random) balancing transport kernel that is invariant under translations is implied by the result of~\cite{Th96} (proved in the general case in~\cite{ThLa09}). % under some necessary and sufficient condition.
%A result by Last and Thorisson, stated in~\cite{ThLa09}, implies the existence of such transport kernels  for general stationary point processes and random measures. 
{Based on this abstract result}, several constructions have been provided, starting from~\cite{Li} and~\cite{HoPe06}, where the latter provides a transport kernel balancing between (a multiple of) the Lebesgue measure and the counting measure of the Poisson point process. 
Here, in the context of unimodular networks, we consider two discrete measures on the vertices of the random network. In Theorem~\ref{thm:balancing}, it will be proved that, roughly speaking, a balancing transport kernel between them exists if and only if the measures have equal \textit{sample intensities}; i.e. have the same expectation conditioned on the invariant sigma-field. A construction of such transport kernels is discussed in Section~\ref{sec:construction} (Theorem~\ref{thm:stable}) based on the construction of stable transports in~\cite{stable}, which is by itself based on~\cite{HoPe06}. It is a generalization of the Gale-Shapley stable matching algorithm in bipartite graphs~\cite{GaSh}.
%Now, let $[\bs G, \bs o]$ be a unimodular network and regard the vertices as a random environment. Then, consider two measures $w_1[\bs G, \cdot]$ and $w_2[\bs G, \cdot]$ on this environment. In Theorem~\ref{thm:balancing}, it will be proved that, roughly speaking, a balancing transport kernel between them \sout{on the random environment} exists if and only if the measures have equal \textit{sample intensities}; i.e. $w_i[\bs G, \bs o]$ have the same expectation conditioned on the invariant sigma-field. Construction of such a transport kernel is discussed in Section~\ref{sec:construction} based on the construction of stable transports in~\cite{stable}, which is by itself based on~\cite{HoPe06} ... ater: Gale-Shapley.

In Section~\ref{sec:extAndTrans}, we describe a general method for constructing unimodular networks. % from other ones. 
In some of the examples in the literature, such a network % a unimodular network 
is constructed by the following steps: Adding some vertices and edges to another unimodular network (called a \textit{network extension} here), then biasing the probability measure and finally applying a root-change. %This is called a \textit{network extension} here. 
These examples are unified in the method presented in Theorem~\ref{thm:extensionT}. It is also proved in Theorem~\ref{thm:uniExtGen} that this method gives all possible ways to \textit{unimodularize} the extension (to be defined more precisely later). %, defined in a precise sense.

Many of the definitions and results in this paper have analogues in the context of point processes and random measures, which are discussed in Section~\ref{sec:analogy}.

%\ali{ater: Bibliography}

This paper is structured as follows. The definition and basic properties of unimodular random networks are given in Section~\ref{sec:basic}. The definition of \textit{unrooted-equivalence} %\textit{having the same non-rooted networks} 
and the main shift-coupling theorem are presented in Section~\ref{sec:shiftCoupling}. The applications of the theorem to the unimodular case are studied in Section~\ref{sec:unimodularCase}. Section~\ref{sec:extAndTrans} deals with extensions of unimodular networks. Section~\ref{sec:construction} presents a construction of balancing transport kernels using stable transports. The proofs of some results are moved to Section~\ref{sec:proofs} to help to focus on the main thread of the paper.
Finally, Section~\ref{sec:analogy} reviews the analogous results in the context of point processes. 

%---------------------------------------------------------------------------
%
%: Change dual to the union of $G$ and its dual.
%
%In some literature, a random graph is constructed from another random graph. For example, given a unimodular planar random rooted graph $[\bs G, \bs o]$, a new probability measure on $\mathcal G_*$ is defined \textit{to make the dual of $\bs G$ unimodular}:
%\begin{equation}
%\label{eq:planarDual}
%\mathbb P^*[A]:=\frac 1{d(\bs o)}\omid{\sum_{f \sim \bs o} \identity A[\bs G^*,f]},
%\end{equation}
%where $G^*$ denotes the dual of a planar graph $G$ and $f$ in the sum ranges over the faces of $\bs G$ that are adjacent to $\bs o$.
%Our question is, the old and new probability measures are just two probability measures on $\mathcal G_*$. How do we say the latter is the dual of the former? Note that the construction doesn't even give a coupling, because it makes a bias in the probability measure (bias by the factor $d(\bs o)$). Likewise, if we embed the graphs in the natural graph that contains $\bs G$ and its dual, we have the intuition that the two random rooted graphs \textit{have the same non-rooted graph but their roots are chosen differently}.

\section{Random Rooted Graphs and Networks}
\label{sec:basic}
%
%: definitions of networks.
%
%unimodularity.
%
%: define the notation $T_G(o,v)$ and $T_G^+(v)$. Define transport, but without stress. Postpone transport and their relations to couplings to Sec~\ref{sec:proofs}?
%
%The notation $[G]=\pi[G,o]$.
%
%Covariant subsets and subnetworks (the latter in the connected case.)
%
%\begin{definition}
%	\label{def:unimodular}
%	: Unimodularity.
%\end{definition}
%
%\begin{lemma}
%	\label{lemma:happensAtRoot}
%	: Everything happens at the root.
%\end{lemma}
%
%We can work with $\mu$ or $[\bs G, \bs o]$ equivalently.

In this section, we recall the concepts of random networks and unimodularity mainly from~\cite{processes}.
%The following definitions are borrowed from~\cite{processes}.
A \defstyle{network} is a (multi-) graph $G=(V,E)$ equipped with a complete
separable metric space $\Xi$, called the \defstyle{mark space} and with
two maps from $V$ and $\{(v,e):v\in V, e\in E, v\sim e\}$ to $\Xi$, where the symbol $\sim$ is used for adjacency of vertices or edges.
The image of $v$ (resp. $(v,e)$) in $\Xi$ is called its \defstyle{mark}.
The degree of a vertex $v$ is denoted by $d(v)$ and the graph-distance of vertices $v$ and $w$ is denoted by $d(u,v)$. The symbol $N_r(v)$ is used for the closed ball centered at $v$ with radius $r$; i.e. the set of vertices with distance at most $r$ to $v$.

In this paper, all networks are assumed to be locally finite;
that is, the degrees of every vertex is assumed to be finite. 
Moreover, a network is assumed to be connected except
when explicitly mentioned. 
An \defstyle{isomorphism}
between two networks is a graph isomorphism that also preserves the marks. 
A \defstyle{rooted network} is a pair $(G,o)$ in which $G$ is a network
and $o$ is a distinguished vertex of $G$ called the \defstyle{root}.
An \defstyle{isomorphism} of rooted networks is a network
isomorphism that takes the root of one to that of the other.
Let $\mathcal G$ denote the set of isomorphism classes of connected %(connected {locally finite}) 
networks and $\mathcal G_*$ the set of 
isomorphism classes of connected %({ locally finite}) 
rooted networks. % \ali{\sout{(they are indeed sets by the assumption of locally finiteness)}}.
The set $\mathcal G_{**}$ is defined similarly for doubly-rooted networks; i.e. those with a pair
of distinguished vertices. The isomorphism class of a network $G$
(resp. $(G,o)$ or $(G,o,v)$) is denoted by $[G]$
(resp. $[G,o]$ or $[G,o,v]$). 

The sets $\mathcal G_*$ and $\mathcal G_{**}$ can be equipped
with natural metrics that make them a complete separable
metric space and equip them with the corresponding Borel sigma-fields. The distance of two rooted networks is defined based on the similarity of finite neighborhoods of their roots. See~\cite{processes} for the precise definition. 
There are two natural projections $\pi_1,\pi_2: \mathcal G_{**}\rightarrow \mathcal G_*$ obtained by {forgetting the second and the first root respectively.}
These projections are continuous and measurable. In contrast, there is no useful metric on $\mathcal G$. However, as will be defined in Definition~\ref{def:J}, the natural projection $\pi:\mathcal G_*\rightarrow\mathcal G$ of forgetting the root induces a sigma-field on $\mathcal G$. This sigma-field is extensively used in this paper although it does not make $\mathcal G$ a standard space.
%has less interesting properties compared to that of $\mathcal G_*$. \ali{We guess that this sigma-field doesn't make $\mathcal G$ a standard Borel space.}

\begin{definition}
	A \defstyle{random rooted network} is a random element in $\mathcal G_*$ and is represented in either of the following ways.
	\begin{itemize}
		\item A probability measure $\mu$ on $\mathcal G_*$.
		\item A measurable function from some probability space
		to $\mathcal G_*$ that is denoted by bold symbols $[\bs G,\bs o]$. Here, $\bs G$ and $\bs o$ represent the network and the root respectively.
	\end{itemize}
%	Moreover, bold symbols are used only in the random case. %For a probability measure $\mu$ on $\mathcal G_*$ and a measurable function $f$ on $\mathcal G_*$, the integral of $f$ w.r.t. $\mu$ is denoted by $\mathbb E_{\mu}[f]$.
%	These representations are mostly treated equally in this paper. Therefore, all definitions and results expressed for random rooted networks $[\bs G, \bs o]$ also make sense for probability measures on $\mathcal G_*$.
\end{definition}

Note that the whole symbol $[\bs G, \bs o]$ represents one random object, which is a random equivalence class of rooted networks. Therefore, any formula using $\bs G$ and $\bs o$ should be well defined for equivalence classes of rooted networks; i.e. should be invariant under rooted isomorphisms. Moreover, bold symbols are used only in the random case.

The relation between the two representations is expressed by the equation $\mu(A)=\myprob{[\bs G, \bs o]\in A}$ for events $A\subseteq\mathcal G_*$; i.e. $\mu$ is the distribution of the random object.
%\ali{Conversely, any probability measure $\mu$ on $\mathcal G_*$ can be considered as a random rooted network by considering the identity function on the canonical probability space $(\mathcal G_*,\mu)$.}
These representations are mostly treated equally in this paper. Therefore, all definitions and results expressed for random rooted networks $[\bs G, \bs o]$ also make sense for probability measures on $\mathcal G_*$.
%Note that $\bs o$ does not have a meaning on its own in this notation, but as an example,
%the degree of $\bs o$ does and is a random variable. 

%A \defstyle{random rooted network} is a random element in $\mathcal G_*$, that is,
%a measurable function from some probability space
%to $\mathcal G_*$.
%This function will be denoted by bold symbols:
%$[\bs G,\bs o]$. Note that $\bs o$ does not have a meaning on its own, but as an example,
%the degree of $\bs o$ does and is a random variable. %In this paper, bold symbols are used only for random networks.
%In most parts of this paper, only the distribution $\mu$ of of a random network $[\bs G, \bs o]$ is used, which is a probability measure on $\mathcal G_*$ define by $\mu(A):=\myprob{[\bs G, \bs o]\in A}$. Therefore, all definitions and results for random networks also make sense for probability measures on $\mathcal G_*$. We will use both notions for better readability. 

%A probability measure $\mu$
%on $\mathcal G_*$ can also be regarded as a random network when
%considering the identity map on the canonical
%probability space $(\mathcal G_*, \mu)$. 

For a measurable function
$g:\mathcal G_{**}\rightarrow \mathbb R^{\geq 0}$, a network $G$ and $u,v\in V(G)$, let
%\mar{\ali{I think the second version is more similar to the other definitions afterwards.}}
%For all measurable functions $g:\mathcal G_{**}\rightarrow \mathbb R^{\geq 0}$,
%all rooted networks $G$ and a vertex $o\in V(G)$, $[G,o]\in\cal G_*$, let 
\[
g_G(u,v):= g[G,u,v],
\]
where brackets $[\cdot]$ are used as a short form of $([\cdot])$. Also, for $o\in V(G)$, let
\begin{eqnarray*}
%\label{eq:gplus}
g^+_G(o)&:=&\sum_{v\in V(G)} g[G,o,v],\\
g^-_G(o)&:=&\sum_{v\in V(G)} g[G,v,o].
%\label{eq:gmin}
\end{eqnarray*}
%It can be shown that these are well-defined measurable functions on $\mathcal G_*$.

\begin{definition} 
	\label{def:unimodular}
	A random rooted network $[\bs{G}, \bs o]$ is \defstyle{unimodular} if for
	all measurable functions $g:\mathcal G_{**}\rightarrow \mathbb R^{\geq 0}$,
	\begin{equation} 
	\label{eq:unimodular}
	\omid{ g_{\bs G}^+(\bs o)} = \omid{ g_{\bs G}^-(\bs o)},
	\end{equation}
	where the expectations may be finite or infinite. The term \defstyle{unimodular network} is used as an abbreviation for {unimodular random rooted network}.
	A probability measure on $\mathcal G_*$ is called \defstyle{unimodular}
	when, by considering it as a random rooted network, one gets a unimodular network. 
\end{definition}

\begin{remark}
	\label{rem:transport}
	For a function $g$ as above, $g_G(\cdot,\cdot)$ can be regarded as a function on $V(G)\times V(G)$ {(or a transport kernel on $V(G)$)} defined for all networks $G$.
	One can interpret $g_G(u,v)$ as the amount of mass that is transported from $u$ to $v$. 
	Using this intuition, $g^+_G(o)$ (resp. $g^-_G(o)$) can be seen as the amount of mass that goes out of (resp. comes into) $o$ and \eqref{eq:unimodular} expresses 
	some conservation of mass in expectation. It is referred to as the \defstyle{mass transport principle} in the literature. With this analogy, a measurable function $g:\mathcal G_{**}\rightarrow\mathbb R^{\geq 0}$ is also called an \defstyle{invariant transport kernel} in this paper. %It is called \defstyle{Markovian} if 
\end{remark}

The \defstyle{invariant sigma-field $I$} on $\mathcal G_*$ is the family of events in $\mathcal G_*$ that are invariant under changing the root; i.e., events $A\subseteq\mathcal G_*$ such that for every rooted network $(G,o)$ and every $v\in V(G)$, if $[G,o]\in A$, then $[G,v]\in A$. Events in $I$ are also called \defstyle{invariant events} here. A unimodular network $[\bs G, \bs o]$ is called \defstyle{extremal} %if for any invariant event $A\in I$, one has $\myprob{[\bs G, \bs o]\in A}\in \{0,1\}$.
if any invariant event has probability 0 or 1.

A measurable function on $w:\mathcal G_*\rightarrow\mathbb R$ is $I$-measurable if and only if it \defstyle{doesn't depend on the root}; i.e. for every rooted network $(G,o)$ and every $v\in V(G)$, one has $w[G,o]=w[G,v]$. Also, if $[\bs G, \bs o]$ is a random rooted network, we say $w[\bs G, \bs o]$ \defstyle{doesn't depend on the root almost surely} if almost surely, for all $v\in V(\bs G)$, one has $w[\bs G,\bs o]=w[\bs G,v]$.

The following definition is borrowed from~\cite{vertexShift}.
\begin{definition}
	\label{def:covariantSubset}
	A \defstyle{covariant subset}  (of the vertices) is a function
	$S$ which associates to each network $G$ a set $S_G\subseteq V(G)$
	such that $[G,o]\mapsto \identity{\{o\in S_G\}}$ is a well-defined and measurable function on $\mathcal G_*$.
%	which is covariant under network isomorphisms (that is,
%	for all isomorphisms $\rho:G\rightarrow G'$, one has $\rho(S_G) = S_{G'}$),
%	and such that the function
%	$[G,o]\mapsto \identity{\{o\in S_G\}}$ is measurable.
	By an abuse of notation, we use the same symbols for the subnetwork induced by $S_{G}$ (i.e. the restriction of $G$ to $S_{G}$) for all networks $G$. This is called a \defstyle{covariant subnetwork}.
\end{definition}

Note that in this definition, $S$ should be covariant under network isomorphisms, that is, for all isomorphisms $\rho:G\rightarrow G'$, one should have $\rho(S_G) = S_{G'}$. Moreover,
For any event $A\subseteq\mathcal G_*$, $S_{G}:=\{v\in V(G): [G,v]\in A\}$ is a covariant subset. This easily implies that covariant subsets are in one-to-one correspondence with measurable subsets of $\mathcal G_*$.

The following lemma is straightforward and its proof is skipped. See~\cite{vertexShift} or~\cite{processes}.
\begin{lemma} 
	\label{lemma:happensAtRoot}
	Let $[\bs G,\bs o]$ be a unimodular network and $S$ be a covariant
	subset of the vertices. Then $\mathbb P[S_{\mathbf G}\neq \emptyset]>0$
	if and only if $\mathbb P[\bs o \in S_{\mathbf G}]>0$. Equivalently, $S_{\bs G}=V(\bs G)$ a.s. if and only if $\bs o\in S_{\bs G}$ a.s.
\end{lemma}

%\sout{For $r\geq 0$, the ball of
%	(graph-distance) radius $\lceil r \rceil$ with center $v\in V(G)$ 
%	is denoted by $N_r(G,v)$. }
%
%\sout{The set of all network isomorphisms from $G$ to itself is referred
%	to as the set of automorphisms of $G$ and is denoted by $Aut(G)$. }
%
%\sout{With the Borel sigma-field on $\mathcal G_*$, measurable
%	subsets of $\mathcal G_*$ and measurable functions on $\mathcal G_*$ are,
%	roughly speaking, those which can be identified by looking at  finite
%	neighborhoods of the root. For example, the degree of the root is a
%	measurable function and 
%	the set of rooted networks with sub-polynomial growth of neighborhoods
%	(i.e. $\limsup_{r\rightarrow\infty}\frac 1{r^k} {\card{N_{r}(o)}}<\infty$
%	for some $k$) is a measurable subset.}
%
%\sout{The value of $[\bs G,\bs o]$ on $\omega\in \Omega$ is denoted by
%	$[\bs G_{\omega}, \bs o_{\omega}]$, which is an isomorphism class
%	of rooted networks. The subscript $\omega$ is omitted if it is
%	clear from the context.}

\section{Shift-Coupling of Random Rooted Networks}
\label{sec:shiftCoupling}
In this section, different formulations of \textit{unrooted-equivalence} %(or \textit{having the same non-rooted networks}) 
are defined %(Definition~\ref{def:same}) after some preliminary definitions. Then, 
and the main theorem of this paper is presented, which studies the implications between these formulations. The proofs of most of the results are moved to Section~\ref{sec:proofs} to help to focus on the main thread. The reader can either see the proofs first or proceed to the next results with no problem.

\subsection{Definitions}
The following definitions are needed for stating the main definition (Definition~\ref{def:same}).

\begin{definition}
	\label{def:J}
	The projection $\pi:\mathcal G_*\rightarrow\mathcal G$ defined by $\pi[G,o]:=[G]$ induces a sigma-field, namely $J$, on $\mathcal G$ as follows.
	%The family of subsets $B\subseteq\mathcal G$ such that $\pi^{-1}(B)$ is measurable. 
	\[
	J:=\{B\subseteq\mathcal G: \pi^{-1}(B)\text{ is measurable}\}.
	\]
	A \defstyle{random non-rooted network} is a random element in, or a probability measure on $(\mathcal G,J)$, although it does not form a standard probability space (Proposition~\ref{prop:nonstandard} below). 
	If $[\bs G, \bs o]$ is a random rooted network with distribution $\mu$, the symbol $[\bs G]$ is used for its corresponding random non-rooted network whose distribution is $\pi_*\mu$. It can also be seen as a natural coupling of $\mu$ and $\pi_*\mu$.
	
	%From now on, the same symbol $\mathcal G$ is used for the measurable space $(\mathcal G, J)$.
	%%	\[
	%%		 I:=\{measurable! A\subseteq\mathcal G_*: \text{ if } [G,o]\in A \text{ and } v\in V(G), \text{ then } [G,v]\in A \}.
	%%	\]
	%%	
	%%	The \defstyle{invariant sigma-field $I$} on $\mathcal G_*$ is the family of events that are invariant under changing the roots of the graphs; i.e. the family of measurable sets $A\subseteq\mathcal G_*$ such that if $[G,o]\in A$ and $v\in V(G)$, then $[G,v]\in A$.
	%%	
	%%	The invariant sigma-field induces a sigma-field on $\mathcal G$: The family of sets 
	%%	\[
	%%	\{[G]: [G,v]\in A \text{ for some (equivalently, for any) } v\in V(G)\}
	%%	\]
	%%	for all $A\in I$. This is equal to the family of subsets $B\subseteq \mathcal G$ such that $\pi^{-1}(B)$ is measurable, where $\pi:\mathcal G_*\rightarrow \mathcal G$ is the projection map $\pi[G,o]:=[G]$.
	%
	%If $[\bs G, \bs o]$ is a random rooted network with distribution $\mu$, it makes sense to consider its random non-rooted network $[\bs G]$ as a random element in $(\mathcal G,J)$ whose distribution is $\pi_*\mu$. Note also that probability measures on $(\mathcal G,J)$ are in one-to-one correspondence with probability measures on $(\mathcal G_*,I)$.
\end{definition}

Non-standardness of $(\mathcal G,J)$ is stated in the following proposition. 
%In contrast with $\mathcal G_*$, there is no useful metric on $\mathcal G$; which is stated concretely in the following proposition. 
It is essentially an easy result in theory of \textit{smooth Borel equivalence relations}. See the notes in Subsection~\ref{subsec:shiftCoupling:notes}. % for an introduction to this subject.

%to the best of the author's knowledge. 
%So, $J$ has less interesting properties than that of $\mathcal G_*$.
%For example an {extremal} unimodular probability measure $\mu$
%induces the probability measure $\pi_*\mu$ on $(\mathcal G,J)$ that is $\{0,1\}$-valued.
%\ali{Indeed, we guess the answer to the following problem is negative.}

\begin{proposition}
	\label{prop:nonstandard}
	The measurable space $(\mathcal G, J)$ is not a standard Borel space. More precisely, there is no metric on $\mathcal G$ that makes it a Polish space whose Borel sigma-field is $J$.
	%\mar{\ali{or $J$ is included in (or includes?) the Borel sigma-field?}}	
\end{proposition}

Due to non-standardness, several classical tools of probability theory may fail for random non-rooted networks; e.g. conditional expectation. However, it poses no problem for the arguments in this paper; e.g. equality of distributions, pull-back and push-forward of distributions, etc.

Note also that the map $\pi^{-1}$ corresponds $J$ bijectively to the invariant sigma-field $I$ on $\mathcal G_*$. Therefore, probability measures on $(\mathcal G,J)$ are in one-to-one correspondence with probability measures on $(\mathcal G_*,I)$.

\begin{definition}%[Biasing]
	\label{def:bias}
	%\mar{\ali{Move to prev section? with the next lemma.}}
	Let $\mu$ be a probability measure on $\mathcal G_*$
	%unimodular network with distribution $\mathbb P$ 
	and $w:\mathcal G_*\rightarrow\mathbb R^{\geq 0}$ be a measurable function. Assume $0<\int_{\mathcal G_*}w d\mu<\infty$. By \defstyle{biasing $\mu$ by $w$} we mean the following measure on $\mathcal G_*$. % measure $w\mu$ on $\mathcal G_*$ defined by
	%	\mar{\ali{ater: Is the notation $w\mu$ good?\\
	%		Move to prev sec?}}
	\[
	A\mapsto\frac 1{\int_{\mathcal G_*} wd\mu} \int_{\mathcal G_*} w\identity{A} d\mu.
	\]
	%	\[
	%		A\mapsto\frac 1{\omid{w[\bs G, \bs o]}} \omid{w[\bs G, \bs o]\identity{A}[\bs G, \bs o]}.
	%	\]
	The choice of the denominator ensures that the result is a probability measure. It is the unique probability measure on $\mathcal G_*$ whose Radon-Nikodym derivative w.r.t. $\mu$ is proportional to $w$.
	Biasing a probability measure on $\mathcal G$ is defined similarly.
\end{definition}

It can be seen that biasing $\mu$ by $w$ is equal to $\mu$ if and only if $w$ is essentially constant (w.r.t. $\mu$); i.e. for some constant $c$ one has $w=c$, $\mu$-a.s.
Note that $\int_{\mathcal G_*}wd\mu$ is not assumed to be equal to one. 
As an example, for an event $B\subseteq \mathcal G_*$, conditioning $\mu$ on $B$ is just biasing $\mu$ by the indicator function $1_B$. 

%The following lemma is straightforward and we skip its proof. In the lemma, note that $I$-measurability of a function of $[\bs G, \bs o]$ means that the function doesn't depend on the root, thus, it induces a function on $\mathcal G$.
\begin{lemma}
	\label{lemma:biasG}
	By biasing the distribution of a random rooted network $[\bs G, \bs o]$ by a function $w$, the distribution of $[\bs G]$ becomes biased by $\omidCond{w[\bs G, \bs o]}{I}$, where the latter, which is $I$-measurable, is considered as a function 
	of $[\bs G]$ 
	%on $(\mathcal G, J)$ 
	with a slight abuse of notation (see Section~\ref{sec:basic}).
\end{lemma}
This lemma is straightforward and we skip its proof.
%In this lemma, note that $\omidCond{w([\bs G, \bs o])}{I}$ is an $I$-measurable function. In words, it doesn't depend on the root almost surely. % and thus it doesn't depend on the root. 
%Therefore, it can be considered as a function on $\mathcal G$ with an abuse of notation.

\begin{definition}%[Root-Change]
	\label{def:rootChange}
	Let $[\bs G, \bs o]$ be a (not necessarily unimodular) random rooted network and $T:\mathcal G_{**}\rightarrow\mathbb R^{\geq 0}$ be a measurable function. Assume 
	$T_{\bs G}^+(\bs o)=1$ a.s.
	%\[T_{\bs G}^+(\bs o):=\sum_{v\in V(\bs G)}T_{\bs G}(\bs o, v)=1, \quad a.s.\]
	Conditioned on $[\bs G, \bs o]$, choose a new root in $V(\bs G)$ with distribution $T_{\bs G}(\bs o, \cdot)$; i.e. consider the following probability measure on $\mathcal G_*$.
	%	By \defstyle{root-change of $[\bs G, \bs o]$ with law $T$}, we mean the random rooted network obtained from $[\bs G,\bs o]$ by choosing 
	%	
	%	(a random rooted network with) the probability measure $\transport{T}{\mathbb P}$ on $\mathcal G_*$; i.e.
	\begin{equation}
	\label{eq:rootChange}
	A\mapsto \omid{\sum_{v\in V(\bs G)}T_{\bs G}(\bs o, v) \identity{A}[\bs G, v]}.
	\end{equation}
	%	In other words, conditional on $[\bs G, \bs o]$, choose a new root in $\bs G$ with distribution $T_{\bs G}[\bs o, \cdot]$. 
	%	
	%	: $T$ can be regarded as a Markov kernel that transports $\mathbb P$ to $\transport{T}{\mathbb P}$.
	%This probability measure can be written as $\pi_{2*}(T^{\uparrow}\mathbb P)$, where $\mathbb P$ is the distribution of $[\bs G, \bs o]$.
	Any random rooted network with this distribution is called the \defstyle{root-change of $[\bs G, \bs o]$ by kernel $T$}. 
\end{definition}

%Note that only the distributions of the {two} random rooted networks are important in this definition. We mostly don't use the natural coupling of them suggested by the root-change.

\begin{lemma}
	\label{lemma:rootChange}
	If $[\bs G',\bs o']$ is a root-change of $[\bs G, \bs o]$, then $[\bs G, \bs o]$ is also a root-change of $[\bs G',\bs o']$.
\end{lemma}

%	With the above definitions, one can describe \eqref{eq:planarDual} as biasing the probability measure $\mathbb P$ by $w[\bs G, \bs o]:=d(\bs o)$ and then applying a root-change by moving the root to a uniformly random adjacent face.

We are now ready to present the main definition. % define when two random rooted networks have the same non-rooted networks.

\begin{definition}
	\label{def:same}
	Let $[\bs G_1, \bs o_1]$ and $[\bs G_2, \bs o_2]$ be (not necessarily unimodular) random rooted networks. The following conditions are different definitions for $[\bs G_1,\bs o_1]$ and $[\bs G_2,\bs o_2]$ to be \defstyle{unrooted-equivalent} (or to \defstyle{have the same non-rooted networks}).
	
	\begin{description}
		\item [(B)\label{conditionB}]
		%\mar{"each one" is important}
		The distribution of each one is obtained from the other by a \textit{biasing} and then a \textit{root-change}.
		\item [(R)\label{conditionR}] 
		The distribution of $[\bs G_2,\bs o_2]$ is obtained from $[\bs G_1,\bs o_1]$ by a {root-change}.
		\item [(C)\label{conditionC}]
		%\mar{doubt: isn't there a measurability issue for \textit{selecting} the isomorphism? Answer: No!}
		There is a coupling of them (i.e. a probability measure on $\mathcal G_*\times\mathcal G_*$ whose marginals are identical with the distributions of $[\bs G_i,\bs o_i]$'s) which is concentrated on the set of pairs of rooted networks with the same non-rooted networks; i.e. $\{([G_1,o_1],[G_2,o_2]): [G_1]=[G_2]\}$. 
		\item [(D)\label{conditionD}]
		There is a random doubly-rooted  network $[\bs G, \bs o, \bs o']$ such that $[\bs G, \bs o]$ and $[\bs G, \bs o']$ have the same distributions as $[\bs G_1, \bs o_1]$ and $[\bs G_2, \bs o_2]$ respectively.
		\item [(F)\label{conditionF}]
		By forgetting the roots, the random non-rooted networks $[\bs G_1]$ and $[\bs G_2]$ have the same distribution on $(\mathcal G, J)$. Equivalently, the distributions of $[\bs G_1, \bs o_1]$ and $[\bs G_2,\bs o_2]$ agree on the invariant sigma-field $I$.
	\end{description}
	%	ater: update this. Condition~\ref{conditionB} is called \textbf{weak} and all other senses are called \textbf{strong}. This will justified by Theorem~\ref{thm:implications}.
\end{definition}

As mentioned in the introduction, the definition with Condition~\ref{conditionB} is used heuristically in some examples in the literature, some of which will be mentioned in Section~\ref{sec:extAndTrans}.

\subsection{Main Theorems}

Here, we study the 
%\ali{The main theorem of this paper studies the} 
implications between the conditions in Definition~\ref{def:same}. At first sight, Condition~\ref{conditionF} may seem weaker than the other ones, because the other conditions assume the existence of a third object. But this is not the case as shown below.
% This is the hardest part in the implication diagram.

%\begin{remark}
%	Non-unimodular random networks are also considered in Definition~\ref{def:same}.
%\end{remark}

%Heuristically, it seems that Sense 4 is the weakest and senses 2 and 3 are the strongest. More precisely,

%At first sight, Sense~\ref{conditionF} may seem too weak to imply the other senses. But surprisingly, it does imply them as in the following theorem. This is the hardest part in the implication diagram.

\begin{theorem}[Shift-Coupling]
	\label{thm:shiftCoupling}
	Let $[\bs G, \bs o]$ and $[\bs G', \bs o']$ be (not necessarily unimodular) random rooted networks. Then, $[\bs G',\bs o']$ can be obtained from $[\bs G,\bs o]$ by a root-change if and only if their distributions agree on the invariant sigma-field. In other words, conditions \ref{conditionF} and~\ref{conditionR} are equivalent.
\end{theorem}

Most results of this paper are based on the above Theorem.
Also, the chosen name \textit{shift-coupling} is justified in the notes in Subsection~\ref{subsec:shiftCoupling:notes}.
This result is the main part in the following implications.
%Theorem~\ref{thm:shiftCoupling} is the main part in the following implications.

\begin{theorem}
	\label{thm:implications}
	Conditions~\ref{conditionR}, \ref{conditionC}, \ref{conditionD} and \ref{conditionF} are equivalent and imply Condition~\ref{conditionB}.
\end{theorem}

It should be noted that Condition~\ref{conditionB} does not imply the other conditions (see Remark~\ref{rem:weakVSstrong} below).
Theorem~\ref{thm:implications} allows us to define the following.

\begin{definition}
	\label{def:weakStrong}
	Under the assumptions of Definition~\ref{def:same}, $[\bs G_1,\bs o_1]$ and $[\bs G_2,\bs o_2]$ are \defstyle{weakly unrooted-equivalent} %have the same non-rooted networks 	\defstyle{in the weak sense} 
	if Condition~\ref{conditionB} holds and \defstyle{strongly unrooted-equivalent} if the other equivalent conditions hold.
	%\textbf{in the weak sense} if Condition~\ref{conditionB} holds. Similarly, they  (resp. one of the other conditions holds).
\end{definition}

\begin{remark}
	\label{rem:weakVSstrong}
	According to Condition~\ref{conditionF}, when two random rooted networks are weakly unrooted-equivalent, % have the same non-rooted networks in the weak sense, 
	the distributions of the corresponding non-rooted networks may be different (but are always mutually absolutely continuous by Lemma~\ref{lemma:biasG}). This difference can be seen clearly in Example~\ref{ex:path}. %It also shows that Condition~\ref{conditionB} doesn't imply the other conditions.
\end{remark}

\subsection{Some Applications}
The following propositions are presented here as corollaries of Theorem~\ref{thm:shiftCoupling}. More important applications of the theorem will be presented in the next sections. %Proposition~\ref{prop:conditioning} is also a special case of Theorem~\ref{thm:balancing} below. %Example~\ref{ex:subnetwork} 

%MEASURE NOTATION
%\begin{proposition}
%	\label{prop:conditioning}
%	%\mar{write for biasing instead of conditioning?}
%	%\ali{(delete this? might look non-interesting for the reader; e.g. it always holds if $I$ is trivial. Thm \ref{thm:balancing} is more non-intuitive. Maybe mention this prop after that prop? Or only make it an example)}
%	Let $\mu$ be a (not necessarily unimodular) probability measure on $\mathcal G_*$ and $S$ be a covariant subset such that $\mu(B)>0$. Let $\nu$ be $\mu$ conditioned on $B$. Then, the following are equivalent.
%	\begin{enumerate}[(i)]
%		\item $\nu$ can be obtained from $\mu$ by a root-change.
%		\item $\mu(B|I)$ is essentially constant. %= \lambda \probCond{B_2}{I}$ a.s. for some constant $\lambda$.
%		 %\ali{(what is $\mathbb P$? change $\mu$ to $\mathbb P$?)}
%	\end{enumerate}
%	%Let $[\bs G, \bs o]$ be an arbitrary random network and $S$ and $S'$ be covariant subsets. 
%\end{proposition}

%GRAPH NOTATION
\begin{proposition}
	\label{prop:conditioning}
	%\ali{(delete this? might look non-interesting for the reader; e.g. it always holds if $I$ is trivial. Thm \ref{thm:balancing} is more non-intuitive. Maybe mention this prop after that prop? Or only make it an example)}
	Let $[\bs G, \bs o]$ be a (not necessarily unimodular) random rooted network and $S$ be a covariant subset (Definition~\ref{def:covariantSubset}) such that $\myprob{\bs o\in S_{\bs G}}>0$. Denote by $[\bs G', \bs o']$ the random rooted network obtained by conditioning $[\bs G, \bs o]$ on $\bs o\in S_{\bs G}$. Then, the following are equivalent.
	\begin{enumerate}[(i)]
		\item $[\bs G', \bs o']$ can be obtained from $[\bs G, \bs o]$ by a root-change.
		\item $\probCond{\bs o\in S_{\bs G}}{I}$ is essentially constant. %= \lambda \probCond{B_2}{I}$ a.s. for some constant $\lambda$.
	\end{enumerate}
\end{proposition}

\begin{proof}
	The distribution of $[\bs G', \bs o']$ is obtained from that of $[\bs G, \bs o]$ by biasing by the function $\identity{\{\bs o\in S_{\bs G} \}}$. Lemma~\ref{lemma:biasG} implies that the distribution of $[\bs G']$ is obtained from that of $[\bs G]$ by biasing  by $\probCond{\bs o\in S_{\bs G}}{I}$ (considered as a function on $\mathcal G$).
	
	First, assume the bias function $\probCond{\bs o\in S_{\bs G}}{I}$ is essentially constant. It follows that  $[\bs G']$ and $[\bs G]$ are identically distributed; i.e. the distributions of $[\bs G, \bs o]$ and $[\bs G', \bs o']$ agree on the invariant sigma-field. 
	Thus, Theorem~\ref{thm:shiftCoupling} implies that $[\bs G',\bs o']$ can be obtained from $[\bs G, \bs o]$ by a root-change.
	
	Conversely, assume $[\bs G',\bs o']$ can be obtained from $[\bs G, \bs o]$ by a root-change. Theorem~\ref{thm:shiftCoupling} implies that their distributions agree on the invariant sigma-field. In other words, $[\bs G']$ and $[\bs G]$ have the same distribution. Since the former is obtained by biasing the latter by $\probCond{\bs o\in S_{\bs G}}{I}$, it follows that the bias function is essentially constant and the claim is proved.
	%	
	%	
	%	
	%	
	%		\mar{\ali{ater: write for $[\bs G, \bs o]$ because of the concept of sample intensity}}
	%	Since $\nu$ is just biasing $\mu$ by the function $\identity{B}$, Lemma~\ref{lemma:biasG} implies that $\pi_*\nu$ is biasing $\pi_*\mu$ by $\mu(B|I)$ (considered as a function on $\mathcal G$).
	%	
	%	First, assume the bias function $\mu(B|I)$ is essentially constant. 	It follows that  $\pi_*\mu=\pi_*\nu$. 
	%	In other words, $\mu$ and $\nu$ agree on the invariant sigma-field. 
	%	Thus, Theorem~\ref{thm:shiftCoupling} implies that $\nu$ can be obtained from $\mu$ by a root-change.
	%	
	%	Conversely, assume $\nu$ is obtained from $\mu$ by a root change. Theorem~\ref{thm:shiftCoupling} implies that $\nu$ and $\mu$ agree on the invariant sigma-field.
	%	In other words, $\pi_*\nu=\pi_*\mu$. Since $\pi_*\nu$ is biasing $\pi_*\mu$ by $\mu(B|I)$, it follows that $\mu(B|I)$ is essentially constant and the claim is proved.
\end{proof}

%The proof of this proposition using Theorem~\ref{thm:shiftCoupling} does not give a constructive way to obtain a desired root-change (in terms of a realization of the network). A construction will be presented in Section~\ref{sec:construction}, though for the unimodular case only. Also, the construction will be restated for the following application of Proposition~\ref{prop:conditioning}.

%\ali{ater: postpone this par to Section~\ref{sec:construction} and only mention non-constructiveness here.}
%This result is an application of Theorem~\ref{thm:shiftCoupling}. Note that $\mu_1$ and $\mu_2$ always have the same non-rooted network in the weak sense. However, Theorem~\ref{thm:shiftCoupling} does not construct the root change by looking only at a realization of the network. In Section~\ref{sec:construction}, a construction is presented by assuming $\mu$ is unimodular. Currently, the authors don't know a construction in the general case. \ali{ater: see the description after Proposition ???}

\begin{proposition}[Extra Head Scheme]
	\label{prop:extraHead}
	Let $[\bs G, \bs o]$ be a unimodular graph. Add i.i.d. marks in $\{0,1\}$ to the vertices with Bernoulli distribution with parameter $0<p\leq 1$. 
	If $[\bs G, \bs o]$ is infinite a.s. then there exists a root-change that when applied to $[\bs G, \bs o]$, the result is the same (in distribution) as $[\bs G, \bs o]$ except that the mark of the root is forced to be 1.
	%Let $[\bs G', \bs o']$ be same as $[\bs G, \bs o]$, except that the mark of the root is forced to be 1. Then, $[\bs G', \bs o']$ can be obtained from $[\bs G, \bs o]$ by a root-change.
\end{proposition}

The condition of being infinite is necessary in this proposition as explained in Remark~\ref{rem:extrahead}.
See also~\cite{BeLySc} for the precise definition of adding i.i.d. marks to the vertices. The name \textit{extra head scheme} is borrowed from an analogous definition in~\cite{HoPe05} as will be explained in Section~\ref{sec:analogy}.

\begin{proof}[Proof of Proposition~\ref{prop:extraHead}]
	Note that the desired random rooted network can be obtained by conditioning $[\bs G, \bs o]$ on $m(\bs o)=1$, where $m(\cdot)$ denotes the marks of the vertices. Therefore, by Proposition~\ref{prop:conditioning}, it is enough to prove that $\probCond{m(\bs o)=1}{I}$ is essentially constant.
%	\[
%	\probCond{m(\bs o)=1}{I}=p,\quad a.s.
%	\]
	Let $A\in I$ be an invariant event. By Lemma~\ref{lemma:ergodic} below,  $\probCond{A}{[\bs G, \bs o]}$ is $\{0,1\}$-valued and does not depend on the root a.s. Therefore, conditioned on $[\bs G, \bs o]$, $A$ is independent of any random variable including $m(\bs o)$. Thus,
	\[
	\omid{m(\bs o) \identity{A}} = \omid{\omidCond{m(\bs o)}{[\bs G, \bs o]} \omidCond{\identity{A}}{[\bs G, \bs o]}} = \omid{p \omidCond{\identity{A}}{[\bs G, \bs o]}} = \omid{p\identity{A}}.
	\]
	This equation for all $A\in I$ implies that $\omidCond{m(\bs o)}{I}=p$ a.s. So $\probCond{m(\bs o)=1}{I}=p$ a.s. and the claim is proved.
\end{proof}

The following lemma is used in the proof of Proposition~\ref{prop:extraHead} and is interesting in its own. %In this lemma, following~\cite{processes}, a unimodular network is called \defstyle{extremal} if any event that does not depend on the root has probability 0 or 1.
It is similar to the ergodicity of the Bernoulli point process on $\mathbb Z^d$ or the Poisson point process in $\mathbb R^d$ (see Section~\ref{sec:analogy}). 

\begin{lemma}
	\label{lemma:ergodic}
	%\mar{Change to proposition?}
	Let $[\bs G, \bs o]$ be a unimodular graph and $[\bs G', \bs o']$ be a random network obtained by adding i.i.d. marks to the vertices of $[\bs G, \bs o]$. If $[\bs G, \bs o]$ is extremal and almost surely infinite, then so is $[\bs G', \bs o']$. More generally, if $[\bs G, \bs o]$ is infinite a.s., then for any invariant event $A\in I$, 
	\begin{equation}
	\label{eq:ergodic:1}
	\probCond{[\bs G', \bs o']\in A}{[\bs G, \bs o]}\in \{0,1\}, \quad a.s.
	\end{equation}
	and the left hand side does not depend on the root a.s.
\end{lemma}

Note that in the statement of the lemma, the natural coupling of $[\bs G, \bs o]$ and $[\bs G', \bs o']$ is considered to enable us to condition $[\bs G', \bs o']$ on $[\bs G, \bs o]$. The proof is presented in Section~\ref{sec:proofs}.

\begin{remark}
	\label{rem:extrahead}
	The claims of Lemma~\ref{lemma:ergodic} and Proposition~\ref{prop:extraHead} are false for any finite unimodular network. Note that in this case, conditioned on $[\bs G, \bs o]$, with positive probability the marks of all vertices are 0. This contradicts~\eqref{eq:ergodic:1}. Also, the same property holds in any root-change of the network, contradicting the claim of Proposition~\ref{prop:extraHead}.
\end{remark}

\subsection{Notes}
\label{subsec:shiftCoupling:notes}
The name \textit{shift-coupling} for Theorem~\ref{thm:shiftCoupling} is borrowed from the analogous result of~\cite{Th96}. This result studies when two random elements in a space equipped with some group action have a coupling such that the second one is obtained from the first by a shift corresponding to a random element of the group, called a shift-coupling in the literature.
%More about this analogy will be discussed in Section~\ref{sec:analogy}.
{Here, instead of a group action, we have root-changes as in Condition~\ref{conditionC} of Definition~\ref{def:same}, which don't form a group. 
	In Section~\ref{sec:proofs}, a proof of Theorem~\ref{thm:shiftCoupling} is presented by mimicking that of~\cite{Th96}. A second proof is also presented using the result of~\cite{Th96}.
%	 by considering countably many root-changes that form a group (with the discrete topology), where the existence of such a group is implied by a result of~\cite{FeMo77}. See the second proof of Theorem~\ref{thm:shiftCoupling} presented in Section~\ref{sec:proofs} for the details. 
	With this proof, one can generalize Theorem~\ref{thm:shiftCoupling} to the context of Borel equivalence relations as follows.
	
	%According to~\cite{FeMo77}, 
The following definitions are borrowed from~\cite{FeMo77}.
	An equivalence relation $R$ on a Polish space $X$ is a \defstyle{{countable Borel equivalence relation}} if when considered as a subset of $X\times X$, it is a Borel subset and each equivalence class is countable. The \defstyle{$R$-invariant sigma-field} on $X$ consists of Borel subsets of $X$ which are formed by unions of $R$-equivalence classes. In the following result, a Borel automorphism $F:X\rightarrow X$ is called {\defstyle{$R$-stabilizing}} if $F(x)Rx$ for each $x\in X$.

	\begin{customthm}{1'}
		Let $R$ be a countable Borel equivalence relation on $X$ and $Y_1$ and $Y_2$ be random elements in $X$.
		Then there exists a random $R$-stabilizing Borel automorphism $F$ such that $F(Y_1)$ has the same distribution as $Y_2$ if and only if the distributions of $Y_1$ and $Y_2$ agree on the $R$-invariant sigma-field.
	\end{customthm}
	
	In fact, in the converse, $F$ can be chosen to be supported on countably many automorphisms.
	As mentioned above, the proof of this theorem is similar to one of the proofs given for Theorem~\ref{thm:shiftCoupling} and is skipped here.
	
	A Borel equivalence relation $R$ is \defstyle{smooth} if the quotient space $X/R$ with the induced Borel structure is a standard Borel space. Therefore, Proposition~\ref{prop:nonstandard} just claims that the equivalence relation on $\mathcal G_*$ induced by $\pi$ (see the proof of Theorem~\ref{thm:shiftCoupling}) is not smooth, which is implied by Corollary~1.3 of~\cite{kechris}. A direct proof is also presented in Section~\ref{sec:proofs}.
}

%\section{The Unimodular Case}
\section{The Unimodular Case and Balancing Transport Kernels}
\label{sec:unimodularCase}

In this section, some applications of Theorem~\ref{thm:shiftCoupling} are presented for the case of unimodular networks. The main results are theorems~\ref{thm:unimodularUnique} and~\ref{thm:balancing} whose
%First, we state the following theorems which are the main results of this section as applications of Theorem~\ref{thm:shiftCoupling}. 
proofs are postponed to the end of the section after presenting some minor results. 

\begin{theorem}[Uniqueness]
	\label{thm:unimodularUnique}
	The distribution of a unimodular network $[\bs G, \bs o]$ is uniquely determined by its restriction to the invariant sigma-field (or equivalently, by the distribution of the non-rooted network $[\bs G]$).  In other words, if two unimodular networks are strongly unrooted-equivalent, %have the same non-rooted networks in the strong sense, 
	then they are identically distributed.
	
%	\sout{A unimodular probability measure on $\mathcal G_*$ is uniquely determined by its restriction to the invariant sigma-field (or its projection on $(\mathcal G,J)$). In other words, if two unimodular networks have the same non-rooted networks in the strong sense, then they are identically distributed.}
\end{theorem}

Theorem~\ref{thm:unimodularUnique} is a precise formulation of a comment in~\cite{processes} saying that `intuitively, the distribution of the root is forced given the distribution of the unrooted network'.
%The result of Theorem~\ref{thm:unimodularUnique} is quoted in~\cite{processes} in a heuristic sentence. 
Note also that if we replace `strongly' with `weakly' in this theorem, the claim no longer holds. This case will be considered in Lemma~\ref{lemma:unimodularNonErgodic} and Proposition~\ref{prop:unimodularErgodic} below.

\begin{theorem}[Balancing Transport Kernel]
	\label{thm:balancing}
	Let $[\bs G, \bs o]$ be a unimodular network and $w_i:\mathcal G_*\rightarrow\mathbb R^{\geq 0}$ be measurable functions for $i=1,2$. Assume $\omidCond{w_1}{I}<\infty$ a.s. Then, the following are equivalent.
	\begin{enumerate}[(i)]
		\item 	\label{thm:balancing:1}
		There is an {invariant transport kernel} that almost surely balances between the functions $w_1[\bs G, \cdot]$ and $w_2[\bs G, \cdot]$ on the vertices; i.e. a measurable function $T:\mathcal G_{**}\rightarrow\mathbb R^{\geq 0}$ such that almost surely, $T^+_{\bs G}(v)=w_1[\bs G,v]$ and $T^-_{\bs G}(v)=w_2[\bs G,v]$ for all $v\in V(\bs G)$.
%		\begin{eqnarray*}
%			T^+_{\bs G}(v)&=&w_1[\bs G,v],\\
%			T^-_{\bs G}(v)&=&w_2[\bs G,v].
%		\end{eqnarray*}
		\item 	\label{thm:balancing:2}
		One has
		\begin{equation}
		\label{eq:thm:balancing0}
		\omidCond{w_1[\bs G, \bs o]}I = \omidCond{w_2[\bs G, \bs o]}{I}.
		\end{equation}
	\end{enumerate}
	%	\ali{(delete this? or another prop?)}
	%	Moreover, biasing by $w_1$ and $w_2$ give random rooted networks with the same non-rooted networks in the strong sense if and only if the above conditions hold for $w_1$ and $\lambda w_2$ for some constant $\lambda$.
\end{theorem}

Theorem~\ref{thm:balancing} is analogous to similar results for stationary point processes and random measures (\cite{HoPe05} and~\cite{ThLa09}). This analogy will be explained in Section~\ref{sec:analogy}. 
%Also, the proof of this theorem is not constructive in terms of a realization of the network. A construction will be presented in Section~\ref{sec:construction}.

\begin{remark}
	A result similar to Proposition~\ref{prop:conditioning} holds with the assumptions of Theorem~\ref{thm:balancing}.
	For $i=1,2$, consider biasing the distribution of a (not necessarily unimodular) random rooted network $[\bs G, \bs o]$ by a function $w_i$. Then, the resulting random rooted networks are always weakly unrooted-equivalent, % always have the same non-rooted networks in the weak sense, 
	but this holds strongly % in the strong sense 
	if and only if the ratio $\omidCond{w_1[\bs G, \bs o]}I / \omidCond{w_2[\bs G, \bs o]}{I}$ is essentially constant. However, the existence of a balancing transport kernel as in Theorem~\ref{thm:balancing} is only proved for the unimodular case.
\end{remark}

Before proving the above theorems, we present some other minor results in the unimodular case.

\begin{lemma}
	\label{lemma:unimodularBias}
	Let $[\bs G_1, \bs o_1]$ be a unimodular network and $[\bs G_2,\bs o_2]$ be an arbitrary random rooted network.
	\begin{enumerate}[(i)]
		\item \label{lemma:unimodularBias:1}
		If $[\bs G_2, \bs o_2]$ is a root-change of $[\bs G_1, \bs o_1]$ by kernel $T$, then it can also be obtained by biasing $[\bs G_1, \bs o_1]$ by the function $[G,o]\mapsto T_{G}^-(o)$.
		\item \label{lemma:unimodularBias:2}
		If $[\bs G_2,\bs o_2]$ is weakly unrooted-equivalent to $[\bs G_1, \bs o_1]$, %has the same non-rooted network as $[\bs G_1,\bs o_1]$ in the weak sense, 
		then the distribution of $[\bs G_2,\bs o_2]$ is obtained from that of $[\bs G_1, \bs o_1]$ by only a biasing (i.e. is absolutely continuous w.r.t. the distribution of $[\bs G_1, \bs o_1]$).
	\end{enumerate}
\end{lemma}

%\mar{delete this?} The idea behind part (i) of the above proposition is that by unimodularity, all vertices are equally likely to be the root. Therefore, applying a root-change according to law $T$ is equivalent to biasing by the incoming mass $\sum_{v\in V(G)}T_G(v,o)$ to the root.

\begin{proof}%[Proof of Lemma~\ref{lemma:unimodularBias}]

	\eqref{lemma:unimodularBias:1}. 
	%Let $T$ be a Markovian transport that transports the distribution of $[\bs G_1,\bs o_1]$ to that of $[\bs G_2, \bs o_2]$. 
	Given an event $A\subseteq\mathcal G_*$, define %an invariant transport $S$ by
	$
	g[G,o,v]:=T_G(o,v)\identity{A}[G,v].
	$
	By~\eqref{eq:rootChange} and unimodularity of $[\bs G_1, \bs o_1]$, one gets
	\begin{eqnarray*}
		\myprob{[\bs G_2, \bs o_2]\in A} &=& \omid{g_{\bs G_1}^+(\bs o_1)} = \omid{g_{\bs G_1}^-(\bs o_1)}\\
		&=& \omid{\identity{A}[\bs G_1, \bs o_1] \sum_{v\in V(\bs G_1)}T_{\bs G_1}(v,\bs o_1)}\\
		&=& \omid{\identity{A}[\bs G_1, \bs o_1] T_{\bs G_1}^-(\bs o_1)}.				
	\end{eqnarray*}
	By letting $A:=\mathcal G_*$, one gets $\omid{T_{\bs G_1}^-(\bs o_1)}=1$. Therefore, the above equation means that the distribution of $[\bs G_2,\bs o_2]$ is obtained from that of $[\bs G_1,\bs o_1]$ by the desired biasing (see Definition~\ref{def:bias}). % biasing by $\sum_{v\in V(G)}T_G(v,o)$.
	
	\eqref{lemma:unimodularBias:2}. By part \eqref{lemma:unimodularBias:1} and Definition~\ref{def:weakStrong}, % the definition of the weak sense, 
	$[\bs G_2,\bs o_2]$ is obtained by biasing $[\bs G_1, \bs o_1]$ by a composition of two biasings, say by functions $w_1$ and $w_2$. It is easy to show that the result is just biasing by $w_1w_2$ and the claim is proved.
	% If one biases a probability measure by a function $w_1$ and then biases the result by a function $w_2$, it is easy to show that the result is just biasing by $w_1w_2$. Therefore, the claim is implied by part~\eqref{lemma:unimodularBias:1}.
\end{proof}

\begin{lemma}
	\label{lemma:unimodularBiasUni}
	%\mar{ater: Change to this?: i.e. almost surely, $w[\bs G, \cdot]$ is constant on the vertices}
	Let $[\bs G, \bs o]$ be a unimodular network and $w:\mathcal G_*\rightarrow\mathbb R^{\geq 0}$ be a measurable function. Then, biasing $[\bs G, \bs o]$ by $w$ gives a unimodular probability measure if and only if $w$ doesn't depend on the root a.s. %\sout{; i.e. almost surely, for all $v\in V(\bs G)$, one has $w[\bs G, v]=w[\bs G, \bs o]$.} % there is a measurable function $w':\mathcal G\rightarrow\mathbb R^{\geq 0}$ s.th. $w[\bs G, \bs o]=w'[\bs G]$ a.s.
\end{lemma}
\begin{proof}%{Proof of Lemma~\ref{lemma:unimodularBiasUni}}
	Let $[\bs G', \bs o']$ be a random rooted network whose distribution is obtained by biasing that of $[\bs G, \bs o]$ by $w$. Let $m:=\omid{w[\bs G, \bs o]}$. For a measurable function $g:\mathcal G_{**}\rightarrow\mathbb R^{\geq 0}$, one has
	\begin{eqnarray*}
		\omid{g_{\bs G'}^+(\bs o')} = \frac 1m \omid{w[\bs G, \bs o] g_{\bs G}^+(\bs o)} = \frac 1m \omid{h_{\bs G}^+(\bs o)},
	\end{eqnarray*}
	where $h_G(o,v):=w[G,o]g_G(o,v)$. By unimodularity of $[\bs G, \bs o]$, one obtains
	\[
	\omid{g_{\bs G'}^+(\bs o')} = \frac 1m \omid{h_{\bs G}^-(\bs o)} = \frac 1m \omid{\sum_{v\in V(\bs G)}w[\bs G, v]g_{\bs G}(v,\bs o)}.
	\]
	On the other hand,
	\[
	\omid{g_{\bs G'}^-(\bs o')} = \frac 1m \omid{w[\bs G, \bs o]g_{\bs G}^-(\bs o)} = \frac 1m \omid{\sum_{v\in V(\bs G)}w[\bs G, \bs o]g_{\bs G}(v,\bs o)}.
	\]
	Therefore, $[\bs G',\bs o']$ is unimodular if and only if 
	\begin{equation}
	\label{eq:lemma:unimodularBiasUni}
	\forall g: \omid{\sum_{v\in V(\bs G)}\left(w[\bs G,v]-w[\bs G,\bs o]\right)g_{\bs G}(v,\bs o)}=0.
	\end{equation}
	
	First, suppose that almost surely, $w[\bs G, v]=w[\bs G, \bs o]$ for all $v\in V(\bs G)$. This implies that~\eqref{eq:lemma:unimodularBiasUni} holds and thus, $[\bs G',\bs o']$ is unimodular.
	Conversely, assume $[\bs G',\bs o']$ is unimodular. By substitute $g_G(v,o)$ with the positive and negative parts of $w[G,v]-w[G,o]$ respectively, \eqref{eq:lemma:unimodularBiasUni} gives that almost surely, $w[\bs G,v]=w[\bs G,\bs o]$ for all $v\in V(\bs G)$. 
	So, the claim is proved.
\end{proof}

\begin{lemma}
	\label{lemma:unimodularNonErgodic}
	Let $[\bs G_1, \bs o_1]$ and $[\bs G_2, \bs o_2]$ be random rooted networks that are weakly unrooted-equivalent. % have the same non-rooted networks in the weak sense. 
	If both are unimodular, then the distribution of $[\bs G_2, \bs o_2]$ can be obtained by biasing that of $[\bs G_1, \bs o_1]$ by a function that doesn't depend on the root and is almost surely positive.
\end{lemma}

\begin{proof}%[Proof of Lemma~\ref{lemma:unimodularNonErgodic}]
	Since both are unimodular, by lemmas~\ref{lemma:unimodularBias} and~\ref{lemma:unimodularBiasUni}, $[\bs G_2, \bs o_2]$ is obtained from $[\bs G_1, \bs o_1]$ by biasing by a function $w:\mathcal G_*\rightarrow\mathbb R^{\geq 0}$ that doesn't depend on the root. As a result, the distribution of $[\bs G_2,\bs o_2]$ is absolutely continuous w.r.t. that of $[\bs G_1, \bs o_1]$. 
	%Since $[\bs G_2, \bs o_2]$ is also unimodular, 
	The same holds by swapping the roles of $[\bs G_1, \bs o_1]$ and $[\bs G_2, \bs o_2]$. Therefore, the Radon-Nikodym derivative, which is proportional to $w[\bs G_1,\bs o_1]$, is positive almost surely. This proves the claim.
	%	By Proposition~\ref{lemma:unimodularBias} and Lemma~\ref{lemma:unimodularBiasUni}, 
	%	each of $[\bs G_i, \bs o_i]$ is obtained from the other by biasing by a function $w_i:\mathcal G_*\rightarrow\mathbb R^{\geq 0}$ that doesn't depend on the root. 
	%	\mar{ater: remove details. use abs. cont. ness} Definition~\ref{def:bias} implies that $\myprob{w_1[\bs G_2,\bs o_2]=0}=0$. Now, if $\myprob{w_1[\bs G_1,\bs o_1]=0}>0$, then the distribution of $[\bs G_1,\bs o_1]$ is not absolutely continuous w.r.t that of $[\bs G_2,\bs o_2]$. In this case, $[\bs G_1,\bs o_1]$ cannot be obtained from $[\bs G_2, \bs o_2]$ by a biasing, which is a contradiction. Therefore, $w_1[\bs G_1,\bs o_1]>0$ a.s. and the proof is complete.
\end{proof}

\begin{proposition}
	\label{prop:unimodularErgodic}
	Let $[\bs G_1, \bs o_1]$ and $[\bs G_2, \bs o_2]$ be random rooted networks which are weakly unrooted-equivalent. 
		If at least one of them is an extremal unimodular network,
		then they are also strongly unrooted-equivalent.
	%If $[\bs G_1, \bs o_1]$ is an extremal unimodular network,
%	then \ali{the weak and strong notions in Definition~\ref{def:weakStrong} are equivalent.}
%	%the weak sense in Definition~\ref{def:weakStrong} is equivalent to the strong sense.
\end{proposition}

%\mar{delete this?} The idea behind this proposition is that by ergodicity, every bias becomes constant after taking conditional expectation w.r.t. the invariant sigma-field.

%\begin{proposition}\mar{ater: Move to prev subsection for general networks.}
%	\ali{ater: Delete this because of prev prop} Let $[\bs G, \bs o]$ be a unimodular network. Consider biasing by two functions $w_1$ and $w_2$. These are obtained from each other by a root-change if and only if $\omidCond{w_1}{I}=\omidCond{w_2}{I}$.
%\end{proposition}
%\begin{proof} ater:
%	
%	Proof 1: Theorem~\ref{thm:shiftCoupling}.
%	
%	Proof 2: stable transport!
%\end{proof}

%\begin{proposition}
%	(Delete?)
%	This sense is equivalent to the strong sense: bias and root-change, with the condition that $\omidCond{bias}{I}=1$.
%\end{proposition}

\begin{proof}%[Proof of Proposition~\ref{prop:unimodularErgodic}]
	%Assume $[\bs G_2, \bs o_2]$ is weakly unrooted-equivalent to $[\bs G_1, \bs o_1]$. % has the same non-rooted network as $[\bs G_1,\bs o_1]$ in the weak sense. 
	Assume $[\bs G_1,\bs o_1]$ is an extremal unimodular network.
	Lemma~\ref{lemma:unimodularBias} implies that $[\bs G_2,\bs o_2]$ can be obtained from $[\bs G_1,\bs o_1]$ by biasing by a measurable function $w:\mathcal G_*\rightarrow\mathbb R^{\geq 0}$. Lemma~\ref{lemma:biasG} implies that the distribution of $[\bs G_2]$ is obtained from that of $[\bs G_1]$ by biasing by $\omidCond{w[\bs G_1,\bs o_1]}{I}$. On the other hand, since $[\bs G_1, \bs o_1]$ is extremal, the $I$-measurable function $\omidCond{w[\bs G_1,\bs o_1]}{I}$ is {essentially constant}. It follows that $[\bs G_2]$ and $[\bs G_1]$ have the same distribution, which shows that $[\bs G_2,\bs o_2]$ is strongly unrooted-equivalent to $[\bs G_1, \bs o_1]$. % has the same non-rooted network as $[\bs G_1,\bs o_1]$ in the strong sense.
\end{proof}

We are now ready to prove the main theorems of this section.

\begin{proof}[Proof of Theorem~\ref{thm:unimodularUnique}]

	Let $[\bs G_1, \bs o_1]$ and $[\bs G_2, \bs o_2]$  be unimodular networks such that their distributions agree on the invariant sigma-field. Therefore, they are strongly unrooted-equivalent % have the same non-rooted networks in the strong sense
    (Condition~\ref{conditionF}). The same holds weakly by Theorem~\ref{thm:implications}. Thus, Lemma~\ref{lemma:unimodularNonErgodic} implies that $[\bs G_2, \bs o_2]$ is obtained by biasing the distribution of $[\bs G_1,\bs o_1]$ by a measurable function $w:\mathcal G_*\rightarrow\mathbb R^{\geq 0}$ that doesn't depend on the root. It is enough to show that $w[\bs G_1, \bs o_1]$ is essentially constant. 
	
	By Lemma~\ref{lemma:biasG}, the distribution of $[\bs G_2]$ is obtained from that of $[\bs G_1]$  by biasing  by $\omidCond{w[\bs G_1, \bs o_1]}{I}$. Since the latter distributions are equal by assumption, it follows that $\omidCond{w[\bs G_1, \bs o_1]}{I}$ is essentially constant. On the other hand, since $w$ doesn't depend on the root, it is $I$-measurable and thus, $\omidCond{w[\bs G_1, \bs o_1]}{I}=w[\bs G_1, \bs o_1]$ a.s. It follows that the bias function $w[\bs G_1, \bs o_1]$ is essentially constant. Therefore, $[\bs G_2, \bs o_2]$ and $[\bs G_1, \bs o_1]$ are identically distributed.
\end{proof}

\begin{proof}[Proof of Theorem~\ref{thm:balancing}]
	\eqref{thm:balancing:1}$\Rightarrow$ \eqref{thm:balancing:2}. Let $A\in I$ be an invariant event. Define %a transport $g$ by
	$
	g[G,v,z]:=T_G(v,z)\identity{A}[G,v] = T_G(v,z)\identity{A}[G,z].
	$
	By the assumption, one gets that almost surely, $g_{\bs G}^+(\bs o)=w_1[\bs G,\bs o] \identity{A}[\bs G,\bs o]$ and $g_{\bs G}^-(\bs o)=w_2[\bs G,\bs o] \identity{A}[\bs G,\bs o]$. By unimodularity, one gets
	\[
	\omid{w_1[\bs G, \bs o]\identity{A}[\bs G, \bs o]}= \omid{g_{\bs G}^+(\bs o)} = \omid{g_{\bs G}^-(\bs o)} = \omid{w_2[\bs G, \bs o]\identity{A}[\bs G, \bs o]}.
	\]
	By considering this for all $A\in I$, one obtains~\eqref{eq:thm:balancing0}.% that $\omidCond{w_1[\bs G, \bs o]}I = \omidCond{w_2[\bs G, \bs o]}{I}$.
	
	\eqref{thm:balancing:2} $\Rightarrow$ \eqref{thm:balancing:1}. For $i=1,2$, let $[\bs G_i, \bs o_i]$ be a random rooted network obtained by biasing $[\bs G, \bs o]$ by $w_i$. Assumption~\eqref{eq:thm:balancing0} and Lemma~\ref{lemma:biasG} imply that $[\bs G_1]$ has the same distribution as $[\bs G_2]$. %$\pi_* w_1=\pi_* w_2$. \mar{I use the notation of measures here} 
	In other words, $[\bs G_1, \bs o_1]$ and $[\bs G_2, \bs o_2]$ are strongly unrooted-equivalent % have the same non-rooted networks in the strong sense 
	(Condition~\ref{conditionF}). By Theorem~\ref{thm:shiftCoupling},  $[\bs G_2, \bs o_2]$ can be obtained from $[\bs G_1, \bs o_1]$ by a root-change; i.e. there is a measurable function $t:\mathcal G_{**}\rightarrow\mathbb R^{\geq 0}$ such that $t_{\bs G_1}^+(\bs o_1)=1$ a.s. and
	\[
	\omid{h[\bs G_2, \bs o_2]} = \omid{\sum_{v\in V({\bs G_1})} t_{\bs G_1}(\bs o_1, v) h[\bs G_1, v]}
	\] 
	for any measurable function $h:\mathcal G_*\rightarrow\mathbb R^{\geq 0}$. Fix $h$ arbitrarily. By the definition of $[\bs G_i, \bs o_i]$, one obtains
	\begin{equation}
	\label{eq:thm:balancing1}
	\omid{w_2[\bs G, \bs o] h[\bs G, \bs o]} = \omid{\sum_{v\in V({\bs G})} t_{\bs G}(\bs o, v) w_1[\bs G, \bs o] h[\bs G, v]},	
	\end{equation}
	
	where the equation $\omid{w_1[\bs G, \bs o]}= \omid{w_2[\bs G, \bs o]}$ is used (which holds by~\eqref{eq:thm:balancing0}) to cancel out the denominators. Define an invariant transport kernel $T$ by
	$
	T_G(o,v):=t_G(o,v)w_1(G,o).
	$
	By unimodularity, one has
	\begin{eqnarray*}
		\omid{\sum_{v\in V({\bs G})} t_{\bs G}(\bs o, v) w_1[\bs G, \bs o] h[\bs G, v]} &=& \omid{\sum_{v\in V({\bs G})} T_{\bs G}(\bs o, v) h[\bs G, v]}\\
		&=& \omid{\sum_{v\in V({\bs G})} T_{\bs G}(v, \bs o)  h[\bs G, \bs o]}\\
		&=& \omid{T_{\bs G}^-(\bs o)h[\bs G, \bs o]}.
	\end{eqnarray*}
	So, \eqref{eq:thm:balancing1} implies that $\omid{w_2[\bs G, \bs o]h[\bs G, \bs o]} = \omid{T_{\bs G}^-(\bs o)h[\bs G, \bs o]}$. Since this holds for any $h$, it follows that $T_{\bs G}^-(\bs o)=w_2[\bs G, \bs o]$ a.s. 
	On the other hand, by $t_{\bs G_1}^+(\bs o_1)=1$ a.s., one gets that $T_{\bs G}^+(\bs o)=w_1[\bs G, \bs o]$ a.s. Therefore, Lemma~\ref{lemma:happensAtRoot} implies that the same holds for all vertices; i.e. almost surely, for all $v\in V(\bs G)$, one has $T_{\bs G}^+(v)=w_1[\bs G, v]$ and $T_{\bs G}^-(v)=w_2[\bs G, v]$. So the theorem is proved.
\end{proof}

\section{Network Extension and Unimodularization}
\label{sec:extAndTrans}
%\mar{\ali{Bring the proofs back from Sec\ref{sec:proofs}?}}

In this section, the method of \textit{network extension} is introduced and {the shift-coupling theorem is} applied to it. This method unifies some of the examples in the literature to construct unimodular networks. First, in Subsection~\ref{subsec:unimodularization} we study  \textit{unimodularizations} of a random non-rooted network in general. Then, network extension is studied in Subsection~\ref{subsec:extension}.

%In this section, the concept of \textit{network extension} is studied which unifies some of the examples in the literature to construct new unimodular networks from other unimodular ones. First, in Subsection~\ref{subsec:unimodularization} we study  \textit{unimodularizations} of a random non-rooted network in general. Then, network extension is studied in Subsection~\ref{subsec:extension}. %The proofs of the results are postponed to Section~\ref{sec:proofs}.

%\ali{ater: Bring some parts of the following intro to here? say proofs are postponed.}
%In some examples in the literature, unimodular networks are constructed given some non-unimodular network or even a non-rooted network. For example, given a unimodular plane 

%\subsection{Network Transformation}
\subsection{Unimodularizations of a Non-Rooted Network}
\label{subsec:unimodularization}
%\mar{\ali{Change title to "Unimodular..."?}}

%%\ali{\sout{In \mar{delete this?} some examples in the literature, unimodular networks are constructed from random non-rooted networks. For example, if $[\bs G, \bs o]$ is a unimodular network and $g:\mathcal G\rightarrow\mathcal G$ is a $J$-measurable map,
%%%\ali{\sout{; e.g. $g[G]$ is the dual of a graph $[G]$ which is equipped with a proper embedding in the plane}} 
%%then $g([\bs G])$ is a random non-rooted network (see examples~\ref{ex:dual1} and~\ref{ex:subnetwork}).}}
%%% \ali{\sout{(note that no root is naturally selected for the dual of a rooted plane graph)}}. 
%Here, we define what it means to \textit{unimodularize} a random non-rooted network and present some basic results based on the previous theorems. %A special case of this concept is the subject of Subsection~\ref{subsec:extension}.

%ater: update this for an intro (or after the definition):
%\\
%As another example, one can let $\mu$ be the distribution of $g[\bs G]$, where $[\bs G, \bs o]$ is a random rooted network and $g:\mathcal G\rightarrow\mathcal G$ is a $J$-measurable map. See the following examples.

\begin{definition}
	\label{def:unimodularization}
	Let $\mu_0$ be a probability measure on $(\mathcal G,J)$ (or similarly, on $(\mathcal G_*,I)$).  We say that a random rooted network $[\bs G',\bs o']$ is  \defstyle{unrooted-equivalent to $\mu_0$}, % \defstyle{has the same non-rooted network as $\mu_0$},
	\begin{itemize}
		\item \defstyle{strongly} %\defstyle{in the strong sense} 
		if the distribution of $[\bs G']$ is identical to $\mu_0$.
		\item \defstyle{weakly} %\defstyle{in the weak sense} 
		if the distribution of $[\bs G']$ and $\mu_0$ are mutually absolutely continuous. % w.r.t. each other. \ali{\sout{; i.e. each one is obtained from the other by a biasing.}}
	\end{itemize}
	If in addition $[\bs G', \bs o']$ is unimodular, we say %it \defstyle{makes $\mu$ unimodular}.
	it is a \defstyle{(weak or strong) unimodularization of $\mu_0$}  and $\mu_0$ \defstyle{can be unimodularized}.
	
	%	Let $g:\mathcal G\rightarrow\mathcal G$ be a measurable map (w.r.t. the invariant sigma-field of $\mathcal G_*$) and $[\bs G_1, \bs o_1]$ be a given random rooted network. We say that a random network $[\bs G_2,\bs o_2]$ has \textbf{the same non-rooted network as $g[\bs G_1]$}, if
	%	\begin{description}
	%		\item [Strong sense:] The distribution of $[\bs G_2]$ and $g[\bs G_1]$ are identical.
	%		\item [Weak sense:] The distributions of $[\bs G_2]$ and $g[\bs G_1]$ are obtained from each other by a biasing with a measurable function $w:\mathcal G\rightarrow\mathcal G$; i.e. they are mutually absolutely continuous w.r.t. each other.
	%	\end{description}
	%	If in addition $[\bs G_2, \bs o_2]$ is unimodular, we say it \textbf{makes $g[\bs G_1]$ unimodular}.
	
	%\ali{ater: same works if $[\bs G_1]$ is a non-rooted random network.}
\end{definition}

%Note that no root is needed for the network $g[\bs G_1]$ to be defined. Therefore, $w$ depends only on the non-rooted network $[\bs G_1]$.

%\ali{ater: say why weak sense is not non-sense!}

Heuristically, unimodularization means to choose a random root for a given random non-rooted network 
to obtain a unimodular network.
%in such a way that a unimodular random rooted network is obtained.

To see why the weak sense is ever defined here, it will turn out that some well known examples in the literature are weak unimodularizations %in the weak sense 
(see examples~\ref{ex:replacement} and~\ref{ex:dual2} of Subsection~\ref{subsec:extension}).
Moreover, the notions of weak and strong here are analogous to the previous notions as described in the following lemma.
%Moreover, the notions of weak and strong in Definition~\ref{def:unimodularization} are analogous to those in Definition~\ref{def:weakStrong} as expressed in the following lemma.

\begin{lemma}
	In the case $\mu_0$ is the distribution of $[\bs G]$, where $[\bs G, \bs o]$ is a random rooted network, Definition~\ref{def:unimodularization} is reduced to Definition~\ref{def:weakStrong}.
\end{lemma}

This lemma is straightforward and we skip its proof.

%The following is an immediate corollary of Theorem~\ref{thm:unimodularUnique}.
\begin{proposition}
	\label{prop:unimodulalrizationUnique}
	Under the assumptions of Definition~\ref{def:unimodularization},
	%	there is at most one unimodular network $[\bs G_2, \bs o_2]$ (up do equality in distribution) that has the same non-rooted network as $g [\bs G_1]$ in the strong sense.
	if $\mu_0$ can be unimodularized (either weakly or strongly), %(in either sense), 
	then there is a unique strong unimodularization of $\mu_0$. %way to make it unimodular 
	%in the strong sense.
\end{proposition}

\begin{proof}
	Suppose $[\bs G', \bs o']$ is a weak unimodularization of $\mu_0$. % in the weak sense. 
	Let $w':\mathcal G\rightarrow\mathbb R^{\geq 0}$ be the Radon-Nikodym derivative of $\mu_0$ w.r.t. the distribution of $[\bs G']$. Let $[\bs G, \bs o]$ be the random rooted network obtained by biasing $[\bs G', \bs o']$ by $w:=w'\circ \pi$. %$w:[G,o]\mapsto w'[G]$. 
	Lemma~\ref{lemma:unimodularBiasUni} implies that $[\bs G, \bs o]$ is unimodular. Lemma~\ref{lemma:biasG} implies that the distribution of $[\bs G]$ is equal to $\mu_0$, which means that $[\bs G, \bs o]$ is a strong unimodularization of $\mu_0$. % in the strong sense. 
	Now, Theorem~\ref{thm:unimodularUnique} implies that this is the unique strong unimodularization of $\mu_0$. % in the strong sense.
\end{proof}

Note that some probability measures on $(\mathcal G,J)$ (i.e. some random non-rooted networks) cannot be unimodularized; e.g. a deterministic semi-infinite path. % ater: example: Let $[\bs G_1, \bs o_1]$ have one special vertex...

\begin{example}[Planar Dual I]
	\label{ex:dual1}
%	For a plane graph $G$, let $G^*$ be its dual graph. 
%	Let $[\bs G, \bs o]$ be a unimodular plane graph (see Example~9.6 of~\cite{processes} for the precise definition using networks).
	Let $[\bs G, \bs o]$ be a unimodular plane graph
	(see Example~9.6 of~\cite{processes} for how to regard a plane graph as a network and define its dual).
	%	It is shown in Example~9.6 of~\cite{processes} that one can consider a network with underlying graph $G$ such that the marks capture the structure of the planar embedding. 
	%So, it makes sense to assume $[\bs G, \bs o]$ is a unimodular plane graph. 
	%
	%Although no special vertex of the dual graph is selected as a root, 
	With no need to select a vertex of the dual graph as a root,
	$[\bs G^*]$ makes sense as a random non-rooted network. In~\cite{processes}, a unimodular network is constructed based on the dual graph, which in our language, is a weak unimodularization of $[\bs G^*]$. % in the weak sense. 
	This construction will be discussed in Example~\ref{ex:dual2}.
	%	
	%	ater: Change this
	%	
	%	For a planar graph $G$, let $G^*$ be its dual graph. If we don't fix roots for $G$ and $G^*$, they can be considered as non-rooted networks (\ali{ref}). Let $[\bs G, \bs o]$ be a unimodular plane graph. The random network constructed by~\eqref{eq:planarDual} has the same non-rooted network as $\bs G^*$ in the weak sense. This holds in the strong sense if and only if $\omidCond{d(\bs o)}{I}$ is constant a.s. \ali{proof?} In general, the probability measure
	%	\begin{equation}
	%	\label{eq:planarDualStrong}
	%	A\mapsto \omid{\frac 1{\omidCond{d(\bs o)}{I}}\sum_{f \sim \bs o} \identity A[\bs G^*,f]},
	%	\end{equation}
	%	makes $\bs G^*$ unimodular in the strong sense.
\end{example}

\begin{example} [Subnetwork]
	\label{ex:subnetwork}
	%\mar{\ali{ater: ref to prop~\ref{prop:conditioning}}}
	Let $[\bs G, \bs o]$ be a unimodular network and $S$ be a covariant subnetwork (Definition~\ref{def:covariantSubset}). %By an abuse of notation, we use the same symbol for the subnetwork induced by $S_{\bs G}$ (i.e. the restriction of $\bs G$ to $S_{\bs G}$). 
	Assume $S_{\bs G}$ is nonempty and connected a.s. Therefore, $[S_{\bs G}]$ is a random non-rooted network; i.e. a random element in $(\mathcal G,J)$ (note that $S_{\bs G}$ doesn't need to contain $\bs o$ and no root is chosen for $S_{\bs G}$). % as a function of $[\bs G]$. 
	Condition $[S_{\bs G}, \bs o]$ on the event $\bs o\in S_{\bs G}$; i.e. consider the probability measure
	$
	\mathcal P'[A]:= \probCond{[S_{\bs G}, \bs o]\in A}{\bs o\in S_{\bs G}}
	$
	on $\mathcal G_*$. Considering $\mathcal P'$ as a random rooted network, we claim that it is a weak unimodularization of (the distribution of) $[S_{\bs G}]$. % in the weak sense. 
	Let $\mu:=\pi_*\mathcal P'$ be the distribution of the non-rooted network under $\mathcal P'$.
	By taking conditional expectation w.r.t. $I$, one obtains for any measurable function $g:(\mathcal G,J)\rightarrow\mathbb R^{\geq 0}$ that
	$
	\omid{g[S_{\bs G}]\identity{\{\bs o\in S_{\bs G}\}}} = \omid{g[S_{\bs G}] \probCond{\bs o\in S_{\bs G}}{I}}.
	$
	Therefore,
	\[
	\omidCond{g[S_{\bs G}]}{\bs o\in S_{\bs G}} = \omid{g[S_{\bs G}] \probCond{\bs o\in S_{\bs G}}{I}/\myprob{\bs o\in S_{\bs G}}}.
	\]
	In other words, $\mu$ is just biasing the distribution of $[S_{\bs G}]$ by $\probCond{\bs o\in S_{\bs G}}{I}/\myprob{\bs o\in S_{\bs G}}$, where the latter is considered as a function on $(\mathcal G, J)$. Similar to Lemma~\ref{lemma:happensAtRoot}, one can deduce from $S_{\bs G}\neq \emptyset$ a.s. that $\probCond{\bs o\in S_{\bs G}}{I}>0$ a.s., thus, $\mu$ and the distribution of $[S_{\bs G}]$ are mutually absolutely continuous. On the other hand, it is easy to use~\eqref{eq:unimodular} directly to see that $\mathcal P'$ is unimodular. Thus, $\mathcal P'$ is a weak unimodularization of (the distribution of) $[S_{\bs G}]$. % in the weak sense. % and the claim is proved.}
	
	To obtain a strong unimodularization of $[S_{\bs G}]$, % in the strong sense, 
	one can bias the distribution of $[\bs G, \bs o]$ by $\identity{\{\bs o\in S_{\bs G}\}}/\probCond{\bs o\in S_{\bs G}}{I}$ and then consider the subnetwork rooted at $\bs o$ (see Lemma~\ref{lemma:biasG}). Here, the denominator $\probCond{\bs o\in S_{\bs G}}{I}$ can be regarded as the \textit{sample intensity} of $S_{\bs G}$, which is a random variable and a function of $[\bs G]$.
	
%	By considering $[S_{\bs G}, \bs o]$ after conditioning on the event $\bs o\in S_{\bs G}$, one obtains a random rooted network with the same non-rooted network as (the distribution of) $[S_{\bs G}]$ in the weak sense; i.e. by considering the probability measure
%	\[
%	\mathcal P'[A]:= \probCond{[S_{\bs G}, \bs o]\in A}{\bs o\in S_{\bs G}}.
%	\]
%	It is easy to see that $\mathcal P'$ is unimodular. In other words, this is a unimodularization of (the distribution of) $[S_{\bs G}]$ in the weak sense. %\textit{makes (the distribution of) $[S_{\bs G}]$ unimodular}. 
%	Note that $\mathcal P'$ is equivalent to biasing the distribution of $[\bs G, \bs o]$ by $\identity{\{\bs o\in S_{\bs G}\}}$ and then passing to the subnetwork $[S_{\bs G}, \bs o]$. To obtain a unimodularization of $[S_{\bs G}]$ in the strong sense, one can bias by $\identity{\{\bs o\in S_{\bs G}\}}/\probCond{\bs o\in S_{\bs G}}{I}$ instead (see Lemma~\ref{lemma:biasG}). Here, the denominator $\probCond{\bs o\in S_{\bs G}}{I}$ can be regarded as the \textit{sample intensity} of $S_{\bs G}$ (which is a random variable and a function of $[\bs G]$). Note that similar to Lemma~\ref{lemma:happensAtRoot}, one can deduce from $S_{\bs G}\neq \emptyset$ a.s. that $\probCond{\bs o\in S_{\bs G}}{I}>0$ a.s. %\ali{ater: Say that this bias has the same ... in the strong sense} 
%	Indeed, the unimodularization above in the strong sense is the unique unimodularization in which the distribution of the sample intensity is not changed.
%	
\end{example}

%\subsection{Unimodular Extension}
\subsection{Unimodularizations of a Network Extension}
\label{subsec:extension}

In some examples in the literature, given a unimodular network $[\bs G_0, \bs o_0]$, another (not necessarily unimodular) random rooted network $[\bs G, \bs o]$ is obtained by adding some vertices and edges to the original network, called an \textit{extension} here (Definition~\ref{def:extension}). Then, by biasing the probability measure and changing the root, another unimodular network is constructed. In this subsection, first a general method is presented that covers such examples %(to the best of the author's knowledge) 
and helps to construct new unimodular networks. Then, using the previous theorems, it is shown that this method gives all \textit{unimodularizations} of $[\bs G]$ in the sense given in Subsection~\ref{subsec:unimodularization}. 
A number of basic examples are provided as applications of the definitions and results, although the examples are not new. %\ali{\sout{All proofs in this subsection are moved to Section~\ref{sec:proofs} to help to stay on the main thread.}}

The method presented here needs that the original network can be reconstructed from the {extension}, as explained in the following definition. In applications, to ensure the reconstruction is possible, one may add extra marks to the newly added vertices and edges (e.g. see Example~\ref{ex:ZandZ^2}). Nevertheless, after a new unimodular network is successfully constructed using the method, one may forget the extra marks and unimodularity will be preserved.

%To apply the method presented in this subsection, we need that the original network can be reconstructed from the {extension}. This property is explained in the following definition using the notion of subnetworks. To ensure the reconstruction is possible, one may add extra marks to the newly added vertices (e.g. see Example~\ref{ex:ZandZ^2}). Nevertheless, after a new unimodular network is successfully constructed using the method, one may forget the extra marks because unimodularity will be preserved.

\begin{definition}
	\label{def:extension}
	Let $[\bs G_0, \bs o_0]$ be a unimodular network. An \defstyle{extension} of $[\bs G_0, \bs o_0]$ is a pair $([\bs G, \bs o], S)$, where $[\bs G, \bs o]$ is a (not necessarily unimodular) random rooted network and $S$ is a covariant subnetwork with the conditions that $\bs o\in S_{\bs G}$ a.s., $S_{\bs G}$ is connected a.s. and $[S_{\bs G}, \bs o]$ has the same distribution as $[\bs G_0, \bs o_0]$. It is called a \defstyle{proper extension} if
	\begin{equation}
	\label{eq:mtpOnS}
	\omid{\sum_{v\in S_{\bs G}} g[\bs G, \bs o, v]} = \omid{\sum_{v\in S_{\bs G}} g[\bs G, v, \bs o]}, \forall g
	\end{equation}
	where $g$ runs over all measurable functions $g:\mathcal G_{**}\rightarrow\mathbb R^{\geq 0}$. % (compare this with Definition~\ref{def:unimodular}).
	%	\begin{itemize}
	%		\item almost surely, $\bs o\in S_{\bs G}$ and $[S_{\bs G}, \bs o]$ has the same distribution as $[\bs G_0, \bs o_0]$.
	%		\item $[\bs G, \bs o]$ satisfies the following mass transport principle on the vertices of $S_{\bs G}$. 
	%			\begin{equation}
	%				\label{eq:mtpOnS}
	%				\omid{\sum_{v\in S_{\bs G}} g[\bs G, \bs o, v]} = \omid{\sum_{v\in S_{\bs G}} g[\bs G, v, \bs o]}, \forall g
	%			\end{equation}
	%			where $g$ runs over all measurable functions $g:\mathcal G_{**}\rightarrow\mathbb R^{\geq 0}$ (compare this with Definition~\ali{???(unimodularity)}).
	%	\end{itemize}
	Here, $[\bs G, \bs o]$ is allowed to have a larger mark space than $[\bs G_0, \bs o_0]$. 
	%	If $[\bs G',\bs o']$ is a unimodular network with the same non-rooted network as $[\bs G, \bs o]$ in the weak (resp. strong) sense, then $[\bs G', \bs o']$ is called %a \textbf{unimodularized extension} of $[\bs G_0, \bs o_0]$ or 
	%	a \textbf{unimodularization} of the extension $([\bs G, \bs o],S)$ in the weak (resp. strong) sense. If such a $[\bs G', \bs o']$ exists, one says that the extension \textbf{can be unimodularized}.
\end{definition}

Note that by Lemma~\ref{lemma:happensAtRoot}, $[\bs G, \bs o]$ is non-unimodular except when $S_{\bs G}=\bs G$  a.s. Note also that $[\bs G, \bs o]$ is not necessarily a function of $[\bs G_0, \bs o_0]$; i.e. the newly added vertices and edges might be random. Moreover, \eqref{eq:mtpOnS} is stronger than unimodularity of $[S_{\bs G}, \bs o]$ (compare it with~\eqref{eq:unimodular} for $[S_{\bs G}, \bs o]$). 
We are interested in proper extensions only, since the results in this section only hold in the proper case.
See the following simple examples %and Lemma~\ref{lemma:extension} 
for more clarification of~\eqref{eq:mtpOnS}.

\begin{example}
	\label{ex:extensionZ}
	Let $G_0$ be the usual deterministic graph of $\mathbb Z$. Let $G$ be the graph obtained by adding a new vertex $v_{2n}$ for any even number $2n\in \mathbb Z$ and connecting it to the vertex $2n$. For the networks that have a unique bi-infinite path, let $S$ be the subnetwork that represents that path. Now, $[S_{G}, 0]$ has the same distribution as $[G_0, 0]$, but~\eqref{eq:mtpOnS} does not hold (e.g. let $g(u,v)$ be zero except when $d(u)=2$ and $d(v)=3$). So, $([G, 0],S)$ is an improper extension of $[G_0, 0]$.
	
	%\begin{center} \ali{ater: figure} \end{center}
	
	However, if one chooses $\bs o$ uniformly at random in $\{0, 1\}$, then it can be seen that \eqref{eq:mtpOnS} holds and $[G, \bs o]$ is a proper extension of $[G_0, 0]$. Moreover, by %letting $G'=G$ and 
	choosing $\bs o'$ uniformly at random in $\{0, 1, v_0\}$, $[G, \bs o']$ is unimodular and is a strong unimodularization of $[G]$. % in the strong sense 
	(Definition~\ref{def:unimodularization}). % both $[G, \bs o]$ and $[G, 0]$ in the strong sense.
\end{example}

\begin{example}
	%	\mar{swap $G$ and $G'$ to be compatible with the definition?}
	Let $[\bs G', \bs o']$ be a unimodular network and $S$ be a covariant subnetwork {such that $S_{\bs G'}$ is nonempty and connected a.s. Let} $[\bs G,\bs o]$ be the random rooted network obtained by conditioning $[\bs G', \bs o']$ on $\bs o'\in S_{\bs G'}$ (see Example~\ref{ex:subnetwork}). It can be seen that $([\bs G,\bs o], S)$ is a proper extension of $[S_{\bs G},\bs o]$ and by Example~\ref{ex:subnetwork}, $[\bs G', \bs o']$ is a weak unimodularization of $[\bs G]$. % in the weak sense. 
	This holds strongly %in the strong sense 
	if and only if the sample intensity $\probCond{\bs o'\in S_{\bs G'}}{I}$ of $S$ in $\bs G'$ is {essentially constant}.
\end{example}

We are now ready to state the results of this section. All proofs are postponed to the end of the subsection.

\begin{lemma}
	\label{lemma:uniExt}
	%\mar{Is the name good?}
	Let $([\bs G, \bs o],S)$ be a proper extension of a unimodular network.
	If $[\bs G]$ can be unimodularized, then there is a unique strong unimodularization of $[\bs G]$ and it can be obtained by applying a root-change to $[\bs G, \bs o]$. 
\end{lemma}

Theorem~\ref{thm:extensionT} provides a general method to construct unimodularizations of a given proper extension. Moreover, part~\eqref{thm:extensionT:1} of the theorem gives a criteria for verifying existence or non-existence of a unimodularization. 
%\ali{\sout{A special case of this theorem will be given in Proposition~\ref{prop:extensionTau}, which is sometimes easier to imagine and use (if possible).}}

\begin{theorem}[Unimodularization of an Extension]%[Unimodularization By a Transport Kernel]
	\label{thm:extensionT}
	Let $([\bs G, \bs o],S)$ be a proper extension of a unimodular network. % and assume $[\bs G_0, \bs o_0]$ is unimodular.	
	Assume $T:\mathcal G_{**}\rightarrow \mathbb R^{\geq 0}$ is a measurable function such that $T_{\bs G}$ is almost surely a Markovian transport kernel from $V(\bs G)$ to $S_{\bs G}$; i.e. almost surely, for all $v\in V(\bs G)$,
	%	$T_{\bs G}^+(v) %:=\sum_{z\in V(\bs G)}T_{\bs G}(v,z)
	%	=1$  and $T_{\bs G}(v,z)=0$ for all $z\not\in S_{\bs G}$. 
	$T_{\bs G}^+(v) = 1$ and $T_{\bs G}(v,\cdot)\equiv 0$ on $V(\bs G)\setminus S_{\bs G}$.
%	\mar{\ali{make this inline?}}
%	\begin{eqnarray*}
%		T_{\bs G}^+(v) = 1, &&\\
%		T_{\bs G}(v,z)=0,&& \forall z\not\in S_{\bs G}.
%	\end{eqnarray*}
%	
	Let $M:= M([\bs G, \bs o]):= T_{\bs G}^-(\bs o)$. Then,
	%\[M:= M([\bs G, \bs o]):= T_{\bs G}^-(\bs o).\] 
	%= \sum_{v\in V(\bs G)}T_{\bs G}(v, \bs o).\]
	\begin{enumerate}[(i)]
		\item 
		\label{thm:extensionT:1}
		%The extension $([\bs G, \bs o],S)$ 
		$[\bs G]$ can be unimodularized if and only if $\omidCond{M}{I}<\infty$ a.s.
		\item
		\label{thm:extensionT:2}
		If $\omid{M}<\infty$, then the following probability measure gives a weak unimodularization of $[\bs G]$. % in the weak sense. 
		\[
		\mathcal P_{T}[A]:=\frac 1{\omid{M}} \omid{\sum_{v\in V({\bs G})}T_{{\bs G}}( v, \bs o)\identity{A}[{\bs G},v]}
		\]
		%In other words, this is just biasing by $M$ and moving the root to a random vertex with distribution $\frac 1M T_{{\bs G}}(\cdot, \bs o)$.
		\item
		\label{thm:extensionT:3}
		If $\omidCond{M}{I}<\infty$ a.s., then the following probability measure gives the unique strong unimodularization of $[\bs G]$. % in the strong sense. }
		\[
		\mathcal P'_{T}[A]:= \omid{\frac 1{\omidCond{M}{I}}\sum_{v\in V({\bs G})}T_{{\bs G}}( v, \bs o)\identity{A}[{\bs G},v]}
		\]
		%In other words, this is just biasing by ${M}/{\omidCond{M}{I}}$ and then moving the root to a random vertex with distribution $\frac 1 M T_{{\bs G}}(\cdot, \bs o)$.
%		\item
%		\label{thm:extensionT:4}
%		Any unimodularization of $[\bs G]$, % the extension $([\bs G, \bs o],S)$, 
%		if exists any, can be obtained by biasing $\mathcal P'_{T}$ by a function that doesn't depend on the root.
	\end{enumerate}
\end{theorem}

\begin{remark}
	\label{rem:biasByM}
	The probability measure $\mathcal P_T$ (resp. $\mathcal P'_T$) in Theorem~\ref{thm:extensionT} can be described as biasing the distribution of $[\bs G, \bs o]$ by $M$ (resp. ${M}/{\omidCond{M}{I}}$) and then changing the root to a random vertex with distribution $\frac 1M T_{{\bs G}}(\cdot, \bs o)$.
\end{remark}

As an example, it can be seen that the unimodularization of Example~\ref{ex:extensionZ} can be obtained by the constructions in Theorem~\ref{thm:extensionT}. More examples are provided at the end of this section.

%\begin{remark}
%	ater: Move this to the proof of part~\ref{thm:extensionT:4}.
%	Let $T$ be an arbitrary function in Theorem~\ref{thm:extensionT}.
%	By Proposition~\ref{lemma:unimodularNonErgodic} and part~\eqref{thm:extensionT:1} of Theorem~\ref{thm:extensionT}, any other unimodularization of the extension is obtained by biasing the probability measure $\mathcal P'_T$ in part~\eqref{thm:extensionT:3} by a function which doesn't depend on the root. This includes the one in part~\eqref{thm:extensionT:2} as $\mathcal P_{T}$ is just biasing ${\mathcal P'_{T}}$ by $\omidCond{M}{I}$, provided $\omid{M}<\infty$.
%\end{remark}

\begin{corollary}
	\label{cor:extensionUnique}
	Under the assumptions of Theorem~\ref{thm:extensionT}, if $\myprob{[\bs G, \bs o]\in A}\in \{0,1\}$ for all invariant events $A\in I$, then $\mathcal P_T$ and $\mathcal P'_T$ are equal and extremal. Moreover, $[\bs G]$ has a unique weak unimodularization. % of $[\bs G]$. % (either weakly or strongly). % (in either sense).
\end{corollary}

In Theorem~\ref{thm:extensionT}, the function $T$ is assumed to be given. In the following theorem, the existence of $T$ is studied using Theorem~\ref{thm:shiftCoupling}.

\begin{theorem}[{Existence of $T$}]
	\label{thm:uniExtGen}
	%\mar{Is the name good?}
	Let $([\bs G, \bs o],S)$ be a proper extension of a unimodular network.
	If $[\bs G]$ %this extension 
	can be unimodularized, then 
	\begin{enumerate}[(i)]
		\item 	\label{thm:uniExtGen:T} 
			There exists a function $T:\mathcal G_{**}\rightarrow\mathbb R^{\geq 0}$ satisfying the assumptions in Theorem~\ref{thm:extensionT} such that $\mathcal P'_T$ exists. %Therefore, all unimodularizations of $[\bs G]$ can be obtained by part~\eqref{thm:extensionT:4} of Theorem~\ref{thm:extensionT}.
		
		\item \label{thm:uniExtGen:all}
			Any unimodularization of $[\bs G]$ can be obtained by biasing $\mathcal P'_{T}$ of the previous part by a function that doesn't depend on the root.
		\item 	\label{thm:uniExtGen:Balance}
			$T$ can be chosen such that almost surely, $T_{\bs G}^-(\cdot)$ is constant on $S_{\bs G}$ (and depends only on $[\bs G]$).
		
%		\item 	\label{thm:uniExtGen:Unique} %\mar{delete the ergodic part?} 
%		there is a unique strong unimodularization of $[\bs G]$. % in the strong sense %If in addition, $[\bs G_0, \bs o_0]$ is ergodic, then uniqueness also holds in the weak sense.
%		\item 	\label{thm:uniExtGen:RC}
%		and $[\bs G]$ can be unimodularized strongly % in the strong sense 
%		by applying a root-change to $[\bs G, \bs o]$. 
	\end{enumerate}
\end{theorem}

%The proof of Theorem~\ref{thm:uniExtGen} does not construct $T$ by looking only at a realization of the network. A construction will presented in Section~\ref{sec:construction}.

%ater: unimodularization by applying a root-change is equivalent to a balancing transport between $Counting_{\bs G}$ and $\lambda_S Counting_{S_{\bs G}}$.

\begin{remark}[Comparison of $\mathcal P_T$ vs $\mathcal P'_T$]
	%\ali{ater: link or merge with Rem~\ref{rem:modifiedPalm}.}
	The unimodularization $\mathcal P_{T}$ in Theorem~\ref{thm:extensionT} has some advantages comparing to $\mathcal P'_T$ and the other unimodularizations. One advantage is that it is easier to use since there is no division in the term under expectation and no conditional expectation w.r.t. $I$. Another is that the numerator %$\omid{\sum_{v\in V({\bs G})}T_{{\bs G}}( v, \bs o)\identity{A}([{\bs G},v])}$ 
	in the definition of $\mathcal P_T[\cdot]$ (which is $\omid{M}\times \mathcal P_T[\cdot]$) is a linear function of the distribution of $[\bs G, \bs o]$. Another one is shown in the following proposition. In contrast, an advantage of $\mathcal P'_T$ is that it is the strong unimodularization % in the strong sense 
	and thus, it doesn't change the distribution of the underlying non-rooted network (see Condition~\ref{conditionF} in Definition~\ref{def:same} and Example~\ref{ex:path}). Moreover, in some cases $\mathcal P'_T$ is defined but $\mathcal P_T$ is not, as shown in Example~\ref{ex:path}.
\end{remark}

\begin{proposition}
	\label{prop:P_T}

	Let $([\bs G, \bs o],S)$ be a proper extension of a unimodular network.
	\begin{enumerate}[(i)]
		\item \label{prop:P_T:1} There is at most one unimodularization $[\bs G', \bs o']$ of $[\bs G]$ % $([\bs G, \bs o],S)$ 
		with the property that $[\bs G', \bs o']$ conditioned on $\bs o'\in S_{\bs G'}$ has the same distribution as $[\bs G, \bs o]$.
		\item \label{prop:P_T:2} Under the assumptions of Theorem~\ref{thm:extensionT}, the unimodularization in the previous part exists if and only if $\omid{M}<\infty$ and is given by $\mathcal P_T$, which is a weak unimodularization. % in the weak sense.
	\end{enumerate}
\end{proposition}

\begin{corollary}
	\label{cor:P_T-notDepend}
	The probability measures $\mathcal P_T$ and $\mathcal P'_T$ of Theorem~\ref{thm:extensionT} (and their existence) do not depend on the choice of the function $T$. %Also, if at least two unimodularization exist, then there are unimodularizations which cannot be obtained by part~\eqref{thm:extensionT:2} or~\eqref{thm:extensionT:3} of Theorem~\ref{thm:extensionT} (compare with part~\eqref{thm:uniExtGen:all} of Theorem~\ref{thm:uniExtGen}). 
\end{corollary}

\begin{remark}
	For a proper extension $([\bs G, \bs o],S)$ of a unimodular network, if $[\bs G]$ has at least two different unimodularizations, then by taking convex combinations of the probability measures one finds infinitely many unimodularizations. So, Corollary~\ref{cor:P_T-notDepend} implies that there are unimodularizations which cannot be obtained by parts~\eqref{thm:extensionT:2} and~\eqref{thm:extensionT:3} of Theorem~\ref{thm:extensionT} (note that a further biasing is enough as described in part~\eqref{thm:uniExtGen:all} of Theorem~\ref{thm:uniExtGen}).
\end{remark}

\begin{remark}
	\label{rem:extensionTau}
	There is a special case of the construction in Theorem~\ref{thm:extensionT} which is sometimes easier to imagine and use, if possible. In the theorem, assume that for all vertices $v$, the function $T_{G}(v,\cdot)$ is concentrated on a single vertex, namely $\tau_{G}(v)$. Such a mapping $\tau$ can be called an \defstyle{allocation} (or a \defstyle{vertex-shift} in~\cite{vertexShift}). Then, the assumptions for $T$, apart from measurability, are reduced to $\tau_{\bs G}(\cdot)\in S_{\bs G}$ a.s. Also, $M$ is equal to $\card{\tau_{\bs G}^{-1}(\bs o)}$. However, In the setting of Theorem~\ref{thm:uniExtGen}, existence of such an allocation is not guaranteed in general.
	%\ali{\sout{One obstacle is automorphisms of the network. This is the case in Example~\ref{ex:dual2} for the usual graph of $\mathbb Z^2$.}}
\end{remark}

The following are some basic examples for applications of the definitions and results, although they are not new.

\begin{example}
	\label{ex:ZandZ^2}
	%	\mar{ater: $\mathbb Z^2$ is better than $T$.}
	Let $[G_1, o_1]$ be the deterministic graph of $\mathbb Z$ and $(G_2, o_2)$ be the usual lattice graph of $\mathbb Z^2$ rooted at the origin. Although $[G_2, o_2]$ is unimodular and can be obtained by adding some vertices and edges to $[G_1, o_1]$, it is not regarded as an extension of $[G_1, o_1]$ in Definition~\ref{def:extension} because $[G_1, o_1]$ cannot be recovered from $[G_2, o_2]$ as a subnetwork (that does not depend on the root). However, one may add an extra mark to the vertices outside the $x$ axis in $\mathbb Z^2$ to obtain a network, namely $(G, o)$. By letting $P_G$ be the distinguished path in $G$, $([G,o],P)$ is an extension of $[G_1, o_1]$. It is easy to see that~\eqref{eq:mtpOnS} holds and this extension is proper. But $[G]$ cannot be unimodularized as shown below.
	For $v\in V(G)$, let $\tau_{G}(v)$ be the closest vertex in the distinguished path $P_G$ to $v$. It can be seen that the assumptions in Remark~\ref{rem:extensionTau} hold for $\tau$. Since $\card{\tau_{G}^{-1}(o)}=\infty$, part~\eqref{thm:extensionT:1} of Theorem~\ref{thm:extensionT} implies that $[G]$ cannot be unimodularized.

	%	Although $[T, o]$ is unimodular and can be obtained by adding some vertices and edges to $[G_0, o_0]$, it cannot be considered as an extension of $[G_0, o_0]$ in Definition~\ref{def:extension} because $[G_0, o_0]$ cannot be recovered from $[T, o]$ as a subnetwork. However, one may distinguish a bi-infinite path $P$ in $T$ containing $o$ and add a mark to the vertices outside $P$. The resulting network, namely $(T', o)$, gives an extension $([T',o],P)$ of $[G_0, o_0]$. It is easy to see that~\eqref{eq:mtpOnS} holds and this extension is proper. But $[T']$ cannot be unimodularized as discussed below.
	%	
	%	For $v\in V(T')$, let $\tau_{T'}(v)$ be the closest vertex in the path $P$ to $v$. It can be seen that the assumptions in Proposition~\ref{prop:extensionTau} hold for $\tau$. Since $\card{\tau_{T'}^{-1}(o)}=\infty$, part~\eqref{prop:extensionTau:1} of Proposition~\ref{prop:extensionTau} implies that $[T']$ cannot be unimodularized.
\end{example}

\begin{example}
	\label{ex:path}
	%	Let $(G_n)_n$ be a sequence of networks where $G_n$ has $n$ vertices. Assume each $G_n$ is a subnetwork, namely $S_{\bs G}$, of a random rooted network $[\bs G_n, \bs o_n]$ and $\bs o_n\in S_{\bs G_n}$ a.s. Assume $\card{V(\bs G_n)}=a_n$ a.s. Let $N$ be a random number and $[\bs G, \bs o]:=[\bs G_N, \bs o_N]$. Then by Theorem~\ref{thm:extensionT}, $[\bs G]$ can be unimodularized if and only if each $[\bs G_n]$ can be unimodularized (note that conditioning on $I$ is equivalent to conditioning on $N$ here). This is equivalent to $a_n<\infty$ for each $n$. In this case, $\mathcal P'_T$ exists and gives a unimodularization in the strong sense. But $\mathcal P_T$ exists if and only if $\omid{a_N}<\infty$.
	Let $[G_0,o_0]$ be the graph with a single vertex and $L$ be a random odd number in $\mathbb N$. Let $\bs G$ be a path with length $L$  and $\bs o$ be the middle vertex. Then, $[\bs G, \bs o]$ is a proper extension of $[G_0, o_0]$. By sending unit mass from each vertex to $\bs o$, one gets $M=L$ in Theorem~\ref{thm:extensionT} and $\omidCond{M}{I}=L<\infty$. Therefore, $[\bs G]$ can always be unimodularized. Indeed, $\mathcal P'_T$ is just changing the root to a uniformly at random vertex in $[\bs G]$. However, $\mathcal P_T$ is defined only when $\omid{L}<\infty$. In this case, $\mathcal P_T$ is equivalent to choosing a path whose probability of having length $k$ is $k\myprob{L=k}/\omid{L}$ (i.e. the size-biased version of $L$) and then, choosing a uniformly at random root in the path. Note that the distribution of the underlying non-rooted network is changed under $\mathcal P_T$, but unchanged under $\mathcal P'_T$.
\end{example}

The following examples shows that examples~9.6 and~9.8 of~\cite{processes} are weak unimodularizations %in the weak sense 
and can be obtained by the method of Theorem~\ref{thm:extensionT}. 
\begin{example}[Vertex/Edge Replacement]
	\label{ex:replacement}
	%\mar{\ali{Later: shorten definitions and only cite\\ 2. Add vertex replacement}}
	Let $[\bs G_0, \bs o_0]$ be a unimodular network. In Example~9.6 of~\cite{processes}, it is shown how to attach to each edge $uv$ of $\bs G_0$ a doubly-rooted network $L(u,v)$ that depends on the marks of $u$ and $v$. The two roots of $L(u,v)$ are identified with $u$ and $v$. Here, to make sure $\bs G_0$ can be reconstructed, add some extra mark to the other added vertices and let $\bs G$ be the resulting network. Now, it can be seen that $[\bs G, \bs o_0]$ gives a proper extension of $[\bs G_0, \bs o_0]$. So, we may use Theorem~\ref{thm:extensionT} as follows. Send unit mass from each vertex of $\bs G_0$ to itself. Also, for each edge $uv$ and each vertex in $L(u,v)$ other than the roots, send mass $\frac 12$ to each of $u$ and $v$. It can be seen that this satisfies the assumptions of Theorem~\ref{thm:extensionT} and one has $M=1+\frac 12\sum_{v\sim \bs o}(\card{V(L(\bs o, v))}-2)$. Therefore, by Theorem~\ref{thm:extensionT}, $[\bs G]$ can be unimodularized if and only if $\omidCond{M}{I}<\infty$ and $\mathcal P'_T$ gives the strong unimodularization of $[\bs G]$. % in the strong sense. 
	Moreover, the probability measure constructed in~\cite{processes} (defined when $\omid{M}<\infty$) is identical with $\mathcal P_T$, which is a weak unimodularization. % in the weak sense.
	
	With similar arguments, one can append a rooted network to each vertex of $\bs G_0$ and use the method of Theorem~\ref{thm:extensionT} to obtain another unimodular network.

\end{example}

\begin{example}[Planar Dual II]
	\label{ex:dual2}
	Let $G_0$ be a plane graph. Consider the union of $G_0$ and its dual $G_0^*$ and add an edge $vf$ for each vertex $v$ and each of its adjacent faces $f\sim v$. To make sure $G_0$ can be recovered, add an extra mark to the vertices of $G_0^*$. Call the resulting network $G_0^{\dagger}$ as a function of $G_0$. Let $S$ be a covariant subnetwork such that $S_{G_0^{\dagger}}=G_0$ for all possible cases of $G_0$ in this construction.
	
	Let $[\bs G_0, \bs o_0]$ be a unimodular plane graph (see Example~\ref{ex:dual1}) and consider the random rooted network $[\bs G, \bs o]:=[\bs G_0^{\dagger},\bs o_0]$ as constructed above rooted at $\bs o:=\bs o_0$. %Lemma~\ref{lemma:extension}  implies that 
	It can be seen that $([\bs G, \bs o], S)$ is a proper extension of $[\bs G_0, \bs o_0]$. We will use Theorem~\ref{thm:extensionT}. 
	Send unit mass from each vertex $v\in V(\bs G_0)$ to itself. From each face $f$, send mass $\frac 1{\text{deg}(f)}$ to each of its adjacent vertices in $\bs G_0$, where ${\text{deg}(f)}$ is the number of vertices of $f$. It can be seen that this satisfies the assumptions of Theorem~\ref{thm:extensionT} and $M=1+\sum_{f\sim \bs o} 1/{\text{deg}(f)}$.
%	
%	
%	For $G_0$ as above and $v\in V(G_0)$, let $T_{G_0^{\dagger}}(v,v)=1$. % and $T_{G_0^{\dagger}}(v,w)=0$ for $w\neq v$. 
%	For each face $f$ of $G_0$ and each adjacent vertex $v\sim f$, let $T_{G_0^{\dagger}}(f,v):=\frac 1{\text{deg}(f)}$, where $\text{deg}(f)$ is the number of vertices of $f$. Let $T_{G_0^{\dagger}}(\cdot,\cdot)$ be zero for other pairs of vertices of $G_0^{\dagger}$. %One has $T^+_{G_0^{\dagger}}(\cdot)\equiv 1$ on $V(G_0^{\dagger})$.
%	It can be seen that $T$ satisfies the assumptions of Theorem~\ref{thm:extensionT}. One has
%	\[
%	M= T_{\bs G}^-(\bs o)=1+\sum_{f\sim \bs o} \frac 1{\text{deg}(f)},
%	\]
	%where the sum runs over the faces $f$ adjacent to $\bs o$. 
	Theorem~\ref{thm:extensionT} implies that $[\bs G]$ can be unimodularized if and only if $\omidCond{M}{I}<\infty$.
	If in addition $\omid{M}<\infty$, then the following gives a weak unimodularization of $[\bs G]$. % in the weak sense.
%	$\mathcal P_{T}[A]:= \frac 1{\omid{M}}\omid{\identity{A}[\bs G, \bs o] + \sum_{f\sim \bs o}\frac 1{\text{deg}(f)}\identity{A}[\bs G, f]}.$
	\begin{eqnarray*}
		\mathcal P_{T}[A]:= \frac 1{\omid{M}}\omid{\identity{A}[\bs G, \bs o] + \sum_{f\sim \bs o}\frac 1{\text{deg}(f)}\identity{A}[\bs G, f]}.
	\end{eqnarray*}
	By passing to $\bs G_0^*$ as a subnetwork of $\bs G$ as in Example~\ref{ex:subnetwork},
	one can obtain a weak unimodularization of the dual random non-rooted graph $[\bs G_0^*]$. % in the weak sense. 
	This is obtained by replacing $M$ by $M-1$ in the above formula and deleting the term $\identity{A}[\bs G, \bs o]$. This is identical to the one constructed in Example~9.6 of~\cite{processes}. Similarly, the following is the strong unimodularization of $[\bs G_0^*]$. % in the strong sense.
	%$A \mapsto  \omid{\frac 1{\omidCond{M-1}{I}}\sum_{f\sim \bs o_0}\frac 1{\text{deg}(f)}\identity{A}[\bs G_0^*, f]}.$
	\begin{eqnarray*}
		A \mapsto  \omid{\frac 1{\omidCond{M-1}{I}}\sum_{f\sim \bs o_0}\frac 1{\text{deg}(f)}\identity{A}[\bs G_0^*, f]}.
	\end{eqnarray*}

%	If so, by Theorem~\ref{thm:extensionT} again, the following give unimodularizations of $[\bs G]$ in the weak and strong sense respectively (the first one is defined only when $\omid{M}<\infty$). % (note that $[\bs G, \bs o]$ is a function of $[\bs G_0, \bs o_0]$, thus, the following is written in terms of $[\bs G_0, \bs o_0]$).
%	\begin{eqnarray*}
%		\mathcal P_{T}[A]&:=& \frac 1{\omid{M}}\omid{\identity{A}[\bs G, \bs o] + \sum_{f\sim \bs o}\frac 1{\text{deg}(f)}\identity{A}[\bs G, f]},\\
%		\mathcal P'_{T}[A]&:=& \omid{\frac 1{\omidCond{M}{I}}\left(\identity{A}[\bs G, \bs o] + \sum_{f\sim \bs o}\frac 1{\text{deg}(f)}\identity{A}[\bs G, f] \right)}.
%	\end{eqnarray*}
%	By passing to $\bs G_0^*$ as a subnetwork of $\bs G$ as in Example~\ref{ex:subnetwork},
%	one can obtain that the following give unimodularizations of the dual random non-rooted graph $[\bs G_0^*]$ in the weak and strong sense respectively.
%	\mar{\ali{delete these and mention modifying of prev equations?}}
%	\begin{eqnarray*}
%		A &\mapsto & \frac 1{\omid{M-1}}\omid{\sum_{f\sim \bs o_0}\frac 1{\text{deg}(f)}\identity{A}[\bs G_0^*, f]},\\
%		A &\mapsto & \omid{\frac 1{\omidCond{M-1}{I}}\sum_{f\sim \bs o_0}\frac 1{\text{deg}(f)}\identity{A}[\bs G_0^*, f]}.
%	\end{eqnarray*}
%	The former (which is a unimodularization in the weak sense) is identical to the one constructed in Example~9.6 of~\cite{processes}. 
	
	Similar to this example, one can add a new vertex for each edge-crossing of $\bs G_0$ with its dual. With similar arguments, the method of Theorem~\ref{thm:extensionT} gives the unimodularization of this new network given in Example~9.6 of~\cite{processes}.
	
%	From each such vertex, send mass $\frac 12$ to each of its two neighboring vertices of $\bs G_0$. Similarly, by using part~\eqref{thm:extensionT:2} of Theorem~\ref{thm:extensionT}, one obtains the unimodularization of this new network given in Example~9.6 of~\cite{processes}.
\end{example}

Now, the proofs of the results of this section are presented. First, we start with the following lemma.

\begin{lemma}
	\label{lemma:extBias}
	Let $([\bs G, \bs o], S)$ be a proper extension of a unimodular network. If $[\bs G', \bs o']$ is a unimodularization of $[\bs G]$, %this extension, 
	then $[\bs G, \bs o]$ is obtained by biasing $[\bs G', \bs o']$ by a function $w$ such that almost surely, $w[\bs G',\cdot]$ is constant on $S_{\bs G'}$ (but may depend on $[\bs G']$) and is zero outside $S_{\bs G'}$.
\end{lemma}

\begin{proof}%[Proof of Lemma~\ref{lemma:extBias}]
	By Lemma~\ref{lemma:unimodularBias}, $[\bs G, \bs o]$ can be obtained by biasing $[\bs G', \bs o']$ by a function, say $w$. Since $\bs o\in S_{\bs G}$ a.s., one obtains that on the event $\bs o'\not\in S_{\bs G'}$, one has $w[\bs G', \bs o']=0$ a.s. Therefore, similar to Lemma~\ref{lemma:happensAtRoot}, \eqref{eq:mtpOnS} implies that $w[\bs G',\cdot]$ is zero outside $S_{\bs G'}$ almost surely. Also, using~\eqref{eq:mtpOnS} and an argument {similar} to Lemma~\ref{lemma:unimodularBiasUni}, one obtains that almost surely, $w[\bs G', \cdot]$ is constant on $S_{\bs G'}$.  We skip repeating the arguments for brevity.
	%The property that almost surely $w[\bs G', \cdot]$ is constant on $S_{\bs G}$ can be proved using~\eqref{eq:mtpOnS} similar to the proof of Lemma~\ref{lemma:unimodularBiasUni} and is skipped for brevity.
\end{proof}

\begin{proof}[Proof of Lemma~\ref{lemma:uniExt}]
		Existence and uniqueness of a strong unimodularization $[\bs G', \bs o']$ is directly implied by Proposition~\ref{prop:unimodulalrizationUnique}.
%		Let $[\bs G', \bs o']$ be a strong unimodularization of $[\bs G]$, % in the strong sense, 
%		which exists by part~\eqref{thm:uniExtGen:Unique}. 
		By definition, $[\bs G', \bs o']$ is strongly unrooted-equivalent to $[\bs G, \bs o]$. % this has the same non-rooted network as $[\bs G, \bs o]$ in the strong sense. 
		Therefore, Theorem~\ref{thm:shiftCoupling} implies that $[\bs G' , \bs o']$ can be obtained from $[\bs G, \bs o]$ by a root change. %Indeed, by assuming the claim of part~\eqref{thm:uniExtGen:Balance}, this root 
\end{proof}

\begin{proof}[Proof of Theorem~\ref{thm:extensionT}] We prove the parts of the theorem in a different order. %part~\eqref{thm:extensionT:1} to after part~\eqref{thm:extensionT:3}. Let $m:=\omid{M}$.
	
	\eqref{thm:extensionT:3} 
	%First, note that by the assumptions, one gets 
	The assumptions imply 
	that almost surely, for some vertex $v\in S_{\bs G}$, $T_{\bs G}^-(v)>0$. {By~\eqref{eq:mtpOnS} and an argument} % Since $[S_{\bs G}, \bs o]$ is unimodular, 
	similar to Lemma~\ref{lemma:happensAtRoot}, one can obtain that {$\probCond{T_{\bs G}^-(\bs o)>0}{I}>0$ a.s., hence, $\omidCond{M}{I}>0$ a.s.} Therefore, $\mathcal P'_T$ is well-defined. It can be seen that $\mathcal P'_T$ is a probability measure.
	Now, let $[\bs{\widetilde G}, \bs{\widetilde o}]$
	be a random rooted network with distribution $\mathcal P'_{T}$.
	Let $g:\mathcal G_{**}\rightarrow\mathbb R^{\geq 0}$
	be a measurable function. By the definition of $\mathcal P'_T$ and conditioning on $I$, one gets
	\begin{eqnarray*}
	&& \omid{\sum_{u\in V({\bs{\widetilde G}})}g[{\bs {\widetilde G}, \bs{\widetilde o}}, u]} \\
	&=& \omid{\frac 1{\omidCond{M}{I}} \omidCond{\sum_{v\in V({\bs G})}\sum_{u\in V({\bs G})} T_{{\bs G}}(v, \bs o)g[{\bs G}, v, u]}{I}}\\
	&=& \omid{\frac 1{\omidCond{M}{I}} \omidCond{\sum_{z\in S_{\bs G}}\sum_{v\in V({\bs G})}\sum_{u\in V({\bs G})} T_{{\bs G}}( v, \bs o)T_{{\bs G}}( u, z)g[{\bs G}, v, u]}{I}}\\
	%\label{eq:extension:1} 
	&=& \omid{\frac 1{\omidCond{M}{I}} \omidCond{\sum_{z\in S_{\bs G}} \hat g[\bs G, \bs o, z]}{I}},
	\end{eqnarray*}
	where 
	$\hat g[G, o, z]:=\sum_{v\in V(G)}\sum_{u\in V(G)} T_{G}(v, o)T_{G}(u,z)g[G, v, u]$
	and in the second equality, the equation $\sum_{z\in S_{\bs G}} T_{\bs G}(u,z)=1$ a.s. is used, which holds by the assumptions. It can be seen that $\hat g$ is well-defined and measurable.
	One obtains a similar equation by swapping the two roots in all instances of 
	$g$ and $\hat g$. To summarize, 
	\begin{equation}
	\label{eq:extension:1}
	\left\{
	\begin{array}{rcl}
	\omid{g^+_{\widetilde{\bs G}}(\widetilde{\bs o})} &=& \omid{\frac 1{\omidCond{M}{I}} \omidCond{\sum_{z\in S_{\bs G}} \hat g[\bs G, \bs o, z]}{I}},\\
	\omid{g^-_{\widetilde{\bs G}}(\widetilde{\bs o})} &=& \omid{\frac 1{\omidCond{M}{I}} \omidCond{\sum_{z\in S_{\bs G}} \hat g[\bs G,z, \bs o]}{I}}.
	\end{array}
	\right.
	\end{equation}

	Let $A\in I$ be an invariant event. By~\eqref{eq:mtpOnS}, one gets
	\[
	\omid{\sum_{z\in S_{\bs G}} \hat g[\bs G, \bs o, z]\identity{A}[\bs G, \bs o]} = \omid{\sum_{z\in S_{\bs G}} \hat g[\bs G, z, \bs o]\identity{A}[\bs G, \bs o]}.
	\]
	Since this holds for all $A\in I$, it follows that the right hand sides of~\eqref{eq:extension:1} are equal. Thus, so are the left hand sides. This implies that~\eqref{eq:unimodular} holds for $[\widetilde{\bs G}, \widetilde{\bs o}]$, hence, $\mathcal P'_T$ is unimodular.
	To show that $\mathcal P'_T$ is a strong unimodularization of $[\bs G]$ %in the strong sense 
	(which implies that it is the unique one by Proposition~\ref{prop:unimodulalrizationUnique}), we will prove that $[{\widetilde{\bs G}}]$ has the same distribution as $[\bs G]$. Let $B\in I$. By the definition of $\mathcal P'_T$ and invariance of $B$ under changing root, one gets
	\begin{eqnarray*}
		\myprob{[{\widetilde{\bs G}}, \widetilde{\bs o}]\in B} &=& \omid{\frac 1{\omidCond{M}{I}} \sum_{v\in V({\bs G})} T_{\bs G}(v,\bs o)\identity{B}[{\bs G}, v] }\\
		&=& \omid{ \frac 1{\omidCond{M}{I}} M\identity{B}[{\bs G}, \bs o]}\\
		&=& \omid{\frac 1{\omidCondL{M}{I}} \omidCond{M\identity{B}[{\bs G}, \bs o]}{I}}\\
		&=& \omid{\identity{B}[{\bs G}, \bs o]}\\
		&=& \myprob{[\bs G, \bs o]\in B}.
	\end{eqnarray*}
	Therefore, the distributions of $[\widetilde{\bs G}, \widetilde{\bs o}]$ and $[\bs G, \bs o]$ agree on $I$, which proves the claim.
	%So the claim is proved.
	
	\eqref{thm:extensionT:2} $\mathcal P_{T}$, whenever defined, is just biasing ${\mathcal P'_{T}}$ by $\omidCond{M}{I}$. Since the bias is $I$-measurable, lemmas~\ref{lemma:unimodularBiasUni} and~\ref{lemma:biasG} imply that $\mathcal P_T$ is unimodular and is unrooted-equivalent to $[\bs G, \bs o]$ (not necessarily strongly). % has the same non-rooted network as $[\bs G, \bs o]$ (not necessarily in the strong sense).
	
	\eqref{thm:extensionT:1} If $\omidCond{M}{I}<\infty$, then $\mathcal P'_T$ is well-defined and is a unimodularization of the extension by part~\eqref{thm:extensionT:3} proved above. Conversely, assume $[\bs G', \bs o']$ is a unimodularization of the extension. By definition, the distributions of $[\bs G']$ and $[\bs G]$ are mutually absolutely continuous. Therefore, by the assumptions on $T$, almost surely, for all $v\in V(\bs G')$ one has $T^+_{\bs G'}(v)=1$ and $T_{\bs G'}(v,z)=0$ for all $z\not\in S_{\bs G'}$. Now, $\omid{T_{\bs G'}^+(\bs o')}=1$. So, unimodularity of $[\bs G', \bs o']$ implies that $\omid{T_{\bs G'}^-(\bs o')=1}$. In particular, one obtains that $\omidCond{T_{\bs G'}^-(\bs o')}{I}<\infty$ a.s. 
	Consider the function $w$ in Lemma~\ref{lemma:extBias} and let $c=c([\bs G'])$ be the common value of $w[\bs G', \cdot]$ on the vertices of $S_{\bs G'}$. Lemma~\ref{lemma:extBias} implies that $\omidCond{T_{\bs G}^-(\bs o)}{I}$ is proportional to $\omidCond{w[\bs G', \bs o'] T_{\bs G'}^-(\bs o')}{I}$, which is equal to $c[\bs G'] \omidCond{T_{\bs G'}^-(\bs o')}{I}$ (note that $w[\bs G', \bs o'] T_{\bs G'}^-(\bs o') = c[\bs G'] T_{\bs G'}^-(\bs o')$ a.s.) The latter is shown above to be finite a.s. This implies that $\omidCond{M}{I}<\infty$ a.s.
	
%	By the notations in Lemma~\ref{lemma:extBias}, one obtains that $\omidCond{T_{\bs G}^-(\bs o)}{I}$ is proportional to \ali{\mar{correction: w not correct, but the constant of w}$w[\bs G', \bs o'] \omidCond{T_{\bs G'}^-(\bs o')}{I}$}, which is finite a.s. This proves that $\omidCond{M}{I}<\infty$ a.s.
%	%	 Since this term is $I$-measurable and holds for $[\bs G', \bs o']$ a.s., it also holds for $[\bs G, \bs o]$ a.s.; i.e. $\omidCond{T_{\bs G}^-(\bs o)}{I}<\infty$ a.s. This proves the claim.
	
%	\eqref{thm:extensionT:4} Assume a unimodularization $[\bs G', \bs o']$ exists. By parts~\eqref{thm:extensionT:1} and~\eqref{thm:extensionT:3}, $\mathcal P'_T$ also exists. Any two unimodularizations of a given extension are weakly unrooted-equivalent. % have the same non-rooted network in the weak sense. 
%	Therefore, Lemma~\ref{lemma:unimodularNonErgodic} implies that $[\bs G', \bs o']$ can be obtained by biasing $\mathcal P'_T$ by a function that doesn't depend on the root.
\end{proof}

\begin{proof}[Proof of Corollary~\ref{cor:extensionUnique}]
	By the strong sense in part~\eqref{thm:extensionT:3} of Theorem~\ref{thm:extensionT}, one obtains that $\mathcal P'_T$ also satisfies the assumption on $[\bs G, \bs o]$, which implies that $\mathcal P'_T$ is extremal. Therefore, Proposition~\ref{prop:unimodularErgodic} implies that any weak unimodularization of $[\bs G]$ is also strongly unrooted-equivalent to $\mathcal P'_T$. Now, Theorem~\ref{thm:unimodularUnique} implies that $\mathcal P'_T$ is the unique weak unimodularization of $[\bs G]$.
%	
%	By Remark~\ref{rem:biasByM}, $\mathcal P_T$ is obtained by biasing $\mathcal P'_T$ by $\omid{M}{I}$. Since the latter is 
%	
%	Therefore, any $I$-measurable function is essentially constant w.r.t $\mathcal P'_T$. Thus, part~\eqref{thm:extensionT:4} of Theorem~\ref{thm:extensionT}, shows that $\mathcal P'_T$ is equal to any unimodularization of $[\bs G]$ including $\mathcal P_T$.
\end{proof}

\begin{proof}[Proof of Theorem~\ref{thm:uniExtGen}]
	
	\eqref{thm:uniExtGen:T} 
	Let $[\bs G', \bs o']$ be a strong unimodularization of the extension. % in the strong sense. 
	Define $w:\mathcal G_*\rightarrow\mathbb R^{\geq 0}$ such that 
	$w[\bs G', \bs o']= \identity{\{\bs o'\in S_{\bs G'}\}} / \omidCond{\identity{\{\bs o'\in S_{\bs G'}\}}}{I}.$
	Since $S_{\bs G}\neq\emptyset$ a.s. and this property doesn't depend on the root, the same holds for $\bs G'$. Similar to Lemma~\ref{lemma:happensAtRoot} one can deduce that the above denominator is nonzero a.s. and $w$ is well-defined up to an event of measure zero.
	One has $\omidCond{w[\bs G', \bs o']}{I}=1$ a.s. Therefore, Theorem~\ref{thm:balancing} implies that there is a measurable function $T:\mathcal G_{**}\rightarrow\mathbb R^{\geq 0}$ such that almost surely, for all $v\in V(\bs G')$, one has $T_{\bs G'}^+(v)=1$ and $T_{\bs G'}^-(v)=w[\bs G', v]$. In particular, almost surely, for all $v\in V(\bs G')\setminus S_{\bs G'}$,  one has $T_{\bs G'}^-(v)=0$. Since $[\bs G]$ has the same distribution as $[\bs G']$, the same holds for $[\bs G]$; i.e. almost surely, for all $v\in V(\bs G)$, $T_{\bs G}^+(v)=1$ and if $v\not\in S_{\bs G}$, then $T_{\bs G}^-(v)=0$. So, the assumptions of Theorem~\ref{thm:extensionT} are satisfied by $T$. To show that $\mathcal P'_T$ is defined, it remains to prove that $\omidCond{T_{\bs G}^-(\bs o)}I<\infty$.

	Let $\lambda:(\mathcal G,J)\rightarrow\mathbb R^{\geq 0}$ be the measurable function such that $\lambda[\bs G']= \omidCond{1_{\{\bs o'\in S_{\bs G'}\}}}{I}$.
	According the the above arguments, 
	%	 one finds a measurable function $\lambda:\mathcal G\rightarrow\mathbb R^{\geq 0}$ such that $T_{\bs G'}^-(\bs o')= \lambda[\bs G'] \identity{\{\bs o'\in S_{\bs G'}\}}$ a.s. By unimodularity, 
	almost surely, for all $v\in V(\bs G')$, $T_{\bs G'}^-(v)= \frac 1{\lambda[\bs G']} \identity{\{v\in S_{\bs G'}\}}$. Since $[\bs G]$ has the same distribution as $[\bs G']$, the same holds for $[\bs G]$; i.e. almost surely,
	$\forall v\in V(\bs G), T_{\bs G}^-(v)= \frac 1{\lambda[\bs G]} \identity{\{v\in S_{\bs G}\}}.$
%	\begin{equation}
%	\label{eq:thm:uniExtGen:1}
%	\forall v\in V(\bs G), T_{\bs G}^-(v)= \frac 1{\lambda[\bs G]} \identity{\{v\in S_{\bs G}\}}.
%	\end{equation}
	In particular, since $\bs o\in S_{\bs G}$ a.s., one gets $T_{\bs G}^-(\bs o)= \frac 1{\lambda[\bs G]}$ a.s. Since this doesn't depend on the root, $\omidCond{T_{\bs G}^-(\bs o)}{I}= \frac 1{\lambda[\bs G]}<\infty$ a.s. and the claim is proved.

	\eqref{thm:uniExtGen:all}
	Let $[\bs G', \bs o']$ be a unimodularization of $[\bs G]$. Any two unimodularizations are weakly unrooted-equivalent by definition. % have the same non-rooted network in the weak sense. 
	Therefore, Lemma~\ref{lemma:unimodularNonErgodic} implies that $[\bs G', \bs o']$ can be obtained by biasing $\mathcal P'_T$ by a function that doesn't depend on the root.

	\eqref{thm:uniExtGen:Balance} 
	%Equation~\eqref{eq:thm:uniExtGen:1} implies that 
	The function $T$ constructed in the proof of part~\eqref{thm:uniExtGen:T} satisfies the desired property.

\end{proof}

\begin{proof}[Proof of Proposition~\ref{prop:P_T}]
	%	{ }
	%	
	%	\eqref{prop:P_T:1} Let $[\bs G_1, \bs o_1]$ and $[\bs G_2, \bs o_2]$ be such unimodularizations. %By Lemma~\ref{lemma:extBias}, $[\bs G, \bs o]$ can be obtained by biasing $[\bs G_1, \bs o_1]$ by a function, say $w_1$. On the other hand,
	%	Example~\ref{ex:subnetwork} shows that for each $i=1,2$, $[\bs G, \bs o]$ can be obtained by biasing $[\bs G_i, \bs o_i]$ by $\identity{\{\bs{\bs o_i}\in S_{\bs G_i}\}}$. On the other hand, since $[\bs G_1, \bs o_1]$ has the same non-rooted network as $[\bs G_2, \bs o_2]$ in the weak sense, Proposition~\ref{lemma:unimodularNonErgodic} shows that $[\bs G_2, \bs o_2]$ can be obtained by biasing $[\bs G_1, \bs o_1]$ by a measurable function $w:\mathcal G_*\rightarrow\mathbb R^{\geq 0}$ that doesn't depend on the root. This implies that $[\bs G, \bs o]$ is also obtained by biasing $[\bs G_1, \bs o_1]$ by $w[\bs G_1, \bs o_1] \identity{\{\bs{\bs o_1}\in S_{\bs G_1}\}}$. This shows that $w[\bs G_1, \bs o_1] \identity{\{\bs{\bs o_1}\in S_{\bs G_1}\}} =  c \identity{\{\bs{\bs o_1}\in S_{\bs G_1}\}}$ a.s. for some constant $c$. Since $w$ does not depend on the root and $S_{\bs G_1}\neq \emptyset$ a.s., one gets $w[\bs G_1, \bs o_1]=c$ a.s. This implies that $[\bs G_2, \bs o_2]$ has the same distribution as $[\bs G_1, \bs o_1]$ and the claim is proved.
	%	
	%	\eqref{prop:P_T:2}
	By Theorem~\ref{thm:uniExtGen}, one can assume $\mathcal P'_T$ is defined for some function $T$ satisfying the assumptions of Theorem~\ref{thm:extensionT} without loss of generality. Let $[\bs G', \bs o']$ have distribution $\mathcal P'_T$ and $[\widetilde{\bs G}, \widetilde{\bs o}]$ be any unimodularization of the extension. By part~\eqref{thm:uniExtGen:all} of Theorem~\ref{thm:uniExtGen}, $[\widetilde{\bs G}, \widetilde{\bs o}]$ is obtained by biasing $[\bs G', \bs o']$ by a measurable function $w:\mathcal G_*\rightarrow\mathbb R^{\geq 0}$ that does not depend on the root. By a scaling, one may assume $\omid{w[\bs G', \bs o']}=1$. For any event $A\subseteq\mathcal G_*$, one has
	\begin{eqnarray*}
		%\label{eq:prop:P_T}
		\myprob{\widetilde{\bs o}\in S_{\widetilde{\bs G}}, [\widetilde{\bs G}, \widetilde{\bs o}]\in A} &=& \omid{w[\bs G', \bs o'] \identity{\{\bs o'\in S_{\bs G'}\}}  \identity{A}[\bs G', \bs o']}\\
		&=&	\omid{\frac 1{\omidCond M I}\sum_{v\in S_{\bs G}} T_{\bs G}(v, \bs o) w[\bs G, v]\identity{A}[{\bs G},v]}\\
		&=&	\omid{\frac 1{\omidCond M I}\sum_{v\in S_{\bs G}} T_{\bs G}(\bs o, v) w[\bs G, \bs o]\identity{A}[{\bs G},\bs o]}\\
		&=& \omid{\frac 1{\omidCond M I}w[\bs G, \bs o]\identity{A}[{\bs G},\bs o]},
	\end{eqnarray*}
	where in the third equality \eqref{eq:mtpOnS} is used. 
	Therefore,
	\[
	\probCond{[\widetilde{\bs G}, \widetilde{\bs o}]\in A}{\widetilde{\bs o}\in S_{\widetilde{\bs G}}} =
	c \omid{\frac {w[\bs G, \bs o]}{\omidCond M I}\identity{A}[{\bs G},\bs o]},
	\]
	where $c=1/\myprob{\widetilde{\bs o}\in S_{\widetilde{\bs G}}}$.
	Thus, $[\widetilde{\bs G}, \widetilde{\bs o}]$ has the desired property if and only if $\frac {w[\bs G, \bs o]}{\omidCond M I}$ is essentially constant. If so, the distribution of $[\widetilde{\bs G}, \widetilde{\bs o}]$ is equal to biasing $\mathcal P'_T$ by $\omidCond{M}{I}$, which is just $\mathcal P_T$. As a result, $\mathcal P_T$ is defined and thus $\omid{M}<\infty$. So the claim is proved.
	
	%	By the definition of $\mathcal P_T$, one gets that
	%	\begin{eqnarray}
	%		\label{eq:prop:P_T}
	%		\myprob{\bs o'\in S_{\bs G'}, [S_{\bs G'}, \bs o']\in A} &=& \frac 1{\omid{M}} \omid{\sum_{v\in S_{\bs G}}T_{\bs G}(v, \bs o) \identity{A}[S_{\bs G},v]}
	%	\end{eqnarray}
	%	We claim that the RHS is equal to $\frac 1{\omid{M}} \myprob{[S_{\bs G}, \bs o]\in A}$. Assuming this holds, one obtains that $\probCond{[S_{\bs G'}, \bs o']\in A}{\bs o'\in S_{\bs G'}}$ is proportional to $\myprob{[S_{\bs G}, \bs o]\in A}$, which implies the result.
	%	
	%	Let $G_0$ be a network and $o,v\in V(G)$. Define
	%	\[
	%		g[G_0,v,o]:= \frac{\omidCond{\sum_{z\in S_{\bs G}} T_{\bs G}(z, \bs o) \identity{\{[S_{\bs G}, z, \bs o]=[G_0,v, o]\}}}{[S_{\bs G}, \bs o]=[G_0,o]}} {\card{\{z\in V(G_0): [G_0,z, o]=[G_0,v, o] \}}}.
	%	\]
	%	By $\sum_{v\in S_{\bs G}}T_{\bs G}(\bs o, v)=1$ a.s., one can get 
	%	
	%	It can be seen that %the RHS of~\eqref{eq:prop:P_T} is equal to
	%	\begin{eqnarray*}
	%		\text{RHS} &=& \frac 1{\omid{M}} \omid{\sum_{v\in {\bs G_0}}g[\bs G_0, v, \bs o_0] \identity{A}[\bs G_0,v]}\\
	%		&=& \frac 1{\omid{M}} \omid{\sum_{v\in {\bs G_0}}g[\bs G_0, \bs o_0, v] \identity{A}[\bs G_0,\bs o]}\\
	%		&=& \frac 1{\omid{M}} \omid{\identity{A}[\bs G_0,\bs o]}.\\
	%	\end{eqnarray*}
	%	This completes the proof as mentioned above.
\end{proof}

\begin{proof}[Proof of Corollary~\ref{cor:P_T-notDepend}]
	Lemma~\ref{lemma:uniExt} and Proposition~\ref{prop:P_T} imply the claim. %For the second claim, it can be seen that $\frac 12 (\mathcal P_T + \mathcal P'_T)$ is a unimodularization different from both $\mathcal P_T$ and $\mathcal P'_T$. The first claim implies that this unimodularization cannot be obtained by Theorem~\ref{thm:extensionT} even by changing the function $T$.
\end{proof}

%\begin{proof}[Proof of Proposition~\ref{prop:extensionTau}]
%	By letting $T_G(v, o)=\identity{\{\tau_G(v)=o\}}$, Theorem~\ref{thm:extensionT} implies the claims.
%\end{proof}

\section{A Construction Using Stable Transports}
\label{sec:construction}

In some results in this paper, the existence of specific objects %root-changes of networks 
are proved based on Theorem~\ref{thm:shiftCoupling}, including propositions~\ref{prop:conditioning} and~\ref{prop:extraHead} and theorems~\ref{thm:balancing} and~\ref{thm:uniExtGen}. However, Theorem~\ref{thm:shiftCoupling} does not help to construct such root-changes by looking only at a realization of the given networks. In this section, we present an algorithm to construct a balancing transport kernel as described in Theorem~\ref{thm:balancing}. Special cases of this algorithms will provide the desired constructions in the other results mentioned above.
%in Proposition~\ref{prop:conditioning} (in the unimodular case) and Proposition~\ref{prop:extraHead}. % It also works in the settings of propositions~\ref{prop:conditioning} and~\ref{prop:extraHead}. 
The algorithm is based on the one in~\cite{stable}, which is by itself based on~\cite{HoPe06}. It is a generalization of the Gale-Shapley stable matching algorithm for bipartite graphs~\cite{GaSh}. In fact, it is similar to the many-to-many stable matching algorithm. We should note that by the terms \textit{construction} and \textit{algorithm} we do not mean a computational algorithm, but an explicit definition using formulas that might be defined iteratively.

Fix a rooted network $(G, o)$ and measurable functions $w_i:\mathcal G_*\rightarrow\mathbb R^{\geq 0}$ for $i=1,2$. We will use two names \defstyle{sites} and \defstyle{centers} for the vertices and use Roman letters for centers for better readability. Given a measurable function $T:\mathcal G_{**}\rightarrow\mathbb R^{\geq 0}$, we say a site $x\in V(G)$ sends mass $T_G(x,\xi)$ to the center $\xi\in V(G)$. Here is an overview of Algorithm~\ref{alg:gale}. It will finally produce a measurable function $T:\mathcal G_{**}\rightarrow\mathbb R^{\geq 0}$ such that $T_G^+(\cdot) \leq w_2[G, \cdot]$ and $T_G^-(\cdot)\leq w_1[G, \cdot]$ (the goal is equality which will hold under some conditions). $T_G(x,\xi)$ will be defined as the mass $x$ \textit{applies} to $\xi$ minus the mass $\xi$ \textit{rejects} from $x$. The algorithm consists of infinitely many stages and each stage has two steps. At stage $n$, each site $x_0$ \textit{applies} to the closest possible centers with \textit{weight} $A_n(x_0,\cdot)$. A constraint is chosen for the applications, which is $0\leq A_n(x_0,\cdot)\leq w_2[G,\cdot]$. % (any positive constraint would work; e.g. $A_n\leq 1$). 
Then, each center $\xi_0$ \textit{rejects} some of the weights applied to $\xi_0$ if the sum of the incoming applications exceeds $w_2[G,\xi_0]$. The amount of rejection is denoted by $R_n(\cdot,\xi_0)$. %Finally, $T_G(\cdot,\cdot)$ will be the difference of the application function and the rejection function in the limit. 
The functions $A_n$ and $R_n$ at stage $n$ are chosen such that each site prefers to apply to the closest possible centers and each center prefers to reject (if necessary) the applications of the farthest possible sites.

\begin{algorithm}[Stable Transport]
	\label{alg:gale}
	Let $(G, o)$ be a given deterministic rooted network and $w_1,w_2:\mathcal G_*\rightarrow\mathbb R^{\geq 0}$ be measurable. Let $R_0(x,\xi)=0$ for all $x,\xi\in V(G)$. For each $n\geq 1$, stage $n$ consists of the following two steps.
	\begin{enumerate}[(i)]
		\item For each site $x_0$, define its \textit{application radius} at stage $n$ by
		\[
		a_n(x_0):=\min \{a\geq 0: \sum_{\xi\in N_a(x_0)} (w_2[G,\xi]-R_{n-1}(x_0,\xi)) \geq w_1[G,x_0]  \}.
		\]
		Define \textit{the $n$-th application function} by
		\[
		A_n(x_0,\xi):=
		\left\{ 
		\begin{array}{ll}
		w_2[G,\xi] & d(x_0,\xi)<a_n(x_0),\\
		cR_{n-1}(x_0,\xi)+(1-c) w_2[G,\xi] & d(x_0,\xi)=a_n(x_0),\\
		0 & d(x_0,\xi)>a_n(x_0),
		\end{array}
		\right.
		\]
		where in the case $a_n(x_0)<\infty$, $c=c_n(x_0)$ is chosen in $[0,1]$ such that
		\[
		\sum_{\xi\in V(G)} (A_n[x_0,\xi]-R_{n-1}(x_0,\xi)) = w_1[G,x_0].
		\]
		%provided that $a_n(x_0)<\infty$. Let $c=1$ if $a_n(x_0)=\infty$.
		\item For each center $\xi_0$, define its \textit{rejection radius} at stage $n$ by
		\[
		r_n(\xi_0):=\min\{r\geq 0: \sum_{x\in N_r(\xi_0)} A_n(x,\xi_0) \geq w_2[G,\xi_0] \}.
		\] 
		Define \textit{the $n$-the rejection function} by
		\[
		R_n(x,\xi_0):=
		\left\{ 
		\begin{array}{ll}
		0 & d(x,\xi_0)<r_n(\xi_0),\\
		c'A_n(x,\xi_0) & d(x,\xi_0)=r_n(\xi_0),\\
		A_n(x,\xi_0) & d(x,\xi_0)>r_n(\xi_0),
		\end{array}
		\right.
		\]
		where in the case  $r_n(\xi_0)<\infty$, $c'=c'_n(\xi_0)$ is chosen in $[0,1]$ such that
		\[
		\sum_{x\in V(G)} (A_n[x,\xi_0]-R_{n}(x,\xi_0)) = w_2[G,\xi_0].
		\]
		%provided that $r_n(\xi_0)<\infty$.	Let $c'=0$ if $r_n(\xi_0)=\infty$.
	\end{enumerate}
	Finally, define
	\[
	T_G(x,\xi):=\lim_{n\rightarrow\infty} A_n(x,\xi)-\lim_{n\rightarrow\infty} R_n(x,\xi).
	\]
\end{algorithm}

Here are some basic facts about this algorithm. The proofs are similar to~\cite{stable} and are skipped here for brevity. The sequences of functions $A_n, R_n$ and $a_n$ are non-decreasing w.r.t $n$ and $r_n$ is non-increasing. So, the limit function $T$ is well defined. Moreover, $T_G^+(\cdot)\leq w_1[G,\cdot]$ and $T_G^-(\cdot)\leq w_2[G,\cdot]$. Call a site $x_0$ \defstyle{exhausted} if $T_G^+(x_0)=w_1[G,x_0]$. Similarly, a center $\xi_0$ is \defstyle{sated} if $T_G^-(\xi_0)=w_2[G,\xi_0]$. It is shown below that $T$ is \defstyle{stable} in a sense similar to~\cite{stable} (and many-to-many stable matchings) defined as follows: There is no site $x_0$ and center $\xi_0$ such that both \defstyle{desire} each other, where site $x_0$ desires  center $\xi_0$ if  $T_G(x_0,\xi_0)<w_2[G,\xi_0]$ and either $x_0$ is unexhausted or $T_G(x_0,\xi_1)>0$ for some farther center $\xi_1$. Similarly, center $\xi_0$ desires site $x_0$ if  {$T_G(x_0,\xi_0)<w_2[G,\xi_0]$} and either $\xi_0$ is unsated or $T_G(x_1,\xi_0)>0$ for some farther site $x_1$. Stronger than stability, the following holds. %Moreover, it is true that if $x_0$ desires $\xi_0$, then $\xi_0$ is sated.

\begin{lemma}
	\label{lemma:stable}
	In Algorithm~\ref{alg:gale}, if a site $x_0$ desires a center $\xi_0$ (defined above), then $\xi_0$ is sated and doesn't desire $x_0$. Therefore, $T$ is stable.
\end{lemma}

\begin{proof}
	By definition, either $x_0$ is unexhausted or $T_G(x_0,\xi_1)>0$ for some farther center $\xi_1$. In both cases, $x_0$ has applied to some center farther than $\xi_1$ at some stage. The definition of $A_n$ implies that $A_n(x_0,\xi_0)=w_2[G,\xi_0]$ for large enough $n$. Therefore, by $T_G(x_0,\xi_0)<w_2[G,\xi_0]$, $\xi_0$ has rejected a positive fraction of the application of $x_0$ at some stage. By the definition of the rejection function, $\xi_0$ is sated from that stage on. Moreover, $\xi_0$ has fully rejected the applications of the sites farther than $x_0$. So, $\xi_0$ doesn't desire $x_0$ and the claim is proved.
\end{proof}

\begin{lemma}
	\label{lemma:exhausted}
	In Algorithm~\ref{alg:gale}, if there is an unexhausted site, then all centers are sated and vice versa.
\end{lemma}
\begin{proof}
	Assume $x_0$ is an unexhausted site and $\xi_0$ is an unsated center. Since $\xi_0$ is unsated, one obtains $T_G(x_0,\xi_0)\leq T_G^-(\xi_0)<w_2[G,\xi_0]$. Therefore, $x_0$ desires $\xi_0$ by definition. This contradicts Lemma~\ref{lemma:stable}. %By stability of $T$, $\xi_0$ doesn't desire $x_0$. It follows that $T_G(x_0,\xi_0)=w_1[G,x_0]$. This implies that $T_G^+(x_0)=w_1[G,x_0]$; i.e. $x_0$ is exhausted, which is a contradiction.
\end{proof}

%$\sum w_2[G,\xi]<\infty$, where the sum is over unsated centers $\xi$. Similarly, if there is an unsated center, then $\sum w_1[G,x]<\infty$, where the sum is over unexhausted sites $x$.

\begin{theorem}[Construction of a Balancing Transport Kernel]
	\label{thm:stable}
	Let $[\bs G, \bs o]$ be a unimodular network and $w_i:\mathcal G_*\rightarrow \mathbb R^{\geq 0}$ be measurable functions for $i=1,2$. If 
	\[
	\omidCond{w_1[\bs G, \bs o]}{I}=\omidCond{w_2[\bs G, \bs o]}{I}<\infty, \quad a.s.,
	\]
	then the function $T$ constructed in Algorithm~\ref{alg:gale} satisfies the claims of Theorem~\ref{thm:balancing}; i.e. $T_{\bs G}^+(\cdot)=w_1[\bs G, \cdot]$ and $T_{\bs G}^-(\cdot)=w_2[\bs G, \cdot]$ a.s.
\end{theorem}

Note that the condition $\omidCond{w_1[\bs G, \bs o]}{I}=\omidCond{w_2[\bs G, \bs o]}{I}$ a.s. is also necessary (see Theorem~\ref{thm:balancing}).

\begin{proof}
	First, it is easy to see that $T$ defines a measurable function on $\mathcal G_{**}$.	
	We should prove there is no unexhausted site and no unsated center a.s. 
	%	Suppose there is an unexhausted site with positive probability. 
	Let $A$ be the event that there is an unexhausted site. If $\myprob{A}>0$, then by conditioning on $A$, one may assume $\myprob{A}=1$ (notice that $A\in I$ and thus the assumptions are not changed after conditioning on $A$). Therefore, by Lemma~\ref{lemma:exhausted}, there is no unsated centers a.s. Lemma~\ref{lemma:happensAtRoot} implies that $\bs o$ is unexhausted with positive probability but $\bs o$ is sated a.s. Equivalently, $T_{\bs G}^+(\bs o)<w_1[\bs G, \bs o]$ with positive probability but $T_{\bs G}^-(\bs o)=w_2[\bs G, \bs o]$ a.s. It follows that 
	\[\omid{T_{\bs G}^+(\bs o)}<\omid{w_1[\bs G, \bs o]}=\omid{w_2[\bs G, \bs o]} = \omid{T_{\bs G}^-(\bs o)}.\]
	This contradicts~\eqref{eq:unimodular}. Therefore, all sites are exhausted a.s. One can prove similarly that all centers are sated a.s. So, the proof is complete.
\end{proof}

The following is an application of Theorem~\ref{thm:stable} to Proposition~\ref{prop:conditioning} (see also Remark~\ref{rem:constructionOfIntensity} below). Unimodularity is a crucial assumption to ensure that Algorithm~\ref{alg:gale} works here. 
{The author is not aware of any general construction for the non-unimodular case of Proposition~\ref{prop:conditioning}.}
%As far as the author's knowledge, no general construction is known for the non-unimodular case of Proposition~\ref{prop:conditioning}.

\begin{corollary}
	\label{cor:conditioningConstruction}
	In the setting of Proposition~\ref{prop:conditioning}, assume $[\bs G', \bs o']$ can be obtained from $[\bs G, \bs o]$ by a root-change. If $[\bs G, \bs o]$ is unimodular, then such root-change is obtained by the function $T$ constructed in Algorithm~\ref{alg:gale} for $w_1[G,o]:=1$ and $w_2[G,o]:=\frac 1{\myprob{\bs o\in S_{\bs G}}}\identity{\{o\in S_{G} \}}$.
	%let $w_1[\bs G, \bs o]:=\probCond{\bs o\in S_{\bs G}}{I}$ and $w_2[\bs G, \bs o]:=\identity{\{\bs o\in S_{\bs G} \}}$. If $[\bs G, \bs o]$ is unimodular and $\probCond{\bs o\in S_{\bs G}}{I}$ is constant, then  the function $T$ constructed in Algorithm~\ref{alg:gale} for $w_1$ and $w_2$ gives a root-change that 
\end{corollary}

\begin{proof}
	Let $p:=\myprob{\bs o\in S_{\bs G}}$. By Proposition~\ref{prop:conditioning}, $\probCond{\bs o\in S_{\bs G}}{I}=p$ a.s. It follows that $\omidCond{w_2[\bs G, \bs o]}{I}=1$ a.s.
	Now, Theorem~\ref{thm:stable} implies that $T_{\bs G}^+(\bs o)=1$ and $T_{\bs G}^-(\bs o)=\frac 1p \identity{\{\bs o\in S_{\bs G} \}}$ a.s. Therefore, Lemma~\ref{lemma:unimodularBias} implies that applying the root-change by kernel $T$ to $[\bs G, \bs o]$ gives $[\bs G', \bs o']$. So, the claim is proved.
\end{proof}

By Corollary~\ref{cor:conditioningConstruction} and the proof of Proposition~\ref{prop:extraHead}, the following corollary is readily obtained.

\begin{corollary}[Construction of an Extra Head Scheme]
	In the setting of Proposition~\ref{prop:extraHead}, an extra head scheme is obtained by the function $T$ constructed in Algorithm~\ref{alg:gale} for $w_1[G,o]:=1$ and $w_2[G,o]:=\frac 1p\identity{\{m(o)=1\}}$. 
\end{corollary}

\begin{corollary}[Construction of a Unimodularization of an Extension]
	\label{cor:consExt}
	Let $([\bs G, \bs o], S)$ be a proper extension of a unimodular network $[\bs G_0, \bs o_0]$. If $[\bs G]$ can be unimodularized, then there exists a constant $\lambda_G$ for each non-rooted network $G$ % measurable function $\lambda:\mathcal G\rightarrow\mathbb R^{+}$ 
	such that the root-change corresponding to the function $T$ constructed in Algorithm~\ref{alg:gale} for $w_1[G,v]:=1$ and $w_2[G,v]:=\frac 1{\lambda_G} \identity{\{v\in S_G\}}$ satisfies the assumptions of Theorem~\ref{thm:extensionT}. Therefore, $\mathcal P'_T$ is the strong unimodularization of $[\bs G]$. % in the strong sense.
\end{corollary}

\begin{proof}
	%Although $[\bs G, \bs o]$ is not unimodular here, by the assumption of existence of a unimodularization, one can still use Theorem~\ref{thm:stable}. 
	Let $[\bs G', \bs o']$ be a strong unimodularization of $[\bs G]$ %in the strong sense 
	and let $\lambda$ satisfy $\lambda_{\bs G'}=\probCond{\bs o'\in S_{\bs G'}}{I}$. By Theorem~\ref{thm:stable}, $T_{\bs G'}(\cdot,\cdot)$ balances between $w_1[\bs G', \cdot]$ and $w_2[\bs G', \cdot]$ a.s. as defined in Theorem~\ref{thm:balancing}. %By the definition of $[\bs G', \bs o']$, the same holds for $\bs G$. 
	Now, the proof of Theorem~\ref{thm:uniExtGen} shows that $T$ satisfies the assumptions of Theorem~\ref{thm:extensionT} and $\mathcal P'_T$ is defined.
\end{proof}

\begin{remark}
	\label{rem:constructionOfIntensity}
	One may ask how to construct $\lambda_{\bs G}$ in the above corollary by looking only at a realization of $[\bs G, \bs o]$.
	If $[\bs G', \bs o']$ is a unimodularization of $[\bs G]$, then $\lambda$ is the sample intensity of $S$ in $\bs G'$ defined in Example~\ref{ex:subnetwork}. %; i.e. it satisfies $\lambda_{\bs G'}= \probCond{\bs o'\in S_{\bs G'}}{I}$. 
	One may also ask the same question in Corollary~\ref{cor:conditioningConstruction} on how to construct $\myprob{\bs o\in S_{\bs G}}$.
		Note that averaging on a large ball like $\card{(S_{\bs G}\cap N_r(\bs o))}/\card{N_r(\bs o)}$ and taking limit does not work in general. In fact, it works only for the so called \textit{amenable} unimodular networks~\cite{processes}.
	For general unimodular networks, one construction for the sample intensity can be done by \textit{frequency of visits} to $S_{\bs G}$ of the \textit{delayed simple random walk} in $[\bs G]$ (see~\cite{indistinguish} and~\cite{processes} for the details). Another construction is the following. 
	
	{In Corollary~\ref{cor:consExt}, replace $\lambda_G$ by an arbitrary constant $\lambda>0$.}
	Then, it can be seen that given any $(G,v)$, the value $T_G^+(v)$ %for the function $T$ in Corollary~\ref{cor:consExt} 
	is non-increasing in terms of $\lambda$ (see~\cite{thesis} and also~\cite{stable} and~\cite{HoPe06}). Then, one can let $\lambda_G$ be the supremum value of $\lambda$ such that $T_G^+(v)=1$ for all $v\in V(G)$. It can be proved that this construction works in Corollary~\ref{cor:consExt}. The proof is similar to the arguments in~\cite{stable} and~\cite{HoPe06} and is skipped for brevity.
\end{remark}

\section{Proofs}
\label{sec:proofs}
%\mar{\ali{update reference to proofs of Sec\ref{sec:extAndTrans}}}
This sections is devoted to the proofs of some of the results of Section~\ref{sec:shiftCoupling}. %\ali{For the proofs, the following definition is needed.}

\begin{definition}
%	\mar{%tare: update the notions and notations with theory of Markov processes.\\ 
%	}
	%ater: check problems caused by moving this definition.
	Let $\mu$ be a measure (not necessarily a probability measure)  on $\mathcal G_*$ and $T:\mathcal G_{**}\rightarrow\mathbb R^{\geq 0}$ be a measurable function. %an invariant weighted transport (which is equivalent to a measurable function from  $\mathcal G_{**}$ to $\mathbb R^{\geq 0}$). 
	Define the measures $T^{\uparrow}\mu$ and $\transport{T}{\mu}$ on $\mathcal G_{**}$ and $\mathcal G_*$ respectively by
	\begin{eqnarray*}
	%\label{eq:Tup}
	(T^{\uparrow}\mu)(A) &:=&\int_{\mathcal G_*}\sum_{v\in V(G)}T_G(o,v)\identity{A}[G,o,v]d\mu ([G,o]),\\
	%\label{eq:Tright}
	(\transport{T}{\mu})(B) &:=&  \int_{\mathcal G_*}\sum_{v\in V(G)}T_G(o,v)\identity{B}[G,v]d\mu ([G,o]),
	\end{eqnarray*}
	for measurable subsets $A\subseteq\mathcal G_{**}$ and $B\subseteq \mathcal G_*$.
	It can be seen that $\pi_{2*}T^{\uparrow}\mu = \transport{T}{\mu}$.
%	\begin{eqnarray*}
%		\pi_{2*}T^{\uparrow}\mu &=& \transport{T}{\mu}.
%	\end{eqnarray*}
	Moreover, if $T_G^+(o)=1$ for $\mu$-a.e. $[G,o]$, then $\pi_{1*}T^{\uparrow}\mu = \mu$.
%	\begin{eqnarray*}
%		\pi_{1*}T^{\uparrow}\mu &=& \mu.
%	\end{eqnarray*}		
	If in addition, $\mu$ is a probability measure, % and $T_G^+(o)=1$ for $\mu$-a.e. $[G,o]$, 
	then $\transport{T}{\mu}$ is just the root-change of $\mu$ by kernel $T$ as in Definition~\ref{def:rootChange}. 
	%In this case, $T^{\uparrow}\mu$ can be interpreted as choosing a random rooted network $[\bs G, \bs o]$ with distribution $\mu$ and then, conditioned on $[\bs G, \bs o]$, choosing a second root in $V(\bs G)$ with distribution $T_{\bs G}(\bs o, \cdot)$. 
	It is also worthy to mention that when $T^+(\cdot)$ is always $1$, there is a Markov kernel on $\mathcal G_*$ that transports $\mu$ to $\transport{T}{\mu}$, which is defined by $T'([G,o], [G',o']):=\sum_{v\in V(G)}T_G(o,v)\identity{\{[G,v]=[G',o']\}}$.
	
	%	this can be interpreted as choosing a random rooted network with distribution $\mu$ and then choosing the second root with distribution $T_G(o,\cdot)$. In other words $\pi_{1*}(T^{\uparrow}\mu)=\mu$ and conditional to the network and the first root to be $(G,o)$, the distribution of the second root is $T_G(o,\cdot)$.
\end{definition}

%It can be seen that in the case $T$ is Markovian, $T$ induces a Markov kernel $\mathcal G_*$, which transports $\mu$ to $\transport{T}{\mu}$ for each $\mu$.
%\mar{Is this intuition good?}
%These are measures that can be considered as lifts of $\mu$ to $\mathcal G_{**}$ and $\mathcal G_*$ by $T$. 

\begin{proof}[Proof of Proposition~\ref{prop:nonstandard}]
	Assume  $(\mathcal G, J)$ is a standard Borel space. Let $\mu$ be an arbitrary extremal unimodular probability measure on $\mathcal G_*$ and let $\nu:=\pi_*\mu$. Since $\mu$ is extremal, one gets $\nu(A)\in\{0,1\}$ for any event $A\in J$. The assumption of standardness of $(\mathcal G, J)$ implies that $\nu$ is concentrated on one atom. Therefore, there should be a network $G$ such that $\mu$ is concentrated on $\{[G,v]:v\in V(G)\}$. But this is clearly false for general $\mu$ (see for instance the example in Lemma~\ref{lemma:ergodic}).
\end{proof}

\begin{proof}[Proof of Lemma~\ref{lemma:biasG}]
	%\mar{ater:skip the proof?}
	%Let $\mathcal P$ be the distribution of $[\bs G, \bs o]$ and $[\bs G', \bs o']$ be a random rooted network with distribution $\mathcal P':= w\mathcal P$. 
	%\mar{\ali{correct numbering}}
	Denote by $[\bs G', \bs o']$ the random rooted network obtained by biasing the probability measure by $w$.
	Let $B\in J$ be an event in $\mathcal G$ and $A:=\pi^{-1}(B)$. One has
	\begin{eqnarray*}
	\mathbb P\left[[\bs G']\in B\right] &=& \mathbb P\left[[\bs G', \bs o']\in A\right]\\
	&=& \frac 1c \omid{w[\bs G, \bs o]\identity{A}[\bs G, \bs o]}\\
	&=& \frac 1c \omid{\omidCond{w[\bs G, \bs o]}{I} \identity{A}[\bs G, \bs o]}\\
	&=& \frac 1c \omid{\omidCond{w[\bs G, \bs o]}{I}\identity{B}[\bs G]},
	\end{eqnarray*}
	where $c:=\omid{w[\bs G, \bs o]}$ and in the third equation we have used the fact $A\in I$. Now, the claim is obtained by noting that $\omidCond{w[\bs G, \bs o]}{I}$ is a function of $[\bs G]$.
\end{proof}

\begin{lemma}
	\label{lemma:implicationsCDR}
	Conditions~\ref{conditionR}, \ref{conditionC} and~\ref{conditionD} in Theorem~\ref{thm:implications} are equivalent.
\end{lemma}

\begin{proof} 
	Let $\mathcal P_1$ and $\mathcal P_2$ be the distributions of $[\bs G_1,\bs o_1]$ and $[\bs G_2, \bs o_2]$ respectively.
	
	\ref{conditionR}$\Rightarrow$ \ref{conditionD}.
	Consider a root-change by kernel $T$ such that $\transport{T}{\mathcal P_1}=\mathcal P_2$. Given $[\bs G_1, \bs o_1]$, choose a second root with distribution $T_{\bs G_1}(\bs o_1,\cdot)$. To be more precise, a random doubly-rooted network with distribution $T^{\uparrow}\mathcal P_1$ (Definition~\ref{def:rootChange}) has the desired properties.

	\ref{conditionD}$\Rightarrow$ \ref{conditionC}.
	Let $[\bs G, \bs o, \bs o']$ be such a random doubly-rooted network. Then, the random rooted networks $[\bs G, \bs o]$ and $[\bs G, \bs o']$ provide the desired coupling. More precisely, the desired coupling is obtained by pushing forward the distribution of $[\bs G, \bs o, \bs o']$ by the map $[G,o,o']\mapsto ([G,o], [G,o'])$, which is a well defined measurable function on $\mathcal G_{**}$.
	
	\ref{conditionC}$\Rightarrow$ \ref{conditionR}.
	Let $\mu$ be such a probability measure on $\mathcal G_{*}\times\mathcal G_*$ as assumed. Fix a rooted network $(G_1,o_1)$ and let $\mu_{(G_1,o_1)}$ be the conditional distribution of the second rooted network given that the first rooted network is $[G_1,o_1]$. Note that $\mu_{(G_1,o_1)}$ is defined and is supported on %$A:=\{[G_2,o_2]: [G_2]=[G_1]\}$ 
	$A:=\{[G_1,v]: v\in V(G_1)\}$
	for $\mathcal P_1$-a.e. $[G_1,o_1]$. In the (zero-probability) cases where this doesn't hold, let $\mu_{(G_1,o_1)}$ be concentrated on $[G_1,o_1]$. Note that $A$ is a countable set. % since there is a surjective map $V(G_1)\rightarrow A$ defined by $v\mapsto [G_1,v]$.
	
	For $[G_2,o_2]\in A$, let $S_{[G_2,o_2]}\subseteq V(G_1)$ be the set of the closes vertices $v$ to $o_1$ such that $[G_1,v]=[G_2,o_2]$. This set is a finite subset of $V(G_1)$. Finally, from the measure $\mu_{(G_1,o_1)}$ on $A$, one can construct a measure $T_{G_1}(o_1,\cdot)$ on $V(G_1)$ defined by %transporting the mass at each $[G_2,o_2]\in A$ uniformly in $S_{[G_2,o_2]}$; i.e.
	\[
	T_{G_1}(o_1,v):=\sum_{[G_2,v_2]\in A}\frac 1{\card{S_{[G_2,v_2]}}} \mu_{(G_1,o_1)}([G_2,v_2]) \identity{S_{[G_2,v_2]}}(v) 
	%\frac 1{\card{S_{[G_1,v]}}} \mu_{(G_1,o_1)}([G_1,v])
	\]
	for $v\in V(G_1)$.
	It is easy to see that $T$ is an invariant transport kernel and $T_{G_1}^+(o_1)=1$. Moreover, by choosing a new root in $V(G_1)$ with distribution $T_{G_1}(o_1,\cdot)$, the resulting network has distribution $\mu_{(G_1,o_1)}$. By the definition of $\mu_{(G_1,o_1)}$ and choosing $(G_1,o_1)$ randomly with distribution $\mathcal P_1$, one gets that % Therefore, one gets $T^{\uparrow}\mathcal P_1=\mu$ and thus
	$\transport{T}{\mathcal P_1}=\mathcal P_2$. So, $T$ gives the desired root-change.
	
\end{proof}

\begin{proof}[Proof of Lemma~\ref{lemma:rootChange}]
	By part~\ref{conditionR}$\Rightarrow$\ref{conditionC} of Theorem~\ref{thm:implications} (proved in Lemma~\ref{lemma:implicationsCDR} above), there is a coupling of $[\bs G, \bs o]$ and $[\bs G, \bs o']$ supported on $\{([G,o],[G',o']):[G]=[G']\}\subseteq\mathcal G_*\times\mathcal G_*$. The latter also holds if one swaps $[\bs G, \bs o]$ and $[\bs G', \bs o]$. Therefore, by part~\ref{conditionC}$\Rightarrow$\ref{conditionR} of Theorem~\ref{thm:implications} (proved in Lemma~\ref{lemma:implicationsCDR} above), $[\bs G, \bs o]$ is a root-change of $[\bs G', \bs o']$.

\end{proof}

\begin{proof}[Proof of Theorem~\ref{thm:shiftCoupling}]
	Let $P_1$ and $P'_1$ be the distributions of $[\bs G,\bs o]$ and $[\bs G', \bs o']$ respectively. 
	
	($\Rightarrow$). Assume $[\bs G',\bs o']$ can be obtained from $[\bs G, \bs o]$ by a root-change. Equation~\eqref{eq:rootChange} easily implies that $\myprob{[\bs G', \bs o']\in A}= \myprob{[\bs G, \bs o]\in A}$ for any invariant event $A$. Equivalently, $\mathcal P[A]=\mathcal P'[A]$ for any $A\in I$, which is the desired property.
	
	($\Leftarrow$, \textit{First Proof}). 
	The proof mimics that of~\cite{Th96}. Here is a summary of the proof. The idea is to find two root-changes for $[\bs G_1, \bs o_1]$ and $[\bs G_2, \bs o_2]$ such that the resulting random rooted networks have the same distribution. Then, one can combine them to find the desired root-change. We start with an arbitrary root-change such that every vertex has positive probability to be chosen. Then, update it step by step as will be described. However, in the next steps probability measures will be replaced by finite measures.
	
	For a network $G$ and $o,v\in V(G)$, let 
	\[
	S[G,o,v]:=\left\{
	\begin{array}{ll}
	\frac 1{\card{V(G)}},& \card{V(G)}<\infty\\
	\frac 1{2^{r+1}\card{\partial N_r(o)}}, & \card{V(G)}=\infty
	\end{array}
	\right.
	\]
	where $r:=d(o,v)$ and $\partial N_r(o)$ is the (internal) boundary of the ball; i.e. the set of vertices with distance $r$ from $o$. It can be seen that $S$ is well-defined and measurable on $\mathcal G_{**}$. All we need from $S$ is the following property: For every network $G$ and $o,v\in V(G)$,
	%\mar{I think $S_G^+(o)\leq 1$ is enough.}
	\begin{eqnarray}
	\label{eq:shiftCoupling-T}
	\left\{
	\begin{array}{l}
	S_G^+(o)=1,\\
	S_G(o,v)>0.
	\end{array}
	\right.
	\end{eqnarray}
	
	Starting from $P_1$ and $P'_1$, construct the sequences of finite measures $Q_n,Q'_n, P_n,P'_n$ and $\lambda_n$ for $n\geq 1$ as follows. The first two are one $\mathcal G_{**}$ and the other three on $\mathcal G_*$. Here, the symbol $\wedge$ is used for the minimum of measures. %\mar{define minimum of measures?}
	\begin{itemize}
		\item  $\lambda_n:=(\transport{S}{P_n})\wedge (\transport{S}{P'_n})$.
		\item $Q_n$ is the probability measure constructed in Lemma~\ref{lem:thorisson} for $\lambda_n$, $S^{\uparrow}P_n$ and $i=2$. $Q'_n$ is defined similarly by Lemma~\ref{lem:thorisson} for $\lambda_n$, $S^{\uparrow}P'_n$ and $i=2$.
		\item $P_{n+1}:=P_n-\pi_{1*}Q_n$ and $P'_{n+1}:=P'_n-\pi_{1*}Q'_n$.
	\end{itemize}
	
	By the definition of $Q_n$ and $Q'_n$ and Lemma~\ref{lem:thorisson}, one has
	\begin{eqnarray}
	\label{eq:shiftCoupling-pi2Qn}
	\left\{
	\begin{array}{rcl}
	Q_n\leq S^{\uparrow}P_n,&& \pi_{2*}Q_n=\lambda_n,\\
	Q'_n\leq S^{\uparrow}P'_n,&& \pi_{2*}Q'_n=\lambda_n.
	\end{array}
	\right.
	\end{eqnarray}
	
	Since $\pi_{1*}S^{\uparrow}P_n=P_n$ and $\pi_{1*}S^{\uparrow}P'_n=P'_n$, this implies inductively that all above measures are non-negative (and justifies validity of using Lemma~\ref{lem:thorisson} inductively).
	Define
	\begin{eqnarray*}
		P_{\infty}:=\lim_n P_n, && Q:=\sum_{n=1}^{\infty}Q_n,\\
		P'_{\infty}:=\lim_n P'_n, && Q':=\sum_{n=1}^{\infty}Q'_n.
	\end{eqnarray*}
	
	The limits are well defined since $P_n$ and $P'_n$ are decreasing sequences and the sums of $\|Q_n\|$ and $\|Q'_n\|$ over $n$ are convergent, where the symbol $\|\cdot\|$ is used for the total mass of a measure (note that $\|Q_n\|=\|P_n\|-\|P_{n+1}\|$). This also shows that $Q$ is a finite measure.
	Now, one has
	\begin{eqnarray}
	\label{eq:shiftCoupling-pi1Q}
	\left\{
	\begin{array}{lcl}
		\pi_{1*}Q &= \sum_n (P_n-P_{n+1})=& P_1-P_{\infty},\\
		%\label{eq:shiftCoupling-pi1Q'}
		\pi_{1*}Q' &= \sum_n (P'_n-P'_{n+1})=& P'_1-P'_{\infty}
	\end{array}
	\right.
	\end{eqnarray}
	and
	\begin{equation}
	\label{eq:shiftCoupling-pi2QQ'}
	\pi_{2*}Q = \pi_{2*}Q'.
	\end{equation}
	
	Let $\lambda_{\infty}:=(\transport{S}{P_{\infty}})\wedge (\transport{S}{P'_{\infty}})$.
	By $P_{\infty}\leq P_n$ and $P'_{\infty}\leq P'_n$, it is clear that $\lambda_{\infty}\leq \lambda_n$. Therefore, by~\eqref{eq:shiftCoupling-pi2Qn} we get
	$
	\|\lambda_{\infty}\|\leq \|\lambda_n\|=\|Q_n\|
	$
	for every $n$. The sum of the right hand side over $n$ is convergent (bounded by $\|Q\|$) and so $\|\lambda_{\infty}\|=0$. Therefore, $\lambda_{\infty}=0$. This means that the measures $\transport{S}{P_{\infty}}$ and $\transport{S}{P'_{\infty}}$ are mutually singular; i.e. there is an event $A\subseteq \mathcal G_*$ such that
	\begin{eqnarray}
	\label{eq:shiftCoupling-TPinfty}
	\left\{
	\begin{array}{lcl}
	\transport{S}{P_{\infty}}(A^c) &=& 0,\\
	\transport{S}{P'_{\infty}}(A) &=& 0,
	\end{array}
	\right.
	\end{eqnarray}
	
	Consider the event $B:=\{[G,o]:\exists v\in V(G): [G,v]\in A \}=\pi^{-1}(\pi(A))$ in $\mathcal G_*$. By~\eqref{eq:shiftCoupling-TPinfty}, \eqref{eq:shiftCoupling-T} and the definition of $\transport{S}{P_{\infty}}$ and $\transport{S}{P'_{\infty}}$, one gets
	\begin{equation}
	\label{eq:shiftCoupling-PinftyB}
	\left\{
	\begin{array}{lcl}
	P_{\infty}(B^c) &=& 0,\\
	P'_{\infty}(B) &=& 0.
	\end{array}
	\right.
	\end{equation}
	
	By part~\ref{conditionD}$\Rightarrow$\ref{conditionF} of Theorem~\ref{thm:implications} (proved in Lemma~\ref{lemma:implicationsCDR} above), the measures $\pi_{1*}Q$ and $\pi_{2*}Q$ agree on $I$. The same holds for $Q'$ (and any arbitrary measure on $\mathcal G_{**}$). Therefore,~\eqref{eq:shiftCoupling-pi2QQ'} gives that $\pi_{1*}Q$ and $\pi_{1*}Q'$ agree on $I$. By~\eqref{eq:shiftCoupling-pi1Q} and the assumption that $P_1$ and $P'_1$ agree on $I$, one gets that $P_{\infty}$ and $P'_{\infty}$ also agree on $I$. Since $B$ is clearly an invariant event, one obtains $P_{\infty}(B)=P'_{\infty}(B)$. Now, \eqref{eq:shiftCoupling-PinftyB} readily implies
	$
	P_{\infty}=P'_{\infty}=0
	$.
	Now, one has
	\begin{equation*}
	%\left\{
	\begin{array}{lcl}
	\pi_{1*}Q&=&P_1,\\
	\pi_{1*}Q'&=&P'_1,\\
	\pi_{2*}Q&=&\pi_{2*}Q'.
	\end{array}
	%\right.
	\end{equation*}
	
	As a result, $Q$ and $Q'$ are probability measures. By part \ref{conditionD}$\Rightarrow$\ref{conditionR} of Theorem~\ref{thm:implications} (proved in Lemma~\ref{lemma:implicationsCDR} above), one finds a root-change, say by kernel $T$, that transports $P_1$ to $\alpha:=\pi_{2*}Q=\pi_{2*}Q'$. %(note that this theorem will not be used in the equivalence of conditions~\ref{conditionC}, \ref{conditionD} and \ref{conditionR} in the proof of Theorem~\ref{thm:implications} and in the proof of Lemma~\ref{lemma:rootChange}). 
	Similarly, by the same argument and Lemma~\ref{lemma:rootChange}, one finds a root-change, say by kernel $T'$, that transports $\alpha$ to $\mathcal P_2$. Now, let $t$ be the composition of $T$ and $T'$ defined by
	$
	t_G(o,v):= \sum_{z\in V(G)} T_G(o,z)T'_G(z,v).
	$
	It can be seen that $t$ gives a root-change of $P_1$ (as in Definition~\ref{def:rootChange}) and $\transport{t}{P_1}=P'_1$. Therefore, by Definition~\ref{def:rootChange}, $[\bs G', \bs o']$ can be obtained from $[\bs G, \bs o]$ by  the root-change by kernel $t$, which completes the proof.

		($\Leftarrow$, \textit{Second Proof}). Let $R$ be the equivalence relation on $\mathcal G_*$ in which $[G_1, o_1]$ is $R$-related to $[G_2,o_2]$ if and only if $[G_1]=[G_2]$. Following the definitions in Subsection~\ref{subsec:shiftCoupling:notes}, it can be seen that $R$ is a countable Borel equivalence relation. %Since each $R$-equivalence class is countable, $R$ is a countable equivalence relation. 
		Therefore, by Theorem~1 of~\cite{FeMo77}, there is a countable group $H$ consisting of Borel isomorphisms of $\mathcal G_*$ that generates $R$ in the sense that
		\[
		xRy\Leftrightarrow \exists h\in H: y=h(x).
		\]
		Endow $H$ with the discrete topology. It can be seen that the invariant sigma-field under the action of $H$ is equal to the invariant sigma-field $I$ in Definition~\ref{def:J}. So, the assumption gives that the distributions of $[\bs G_1, \bs o_1]$ and $[\bs G_2, \bs o_2]$ agree on the $H$-invariant sigma-field. Thus, by Theorem~1 of~\cite{Th96}, there is a random element $F$ of $H$ such that $F[\bs G_1,\bs o_1]$ has the same distribution as $[\bs G_2, \bs o_2]$. This provides a coupling of $\mathcal P$ and $\mathcal P'$ that satisfied Condition~\ref{conditionC} of Definition~\ref{def:same}. Therefore, by part \ref{conditionC}$\Rightarrow$\ref{conditionR} of Theorem~\ref{thm:implications} (proved in Lemma~\ref{lemma:implicationsCDR} above), $[\bs G_2, \bs o_2]$ can be obtained from $[\bs G_1, \bs o_1]$ by a root-change and the claim is proved.
\end{proof}

The following lemma %is analogous to Lemma~1 in~\cite{Th96} and 
is used in the proof of Theorem~\ref{thm:shiftCoupling} above.
Note that all measures are assumed to be non-signed in this paper.
\begin{lemma}
	\label{lem:thorisson}
	%\mar{add assumption: all measures are positive}
	Let $P$ and $Q$ be finite measures on $\mathcal G_*$ and $\mathcal G_{**}$ respectively and $i\in \{1,2\}$. If $\pi_{i*}Q\geq P$, then there is a measure $Q'\leq Q$ such that $\pi_{i*}Q'= P$.
\end{lemma}

\begin{proof}
	The claim is a direct consequence of Lemma~1 in~\cite{Th96}. 
\end{proof}

\begin{proof}[Proof of Theorem~\ref{thm:implications}]
	%Let $\mathcal P_1$ and $\mathcal P_2$ be the distributions of $[\bs G_1,\bs o_1]$ and $[\bs G_2, \bs o_2]$ respectively.
	According to Lemma~\ref{lemma:implicationsCDR} and Theorem~\ref{thm:shiftCoupling} proved above, the only remaining part is 
	\ref{conditionR}$\Rightarrow$ \ref{conditionB}, which is trivial.
%	 it remains to prove the following parts.
%	
%	\ref{conditionR}$\Rightarrow$ \ref{conditionB}.
%	Trivial by letting the bias function be equal to 1.
%	
%	\ref{conditionF}$\Rightarrow$ \ref{conditionR}. This part is just Theorem~\ref{thm:shiftCoupling} proved above.
%	
%	\ref{conditionR}$\Rightarrow$ \ref{conditionF}.
%	Consider a root-change by kernel $T$ such that $\transport{T}{\mathcal P_1}=\mathcal P_2$.
%	%Let $T$ be a Markovian invariant transport such that $\transport{T}{\mathcal P_1}=\mathcal P_2$. 
%	For an invariant event $A\in I$, one can easily show by the definition of $\transport{T}{\mathcal P_1}$ that $(\transport{T}{\mathcal P_1})(A)=\mathcal P_1(A)$. This shows that $\mathcal P_1$ and $\mathcal P_2$ agree on the invariant sigma-field, which implies the claim.
\end{proof}

\begin{proof}[Proof of Lemma~\ref{lemma:ergodic}]
	By the natural coupling of the two random rooted networks, one may assume $V(\bs G)=V(\bs G')$ and $\bs o=\bs o'$.
	%	Let $A\in I$ be an invariant event. We claim that 
	%	\begin{equation}
	%		\label{eq:ergodic:1}
	%		\omidCond{\identity{A}[\bs G',\bs o]}{[\bs G, \bs o]}\in \{0,1\} \quad a.s.
	%	\end{equation}
	First, assume~\eqref{eq:ergodic:1} is proved and $[\bs G, \bs o]$ is extremal and infinite a.s. The left hand side of~\eqref{eq:ergodic:1} is an invariant function of $[\bs G, \bs o]$. Therefore, by extremality, it is essentially constant, hence, by~\eqref{eq:ergodic:1}, it is either 1 a.s. or 0 a.s. It follows that $\myprob{[\bs G', \bs o]\in A}\in \{0,1\}$ and so $[\bs G', \bs o]$ is extremal. So, it is enough to prove~\eqref{eq:ergodic:1}. Let
	\[
	f(G,o):=\probCond{[\bs G',\bs o]\in A}{[\bs G, \bs o]=[G, o]}.
	\]
	
	%By $A\in I$ and unimodularity of both $[\bs G, \bs o]$ and $[\bs G', \bs o]$, one obtains the following for any measurable function $g:\mathcal G_{**}\rightarrow\mathbb R^{\geq 0}$.
	For any measurable function $g:\mathcal G_{**}\rightarrow\mathbb R^{\geq 0}$,
	\begin{eqnarray*}
		\omid{f[\bs G, \bs o]\sum_{v\in V({\bs G})}g_{\bs G}(\bs o, v)} &=& \omid{f[\bs G, \bs o] g_{\bs G}^+(\bs o)}\\
		&=& \omid{\identity{A}[\bs G', \bs o] g_{\bs G}^+(\bs o)}\\
		&=& \omid{\identity{A}[\bs G', \bs o] g_{\bs G}^-(\bs o)}\\
		&=& \omid{f[\bs G, \bs o] g_{\bs G}^-(\bs o)}\\
		&=& \omid{\sum_{v\in V(\bs G)}f[\bs G, v]g_{\bs G}(\bs o, v)},
	\end{eqnarray*}
	were in the second and forth equations, conditioning on $[\bs G, \bs o]$ is used, in the third one unimodularity of $[\bs G', \bs o']$ and $A\in I$ are used and in the last equation, unimodularity of $[\bs G, \bs o]$ is used.
	Therefore
	$
	\omid{\sum_{v\in V({\bs G})}(f[\bs G, \bs o]-f[\bs G, v])g_{\bs G}(\bs o, v)}=0.
	$
	By substituting $g_{\bs G}(\bs o, v)$  with the positive and negative parts of $f[\bs G, \bs o]-f[\bs G, v]$ separately, one obtains 
	\begin{equation}
	\label{eq:ergodic:2}
	\forall v\in V(\bs G), f[\bs G, v]=f[\bs G, \bs o], \quad a.s.
	\end{equation}
	In other words, $f$ does not depend on the root a.s.  
	%
	%	\begin{eqnarray*}
	%		\omid{h[\bs G, \bs o]f[\bs G, \bs o]} &=& \omid{h[\bs G, \bs o]\identity{A}[\bs G', \bs o]}\\
	%		&=& \omid{h[\bs G, \bs o] \sum_{v\in V({\bs G})} \identity{A}[\bs G', v] t_{\bs G}(v,\bs o)}\\
	%		&=& \omid{\identity{A}[\bs G', \bs o]\sum_{v\in V(\bs G)} h[\bs G, v]t_{\bs G}(\bs o,v)}\\
	%		&=& \omid{f[\bs G, \bs o]\sum_{v\in V(\bs G)} h[\bs G, v]t_{\bs G}(\bs o,v)}\\
	%		&=& \omid{h[\bs G, \bs o]\sum_{v\in V(\bs G)} f[\bs G, v]t_{\bs G}(v,\bs o)}.
	%	\end{eqnarray*}
	%	In the first and forth equations, we have used the fact that the terms multiplied in $\identity{A}[\bs G', \bs o]$ depend on $[\bs G, \bs o]$ only. In the second equation, it is used that $A$ doesn't depend on the root. In the third and last equations, unimodularity is used. Since $h$ is arbitrary, this implies that
	%	
	%	\begin{equation}
	%	\label{eq:ergodic:2}
	%	f(\bs o) = \sum_{v\in V(\bs G)}f(v)t_{\bs G}(v,\bs o), \quad a.s.
	%	\end{equation}
	
	%	Fix a measurable function $T:\mathcal G_{**}\rightarrow \mathbb R^{\geq 0}$ such that almost surely, for all $v\in V(\bs G)$, one has $T_{\bs G}^+(v)=T_{\bs G}^-(v)=1$ and $T_{\bs G}(v,v)<1$. Such a function can be obtained by composing an arbitrary root-change with its \textit{dual} given by Lemma~\ref{lemma:rootChange}; e.g.
	Consider the root-change of changing the root to a uniformly at random neighbor of the root and compose it with its \textit{dual} given by Lemma~\ref{lemma:rootChange}. This can be explicitly written by
	\[
	T_G(u,v):=\sum_{w\in V(G)} \identity{\{w\sim u, w\sim v\}}\frac 1{d(u)d(v)} \left(\sum_{z\sim w}\frac 1{d(z)}\right)^{-1}.
	\]
	It is straightforward that when $G$ is not a single vertex, for all vertices $v\in V(G)$, one has $T_{G}^+(v)=T_{G}^-(v)=1$. Moreover, if $u,v$ have a common neighbor, then $T_{G}(u,v)>0$. Let $k\in \mathbb N$ and $t$ be the $k$-fold composition of $T$ with itself.  One can see $T$ as the law of (the first step of) a random walk on the vertices that preserves the distribution of $[\bs G, \bs o]$ and $t$ as the law of the $k$'th step of the random walk.
	By infiniteness of $\bs G$, it can be seen that almost surely, when $k\rightarrow\infty$, $t_{\bs G}(\cdot,\cdot)\rightarrow 0$ point-wise. We also have $t_{\bs G}^+(\cdot)=t_{\bs G}^-(\cdot)=1$ a.s.	
	
	For an arbitrary $\epsilon>0$, there exist $n\in\mathbb N$ and an event $A_n$ that depends only on the ball with radius $n$ centered at the root such that $\myprob{[\bs G', \bs o]\in A\Delta A_n}<\epsilon$. 
	One has
	\begin{eqnarray*}
		\omid{f[\bs G, \bs o]} &=& \omid{\identity{A}[\bs G', \bs o]}\\
		&= & \omid{\identity{A}[\bs G', \bs o] \sum_v \identity{A}[\bs G', v]t_{\bs G}(\bs o,v)}\\
		&\leq & \omid{\identity{A_n}[\bs G', \bs o] \sum_v \identity{A}[\bs G', v]t_{\bs G}(\bs o,v)}+\epsilon \\
		&= & \omid{\identity{A}[\bs G', \bs o] \sum_v \identity{A_n}[\bs G', v]t_{\bs G}(v,\bs o)}+\epsilon \\
		&\leq & \omid{\identity{A_n}[\bs G', \bs o] \sum_v \identity{A_n}[\bs G', v]t_{\bs G}(v,\bs o)}+2\epsilon \\
		&\leq & \omid{\identity{A_n}[\bs G', \bs o] \sum_{v\not\in N_{2n}(\bs o)} \identity{A_n}[\bs G', v]t_{\bs G}(v,\bs o)}+3\epsilon, \\
	\end{eqnarray*}
	where to ensure the last inequality holds, by dominated convergence, $k$ can be chosen large enough (depending on $\epsilon$, $n$, $A$ and $A_n$) in the definition of $t$. Now, note that conditioned on $[\bs G, \bs o]=[G,o]$, for $v\not\in N_{2n}(o)$, the balls $N_n(o)$ and $N_n(v)$ are disjoint and their marks are independent. Therefore, conditioned on $[\bs G, \bs o]$, the terms $\identity{A_n}[\bs G', \bs o]$ and $\sum_{v\not\in N_{2n}(\bs o)} \identity{A_n}[\bs G', v]t_{\bs G}(v,\bs o)$ are independent. 
	
	On the other hand, by defining
	$	
	f_n(G,o):=\omidCond{\identity{A_n}[\bs G',\bs o]}{[\bs G, \bs o]=[G, o]},
	$
	one has for any measurable function $h:\mathcal G_*\rightarrow\mathbb R^{\geq 0}$,
	\begin{eqnarray*}
		&& \omid{h[\bs G, \bs o] \sum_{v\not\in N_{2n}(\bs o)} \identity{A_n}[\bs G', v]t_{\bs G}(v,\bs o)} \\
		&=& \omid{\identity{A_n}[\bs G', \bs o] \sum_{v\not\in N_{2n}(\bs o)} h[\bs G, v]t_{\bs G}(\bs o,v)} \\
		&=& \omid{f_n[\bs G, \bs o] \sum_{v\not\in N_{2n}(\bs o)} h[\bs G, v]t_{\bs G}(\bs o,v)}\\
		&=& \omid{h[\bs G, \bs o] \sum_{v\not\in N_{2n}(\bs o)} f_n[\bs G, v]t_{\bs G}(v,\bs o)}.
	\end{eqnarray*}
	This implies that
	\[
	\omidCond{\sum_{v\not\in N_{2n}(\bs o)} \identity{A_n}[\bs G', v]t_{\bs G}(v,\bs o)}{[\bs G, \bs o]} =  \sum_{v\not\in N_{2n}(\bs o)} f_n[\bs G, v]t_{\bs G}(v,\bs o), \quad a.s.
	\]
	Therefore, by the above inequalities and the mentioned independence, one gets by conditioning on $[\bs G, \bs o]$ that
	\begin{eqnarray*}
		\omid{f[\bs G, \bs o]} &\leq& \omid{f_n[\bs G, \bs o] \sum_{v\not\in N_{2n}(\bs o)} f_n[\bs G, v]t_{\bs G}(v,\bs o)} + 3\epsilon\\
		&\leq & \omid{f_n[\bs G, \bs o] \sum_{v\in V(\bs G)} f_n[\bs G, v]t_{\bs G}(v,\bs o)} + 3\epsilon.
	\end{eqnarray*}
	For any measurable function $h$ on $\mathcal G_*$ such that $0\leq h\leq 1$, the fact $\myprob{[\bs G', \bs o]\in A\Delta A_n}<\epsilon$ easily implies that $\norm{\omid{(f[\bs G, \bs o]-f_n[\bs G, \bs o])h[\bs G, \bs o]}}<\epsilon$. Using this and unimodularity two times, the above inequality implies
	\begin{eqnarray*}
		\omid{f[\bs G, \bs o]} &\leq & \omid{f[\bs G, \bs o] \sum_{v\in V(\bs G)} f_n[\bs G, v]t_{\bs G}(v,\bs o)} + 4\epsilon\\
		&=& \omid{f_n[\bs G, \bs o] \sum_{v\in V(\bs G)} f[\bs G, v]t_{\bs G}(\bs o,v)} + 4\epsilon\\
		&\leq & \omid{f[\bs G, \bs o] \sum_{v\in V(\bs G)} f[\bs G, v]t_{\bs G}(\bs o,v)} + 5\epsilon\\
		%&= & \omid{f[\bs G, \bs o] \sum_{v\in V(\bs G)} f[\bs G, v]t_{\bs G}(v,\bs o)} + 5\epsilon\\
		&=&\omid{f[\bs G, \bs o]^2}+5\epsilon,
	\end{eqnarray*}
	where in the last equation, \eqref{eq:ergodic:2} is used. Since $0\leq f\leq 1$ and $\epsilon$ is arbitrary, this implies that $f[\bs G, \bs o]\in \{0,1\}$ a.s. So, \eqref{eq:ergodic:1} is proved and the proof is complete.
\end{proof}

\section{Bibliography of Analogous Results for Point Processes}
\label{sec:analogy}
In this section, we discuss a similarity between unimodular networks and stationary point processes (and random measures). Then, some of the concepts and results for random networks in this paper %sections~\ref{sec:shiftCoupling} and~\ref{sec:unimodularCase}
will be related to existing ones for point processes in the literature.

\subsection{General Analogies}
Let us recall Palm distributions and the mass transport principle for stationary point processes briefly. %$\mathbb P_{\Phi}$ for a stationary point process $\Phi$ in $\mathbb R^d$ briefly.
A stationary point processes is, roughly speaking, a random configuration $\Phi$ of points in $\mathbb R^d$ such that its distribution is invariant under the translations of $\mathbb R^d$. 
%\ali{\sout{In other words, anywhere in the space is statistically similar to anywhere else.}} 
The Palm distribution $P_{\Phi}$ of $\Phi$ is defined by
$
\mathbb P_{\Phi}[A] = \frac 1{\omid{\card{\Phi\cap B}}}  \omid{\sum_{x\in \Phi\cap B} \identity{A}(\theta_x(\Phi))},
$
%For an arbitrary measurable set $B$, 
where $B$ is an arbitrary measurable set and $\theta_x(\Phi)$ is just $\Phi$ translated by the vector $-x$. In words, to obtain the Palm distribution, one should \textit{bias the probability measure by $\card{\Phi\cap B}$ and then move the origin to a uniformly at random point in $\Phi\cap B$}. Notice the similarity of this sentence with the examples in the introduction and Subsection~\ref{subsec:extension}. %in the beginning of the introduction. 
Another equivalent definition of the Palm distribution can obtained by clarifying the idea of conditioning $\Phi$ to have a point at the origin.
The \textit{mass transport principle} for stationary point processes is 
\[
\omidPalm{\Phi}{\sum_{x\in \Phi}g(\Phi,0,x)} = \omidPalm{\Phi}{\sum_{x\in \Phi}g(\Phi,x,0)}
\]
for any (measurable) function $g$ that is translation-invariant (see~\cite{HoPe03} or~\cite{HoPe06}). A result of Mecke~\cite{Me67} gives an extension of this property for stationary random measures (see~\cite{ThLa09}). Notice the similarity of the above mass transport principle with the one in Definition~\ref{def:unimodular}. This implies that, roughly speaking, any graph that is constructed from $\Phi$ in a translation-invariant manner is unimodular \cite{processes}. Mecke's formula~\cite{Me67} may also look related to~\eqref{eq:unimodular}, but all points of the space are taken into account:
\[
\omid{\sum_{x\in \Phi}h(0,x)} = \lambda \omidPalm{\Phi}{\int_{\mathbb R^d}h(x,0)dx}
\]
for all measurable functions $h(x,y)=h(\Phi,x,y)$ that are invariant under the translations. Here, $\lambda$ is the \textit{intensity} of $\Phi$.

A difference of the two concepts is that in point processes, there is a group action (that of translations) for \textit{moving the origin to another point}, but for rooted networks, there is no natural group for \textit{changing the root}. %\ali{\sout{Currently, the author is not aware of a common world containing both stationary point processes and unimodular networks.}}
However, like the above similarities, one can transfer some concepts and results for point processes to analogous ones for random networks.
%However, the above similarity  enables one to obtain some definitions and results in the context of unimodular networks that are analogous to the ones for stationary point processes.
We will discuss in the next subsection that some of the results in sections~\ref{sec:shiftCoupling}, \ref{sec:unimodularCase} and~\ref{sec:construction} have analogous results for point processes in the literature. See also~\cite{vertexShift} for other results and examples of this analogy including Mecke's point-stationarity theorem and Neveu's exchange formula.%\mar{cite?}

Finally, as mentioned in~\cite{processes}, the invariant sigma-field $I$ is analogous to the sigma-field of invariant events under translations and the notion of extremal unimodular networks is analogous to ergodic point processes.

\subsection{Analogies Regarding the Present Paper}
Let $[\bs G, \bs o]$ be a unimodular network. %, seen analogous to a point process.
A covariant subset $S$ (Definition~\ref{def:covariantSubset}) can be considered analogous to a subprocess of (or thinning) a stationary point process. Then, conditioning the probability measure on $\bs o\in S_{\bs G}$ in Example~\ref{ex:subnetwork} is analogous to the Palm distribution of a subprocess (similarly, biasing by $\identity{\{\bs o\in S_{\bs G} \}} / \probCond{\bs o\in S_{\bs G}}{I}$ in Example~\ref{ex:subnetwork} is analogous to the \textit{modified Palm distribution}~\cite{ThLa09} of a stationary point process). Even fancier, one can think of $[\bs G, \bs o]$ and $S_{\bs G}$ as discrete objects analogous to the space and a point process respectively.
With this analogy, Proposition~\ref{prop:conditioning} is analogous to a result of~\cite{Th96} which states that the Palm distribution of a stationary point process $\Phi$ can be obtained from $\Phi$ by \textit{moving the origin to a random point of $\Phi$} if and only if the \textit{sample intensity} of $\Phi$, defined by $\omidCond{\card{\Phi\cap B}}{I}$ for an arbitrary set $B$ with unit volume, is essentially constant (here, $I$ is used for the sigma-field of invariant events under translations). 
This notion of sample intensity is also analogous to that of Proposition~\ref{prop:conditioning} and Example~\ref{ex:subnetwork} which is equal to $\omidCond{\card{S_{\bs G}\cap \{\bs o\}}}{I}$.
%Analogously, in Proposition~\ref{prop:conditioning} and Example~\ref{ex:subnetwork}, the \textit{sample intensity} of $S$ in $[\bs G, \bs o]$ is defined by $\probCond{\bs o\in S_{\bs G}}{I}$.

The special case of the result of~\cite{Th96} (mentioned above) for a Poisson point process $\Phi$ has been of special interest. The result implies that one can move the origin to a random point of $\Phi$ such that (the distribution of) the resulting point process is the same as $\Phi$ except that a point is added at the origin (see Slivnyak's theorem in~\cite{ScWe}).
%The special case of~\cite{Th96} for a Poisson point process $\Phi$ has been of special interest since the Palm distribution of $\Phi$ is obtained by just adding a point at the origin to $\Phi$ (Slivnyak's theorem See~\cite{ScWe}). Then, the theorem implies that one can move the origin to a random point of $\Phi$ such that (the distribution of) the resulting point process is the same as $\Phi$ except that a point is added at the origin. 
Such a change of origin is introduced by Thorisson~\cite{Th96} and is called an \textit{extra head scheme} in~\cite{HoPe05}. The same holds for the Bernoulli point process in $\mathbb Z^d$ and 
is analogous to Proposition~\ref{prop:extraHead}.
%can be defined by tossing a biased coin at each point of $\mathbb Z^d$ and looking at the places of heads. This result is analogous to Proposition~\ref{prop:extraHead}.

%\mar{ater: in the intro, say shift-coupling only for PP's?} 
More general to a stationary point process and its Palm version,~\cite{Th96} also studies when two (not necessarily stationary) point processes $\Phi$ and $\Psi$ can be obtained from each other by changing the origin (and covers even more general cases). This gives a coupling of $\Phi$ and $\Psi$ that is called a \textit{shift-coupling} in the literature. It is proved in~\cite{Th96} that a shift-coupling exists if and only if the distributions of $\Phi$ and $\Psi$ agree on the invariant sigma-field. 
%This result has been the motivation of many works, \discuss{\mar{change this} including all of the results mentioned in this section}. 
Theorem~\ref{thm:shiftCoupling} in the present paper is its analogous result in the context of random networks.

To obtain a shift-coupling of a stationary point process $\Phi$ and its Palm version, one can use a translation-invariant \textit{balancing transport kernel} between (a multiple of) the Lebesgue measure and the counting measure on $\Phi$ (see~\cite{HoPe05} and~\cite{ThLa09}). A transport kernel is, roughly speaking, a function $T(x,y)$ depending on $\Phi$ that shows how much of the mass at each point $x$ of the space goes to each point $y$ in $\Phi$. It is \textit{balancing} when the sum of the outgoing mass is 1 and the integral of the incoming mass is constant for all points. The existence of such a transport kernel is proved in~\cite{HoPe05} and~\cite{ThLa09} using the result of~\cite{Th96} {and also by an explicit construction}. Analogously, Theorem~\ref{thm:balancing} proves the existence of balancing transport kernels in the context of unimodular networks with similar conditions. Invariant transport kernels are analogous to measurable functions on $\mathcal G_{**}$ as mentioned in Remark~\ref{rem:transport}. See also~\cite{ThLa09} for similar results for stationary random measures.

Assume $([\bs G, \bs o], S)$ is a proper extension of a unimodular network as in Section~\ref{sec:extAndTrans}. According to the analogy of subnetworks and Palm distributions mentioned earlier, $[\bs G, \bs o]$ is analogous to a Palm distribution. In fact, \eqref{eq:mtpOnS} is analogous to the mass transport principle for \textit{point-stationary} point processes (see~\cite{ThLa09}); which are more general than Palm distributions (when dealing with probability measures, as assumed here). Unimodularization of $[\bs G, \bs o]$ is analogous to reconstructing the stationary version of the point process from the Palm version~\cite{Me67}. The unimodularization $\mathcal P_T$ in Theorem~\ref{thm:extensionT} is analogous to \textit{the inversion formula} of~\cite{Me67} for the reconstruction. If the sample intensity of a point process is essentially constant, one can do the reconstruction by a shift-coupling as mentioned above. This is analogous to Theorem~\ref{thm:uniExtGen}.

%\mar{\ali{later: mention~\cite{stable} solves completely the stationary version, and others are in special cases.}}
The existence results of~\cite{ThLa09} and~\cite{Th96} are abstract results and cannot be used to construct a balancing transport kernel given realizations of the two point processes. However, several constructions %of invariant balancing transport kernels 
are provided in the literature motivated mainly by~\cite{HoPe06} for point processes, which is motivated by the construction in~\cite{Li} and the stable marriage algorithm. This work is generalized in~\cite{stable} to cover the general case of stationary random measures.
Analogously, Theorem~\ref{thm:balancing} doesn't provide a construction of the balancing transport kernel in the context of unimodular networks. Theorem~\ref{thm:stable} is a construction which is analogous to the one in~\cite{stable}. %, which has been presented for stationary random measures as a generalization of~\cite{HoPe06}.

\section*{Acknowledgements}
This work was motivated by the author's joint works with Francois Baccelli and Mir-Omid Haji-Mirsadeghi. I thank them also for their useful comments. I thank Lewis Bowen as well for his comments on Borel equivalence relations, especially for the discussion on Proposition~\ref{prop:nonstandard}.

%===============================================================

\bibliographystyle{amsplain}

\end{document}